\documentclass[11pt,letterpaper]{amsart}

\usepackage[T1]{fontenc}
\usepackage[utf8]{inputenc}
\usepackage{lmodern}
\usepackage[margin=1.1in]{geometry}
\usepackage{microtype}
\usepackage{xcolor}

\usepackage{amsmath,amssymb,amsfonts,amsthm}
\usepackage[numeric]{amsrefs}
\usepackage{thmtools}
\usepackage{comment}

\newif\ifdraft
\draftfalse    

\ifdraft
  \usepackage[notref,notcite]{showkeys}
  \PassOptionsToPackage{implicit=false}{hyperref}
\fi

\PassOptionsToPackage{hypertexnames=false}{hyperref}
\usepackage{enumerate}

\def\le {\leqslant}
\def\ge {\geqslant}
\DeclareMathOperator{\supp}{supp}

\newtheorem{theorem}{Theorem}[section]

\newtheorem{lemma}[theorem]{Lemma}

\theoremstyle{definition}
\newtheorem{definition}[theorem]{Definition}
\newtheorem{remark}[theorem]{Remark}
\newtheorem{example}[theorem]{Example}

\numberwithin{equation}{section}

\usepackage[hidelinks]{hyperref}

\def\cA{\mathcal{A}}

\def\RR{\mathbb{R}}

\def\EE{\mathbb{E}}
\def\PP{{\mathbb P}}

\def \<{\langle}
\def\>{\rangle}

\def\bT{\mathbf{T}}

\def\NN{\mathbb{N}}
\def\CC{\mathbb{C}}

\def\va{\varepsilon}

\def\da{{\delta}}

\def\ta{\theta}
\def\al{{\alpha}}
\def\be{\beta}

\def\ga{{\gamma}}

\def\vi{\varphi}

\def\Br{\Bigr}
\def\Bl{\Bigl}
\def\f{\frac}

\newcommand{\wt}{\widetilde}

\def\ind{\mathbf{1}}

\def\sub{\substack}

\def\ld{\lambda}

\def\Ld{\Lambda}
\def\Og{\Omega}

\def\ra{{\rangle}}
\def\la{\langle}

\newcommand{\bp}{ \begin{proof} }
  
  \newcommand{\ep}{  \end{proof} }

\newcommand{\beq}{\begin{equation}}
\newcommand{\enq}{\end{equation}}

\newcommand{\eq}[1]{\begin{align*}#1\end{align*}}

\def \supp{\operatorname{supp}}
\def \spn{\operatorname{span}}

\DeclareMathOperator{\Er}{Er}
\DeclareMathOperator{\diam}{diam}

\title[]{Empirical Approximation of $L_p$ Norms} 
\thanks{The first and third  authors  were partially supported by 
NSERC of Canada Discovery Grant RGPIN 2026-05667.
The second author was supported by AEI grant
RYC2023-043616-I and by the Spanish State Research Agency through the Severo Ochoa and Mar\'ia de Maeztu Program for Centers and Units of Excellence in R\&D (CEX2020-001084-M). The second author also thanks CERCA Programme (Generalitat de Catalunya) for its institutional support.}

\author{Feng Dai}
\address{F.~Dai, Department of Mathematical and Statistical Sciences, University of Alberta, Edmonton, Alberta T6G 2G1, Canada}
\email{fdai@ualberta.ca}

\author{Egor Kosov}
\address{E.~Kosov, Centre de Recerca Matem\`atica, Campus de Bellaterra, Edifici~C 08193
  Bellaterra (Barcelona), Spain.}
\email{kosoved09@gmail.com}

\author{Noel Murasko}
\address{N.~Murasko, Department of Mathematical and Statistical Sciences, University of Alberta, Edmonton, Alberta T6G 2G1, Canada}
\email{murasko@ualberta.ca}

\begin{document}

\subjclass[2020]{Primary 41A65; Secondary 60E15, 46B20, 42C15, 60B20.}
\keywords{Marcinkiewicz discretization, random sampling, empirical $L_p$ moments, restricted isometry property, generic chaining, uniformly convex sets}

\begin{abstract}

We study empirical $L_p$ moments of a random vector $\pmb\varphi$ based on its i.i.d.\ copies $\pmb\varphi^1,\ldots,\pmb\varphi^m$, that is, $\frac1m\sum_{j=1}^m |\langle \pmb\varphi^j,y\rangle|^p$. Our main result is a new estimate for the expected uniform deviation
\[
\mathbb{E}\sup_{y\in D}\biggl|
\frac1m\sum_{j=1}^m |\langle \pmb\varphi^j,y\rangle|^p
-\mathbb{E}|\langle \pmb\varphi,y\rangle|^p
\biggr|
\]
over an arbitrary index set $D$.
The proof is based on a new bound for Talagrand's $\gamma$-functional, sharper than the standard Dudley-type entropy estimate.
We then apply this estimate to the following two problems.

First, for $p>2$, we study Marcinkiewicz-type discretization of $L_p$ norms on an $N$-dimensional subspace $X_N\subset B(\Omega)$ of bounded functions on a probability space $(\Omega,\mu)$. We obtain bounds in terms of the norm of the embedding
$
(X_N,\|\cdot\|_{L_p(\mu)})\hookrightarrow B(\Omega).
$
In particular, we prove that when this norm is of order $N^{1/p}$ and
\[
m \ge C(p)\, N\log N\,(\log\log N)^{p-1},
\]
then $m$ random samples suffice to approximate the $L_p(\mu)$ norm uniformly on $X_N$ by the sampled discrete $L_p$ norm. This substantially improves the previously known bound in this setting
$
m \ge C(p)\, N(\log N)^{\min\{p,3\}},
$
and is optimal up to the factor $(\log\log N)^{p-1}$ in the random-sampling setting.

Second, for $1\le p<2$, we obtain an $L_p$ analogue of the restricted isometry property via random sampling for bounded orthogonal systems and, more generally, for $N$-element systems $\mathcal D_N$ satisfying a Riesz-type condition. We prove that when
\[
m \ge C(p)\, s\log N\,(\log s)^2\,\log\log s,
\]
then $m$ random samples suffice to guarantee an $L_p$ restricted isometry-type property uniformly over the class of all $s$-sparse functions generated by $\mathcal D_N$.

\end{abstract}
\maketitle

\section{Introduction}\label{sec1}

\subsection{The Marcinkiewicz discretization}
Let {$(\Omega, \mathcal{F}, \mu)$} be a probability space, where $\mathcal{F}$ is a $\sigma$-algebra of subsets of $\Omega$ and $\mu$ is a probability measure on $\mathcal{F}$. For $1\le p< \infty$, denote by $L_p(\mu) \equiv L_p(\Omega,\mu)$ the usual Lebesgue space of measurable functions $f\colon\Omega \to \mathbb{C}$ equipped with the norm
\[
\|f\|_{L_p(\mu)} = \left( \int_{\Omega} |f|^p\, d\mu \right)^{1/p}.
\]
Let $B(\Omega)$ denote the space of all bounded   functions $f\colon\Omega \to \mathbb{C}$ equipped with the uniform norm
\[
\|f\|_{\infty}:=\sup_{x\in\Omega} |f(x)|.
\]
Let $X_N \subset B(\Omega)$ be an $N$-dimensional linear space of bounded functions. We assume that any function $f \in X_N$ that vanishes $\mu$-a.e.\!\! on $\Omega$ is identically zero (equivalently, $\|\cdot\|_{L_p(\mu)}$ defines a norm on $X_N$).

Our goal is to discretize the $L_p$ norm on $X_N$ by Marcinkiewicz-type inequalities. Specifically, given $1\le p<\infty$ and $\varepsilon\in(0,1)$, we seek sampling points $\xi^1,\dots,\xi^m\in\Omega$ (with $m\ge N$)
such that
\begin{equation}\label{mz}
  (1-\varepsilon)\|f\|_{L_p(\mu)}^p
  \le
  \frac{1}{m}\sum_{j=1}^m |f(\xi^j)|^p
  \le
  (1+\varepsilon)\|f\|_{L_p(\mu)}^p,
  \quad \forall\, f\in X_N.
\end{equation}
The main problem is to determine the asymptotically optimal sample size $m$ for which \eqref{mz} holds on an arbitrary $N$-dimensional subspace $X_N$ under natural structural assumptions.

\subsection{Background and motivation}
In the classical setting, $X_N$ typically consists of trigonometric or algebraic polynomials of bounded degree on a compact domain $\Omega \subset \mathbb{R}^d$, equipped with a normalized (possibly weighted) Lebesgue measure. In this case, inequalities of the form~\eqref{mz} are known as Marcinkiewicz--Zygmund inequalities.
For these polynomial spaces, such inequalities with an asymptotically optimal sample size $m \asymp N$ have been established on various regular compact domains, such as tori, intervals, spheres, cubes, and $C^2$-smooth domains in $\RR^d$; see, for instance, \cites{Lu1, Lu2, Ko-Lu, Ma-To, KKLT, DKP, DP, Zy}. These inequalities play a fundamental role in the convergence of Fourier series, Lagrange interpolation, and polynomial approximation.
However, classical techniques rely heavily on geometric properties of $\Omega$ and often suffer from the curse of dimensionality as the ambient dimension $d$ increases. These limitations motivate a more general high-dimensional theory in which $X_N$ is an arbitrary $N$-dimensional subspace of $B(\Omega)$.
In this general setting, Marcinkiewicz discretization inequalities~\eqref{mz}
are governed by the norms of the embeddings $(X_N,\|\cdot\|_{L_p}) \hookrightarrow B(\Omega)$,
which are conveniently expressed via Nikolskii-type inequalities.

\subsection{Nikolskii-type inequality}
Given $p \in [1,\infty)$, an $N$-dimensional subspace $X_N \subset B(\Omega)$
is said to satisfy the $(p,\infty)$--Nikolskii-type inequality with constant $H\ge1$ if
\[
\|f\|_{\infty} \le H^{1/p}\,\|f\|_{L_p(\mu)}, \quad \forall\, f \in X_N.
\]
In this case, we write $X_N \in \textnormal{NI}_{p,\infty}(H)$.

It is readily seen that for any $ p \ge 2$,
\begin{equation*}
X_N \in \textnormal{NI}_{p,\infty}(H)
\implies
X_N \in \textnormal{NI}_{2,\infty}(H).
\end{equation*}
In particular, for $p \ge 2$, the condition $X_N \in \textnormal{NI}_{p,\infty}(H)$ implies that $H \ge N$.
Therefore, in this case we may write $H = KN$ for some constant $K \ge 1$.

\subsection{Prior work}
In recent years, the Marcinkiewicz discretization problem~\eqref{mz}
has been extensively studied under the $(2,\infty)$--Nikolskii-type inequality,
namely, in the case when
\begin{equation}\label{1-1}
X_N \in \textnormal{NI}_{2,\infty}(K N),
\end{equation}
where $K \ge 1$ is fixed
(see, for instance,~\cites{ACDM, DKT, Ko22, Kos, LimT, DPSTT2, DPSTT1, Tem18, KPUU}).
The techniques employed in these works are largely drawn from  the theory of embeddings of finite-dimensional subspaces (see \cites{BLM, JS, Sche, Sche3, Tal, Tal90, Tal12}).
Although a constant $K$ in~\eqref{1-1} always exists by finite-dimensional norm equivalence, its size is crucial.

We now recall the main known results under the assumption $X_N\in \textnormal{NI}_{2,\infty}(KN)$.

\smallskip
\noindent
\textbf{(i) The case $p=2$.}
By~\cites{LimT,Ko22}, building on the groundbreaking work~\cite{MSS}, the bound
\[
m \ge C\varepsilon^{-2}KN
\]
ensures an $L_2$-discretization inequality of the form~\eqref{mz}, up to an additional factor $K$ on the right-hand side (see also~\cite{Sche}).

\smallskip
\noindent
\textbf{(ii) The case $1\le p<2$.}
It was shown in~\cite{DKT}, using a refined form of Talagrand's estimate~\cite{Tal}*{Theorem~16.8.2}, that~\eqref{mz} holds whenever
\begin{equation*}
m \ge C(p,\varepsilon)
\begin{cases}
(KN)\log(KN), & p=1,\\[1mm]
(KN)\log(KN)(\log\log(KN))^2, & 1<p<2.
\end{cases}
\end{equation*}

\smallskip
\noindent
\textbf{(iii) The case $p>2$.}
This regime is different. If the $L_p$ and $L_2$ norms are equivalent on a subspace $X_N$, that is,
\[
\|f\|_{L_p(\mu)} \le M \|f\|_{L_2(\mu)}, \quad \forall\, f \in X_N,
\]
then the Marcinkiewicz discretization inequality necessarily requires at least
\begin{equation}\label{eq-lower}
m \ge c(p) M^{-p} N^{p/2}
\end{equation}
sampling points (see~\cite{BDGJN} and~\cite{KKLT}*{{\bf D.20}}). Under~\eqref{1-1}, it was shown in~\cite{Kos} that
\begin{equation}\label{p-under-2-nik}
  m \ge C(p,\varepsilon)(KN)^{p/2}\log(KN)
\end{equation}
suffices for~\eqref{mz}.
To obtain nearly linear bounds in $N$ for $p>2$,
\cite{Kos} replaced \eqref{1-1} by the stronger condition
\[
X_N \in \textnormal{NI}_{p,\infty}(KN),
\]
under which it was shown that~\eqref{mz} holds whenever
\begin{equation}\label{eq-p-1}
  m \ge C(p,\varepsilon)KN(\log(KN))^p.
\end{equation}
This was later improved in~\cite{DT1} to
\begin{equation}\label{eq-p-2}
m \ge C(p,\varepsilon)KN(\log(KN))^{\min\{3,p\}}.
\end{equation}

It is known (see Example~\ref{limitations} below) that for certain subspaces
$X_N \in \textnormal{NI}_{p,\infty}(N)$ one cannot achieve~\eqref{mz}
with fewer than $N\log N$ \emph{random} samples.
{ Thus, for $p>2$, the main unresolved issue in the random-sampling setting was whether one could bridge the gap between the lower bound of order $N\log N$ and the previously known upper bounds of order $N(\log N)^{\min\{p,3\}}$.
}

\subsection{The main result on Marcinkiewicz discretization}
In this paper, we establish the following probabilistic version of the
$L_p$ Marcinkiewicz discretization inequality for $p>2$.

{
Given a probability space $(\Omega,\mathcal F,\mu)$, we consider independent
random points $\xi^1,\ldots,\xi^m$ in $\Omega$, not necessarily identically
distributed, such that the average of their distributions is $\mu$, that is,
\begin{equation}\label{probab-cond}
\mu(E)=\frac{1}{m}\sum_{j=1}^m \mathbb P\{\xi^j\in E\},
\quad E\in\mathcal F.
\end{equation}
In particular, this condition is satisfied whenever the points $\xi^1,\dots,\xi^m$ are i.i.d.\!\! with distribution~$\mu$.
}

\begin{theorem}\label{Th-Nik}
Let $(\Omega,\mathcal F,\mu)$ be a probability space, and let
$\xi^1,\ldots,\xi^m$ be independent random points in $\Omega$ satisfying
\eqref{probab-cond}. Let $2<p<\infty$, and let $X_N\subset B(\Omega)$ be a
finite-dimensional linear subspace of bounded functions satisfying
$X_N\in \textnormal{NI}_{p,\infty}(H)$ with $H\ge 16$. Then there exists a
constant $c=c(p)>0$ such that, whenever $m\in\mathbb{N}$, $\lambda\ge e$ and
$\varepsilon\in(0,1/2]$ satisfy
\[
m \ge
\bigl(\lambda \varepsilon^{-1}\log \varepsilon^{-1}\bigr)^p
H\log H\,(\log\log H)^{p-1},
\]
the Marcinkiewicz discretization inequality~\eqref{mz} holds with probability
at least
$
1-2e^{-c\lambda}$.

\end{theorem}

In particular, for $X_N\in \mathrm{NI}_{p,\infty}(KN)$, this improves the previously known bounds~\eqref{eq-p-1} and~\eqref{eq-p-2}, which involved polynomial losses in $\log(KN)$, to
\[
m \asymp KN\log(KN)(\log\log(KN))^{p-1},
\]
which is optimal up to the factor $(\log\log(KN))^{p-1}$ in the \emph{random-sampling} setting.

We also note that, up to the factor $(\log\log(KN))^{p-1}$,
Theorem~\ref{Th-Nik} recovers the bound~\eqref{p-under-2-nik} for spaces
$X_N \in \textnormal{NI}_{2,\infty}(K N)$. Indeed, in this case
\[
\|f\|_\infty \le (K N)^{1/2}\|f\|_{L_2(\mu)}
\le \bigl((K N)^{p/2}\bigr)^{1/p}\|f\|_{L_p(\mu)},
\]
so that $X_N \in \textnormal{NI}_{p,\infty}\bigl((K N)^{p/2}\bigr)$.

\subsection{Restricted isometry property}

Another sampling discretization problem considered in this paper is related to the construction of matrices with the {\it restricted isometry property} (RIP), a central notion in compressed sensing introduced by Cand\`es and Tao~\cite{CT05} (see also~\cite{CRT06}). Recall that an $m\times N$ matrix $A$ is said to satisfy the RIP of order $s$ with constant $\varepsilon\in(0,1)$ if
\begin{equation}\label{RIP}
  (1-\varepsilon)\|{\bf a}\|_{\ell_2^N}^2
  \le \|A{\bf a}\|_{\ell_2^m}^2
  \le (1+\varepsilon)\|{\bf a}\|_{\ell_2^N}^2,
  \quad \forall\, {\bf a}\in\CC^N \text{ with } \|{\bf a}\|_0\le s,
\end{equation}
where $\|{\bf a}\|_0:=|\supp {\bf a}|$ denotes the number of nonzero coordinates of the vector ${\bf a}$.

A common approach (see \cite{FR13}*{Chapter~12}) to construct such matrices is to sample from a bounded orthonormal system
$\mathcal D_N=\{\varphi_1,\dots,\varphi_N\}\subset B(\Omega)$ in $L_2(\mu)$ satisfying
\begin{equation}\label{cond-bound}
  \|\varphi_j\|_\infty \le K_0, \quad j=1,\dots,N.
\end{equation}
More precisely, the central problem is to determine how many random samples
$\xi^1,\ldots,\xi^m$ are required so that the sampled matrix
\begin{equation}\label{matrix}
A=\frac1{\sqrt m}
\begin{pmatrix}
\varphi_1(\xi^1) & \cdots & \varphi_N(\xi^1)\\
\varphi_1(\xi^2) & \cdots & \varphi_N(\xi^2)\\
\vdots & \ddots & \vdots\\
\varphi_1(\xi^m) & \cdots & \varphi_N(\xi^m)
\end{pmatrix}
\end{equation}
satisfies~\eqref{RIP} with high probability.

Rudelson and Vershynin \cite{RV08} proved that this holds whenever
\[
m \ge C(K_0,\varepsilon)\, s \log N \, (\log s)^2 \, \log(s \log N),
\]
and Haviv and Regev~\cite{HR17} {(see also \cite{BDJR})} later improved this to
\begin{equation}\label{eq-HR}
m \ge C(K_0,\varepsilon)\, s\log N\,(\log s)^2.
\end{equation}
Some intermediate results can be found in~\cites{Bou14, ChGV}.

For an orthonormal dictionary $\mathcal D_N$, the RIP for a matrix of the form~\eqref{matrix} is equivalent to sampling discretization of the $L_2$ norm on the class of $s$-sparse functions. Indeed, if ${\bf a}\in\CC^N$ and
\[
f=\sum_{j=1}^N a_j\varphi_j,
\]
then $A{\bf a}$ is the vector of sampled values $(f(\xi^1),\dots,f(\xi^m))$ up to the factor $m^{-1/2}$, and~\eqref{RIP} becomes
\[
(1-\varepsilon)\|f\|_{L_2(\mu)}^2
\le \frac1m\sum_{j=1}^m |f(\xi^j)|^2
\le (1+\varepsilon)\|f\|_{L_2(\mu)}^2,
\quad \forall\, f\in \Sigma_s(\mathcal D_N),
\]
where
\[
\Sigma_s(\mathcal D_N)
:=
\bigcup_{\substack{J\subset \{1,2,\ldots, N\} \\ |J|=s}} \mathrm{span}\{\varphi_j\colon j\in J\}.
\]
This identifies RIP as a special case of universal sampling discretization.

\subsection{Universal sampling discretization}

The RIP formulation above suggests a natural extension in which the $L_2$ norm is replaced by the $L_p$ norm. This leads to a universal version of the Marcinkiewicz discretization problem: instead of discretizing the norm on a fixed finite-dimensional space, one seeks a single random sample that simultaneously discretizes the $L_p$ norm over a family of finite-dimensional spaces generated by sparse subsets of a fixed finite dictionary. Problems of this type have recently been studied in~\cites{DTM1, DTM2, KosTem}.

To move beyond the orthogonal setting, we replace orthogonality by a one-sided $s$-sparse Riesz-type condition.
Namely, we say that a dictionary $\mathcal D_N=\{\varphi_1,\dots,\varphi_N\}\subset B(\Omega)$
is a uniformly bounded one-sided $s$-sparse Riesz system with constant $K\ge1$ if
\[
\|\varphi_j\|_\infty \le 1, \quad j=1,\dots,N,
\]
and
\beq \label{riesz}
\sum_{j=1}^N |a_j|^2
\le
K\Bigl\|\sum_{j=1}^N a_j\varphi_j\Bigr\|_{L_2(\mu)}^2,
\quad
\forall\, {\bf a}=(a_1,\dots,a_N)\in\CC^N
\ \text{with }\|{\bf a}\|_0\le s.
\enq
In particular, if $\{\varphi_1, \ldots, \varphi_N\}$ is an orthonormal system satisfying~\eqref{cond-bound}, then the rescaled system
$
\mathcal{D}_N:=\{\tfrac{1}{K_0}\varphi_1, \ldots, \tfrac{1}{K_0}\varphi_N\}
$
forms a uniformly bounded $s$-sparse Riesz system with constant $K=K_0^2$.
We note that this condition can be viewed as a variant of the Nikolskii-type inequality $\textnormal{NI}_{2,\infty}(\sqrt{Ks})$, since it implies that
\[
\|f\|_\infty\le \sqrt{Ks}\|f\|_{L_2(\mu)}
\quad \forall\, f\in \Sigma_s(\mathcal{D}_N).
\]

The universal sampling discretization problem asks for conditions on $m$ under which, for i.i.d.\!\! random points $\xi^1,\dots,\xi^m\in\Omega$, one has
\begin{equation}\label{2-3-1}
  (1-\varepsilon)\|f\|_{L_p(\mu)}^p
  \le \frac1m\sum_{j=1}^m |f(\xi^j)|^p
  \le (1+\varepsilon)\|f\|_{L_p(\mu)}^p,
  \quad
  \forall\, f\in \Sigma_s(\mathcal D_N),
\end{equation}
with high probability.
The non-Euclidean $L_p$ case with $p\ne 2$ is, as usual, significantly more challenging than the $L_2$ case. The best previously known  result in this direction was obtained in~\cite{DTM2}, where the Rudelson--Vershynin theorem was extended from $p=2$ to the full range $p\in[1,2]$. More precisely, it was shown there that when
\begin{equation}\label{eq-univ-m}
m \ge C(p) K s \log N\, (\log (Ks))^2 \, \log(Ks \log N),
\end{equation}
then $m$ random samples suffice to guarantee~\eqref{2-3-1} for a uniformly bounded $s$-sparse Riesz system with constant $K\ge1$. Our next result improves this bound.

\begin{theorem}\label{Th-univ}
Let $(\Omega,\mathcal F,\mu)$ be a probability space, and let
$\xi^1,\ldots,\xi^m$ be independent random points in $\Omega$ satisfying
\eqref{probab-cond}.
Then, for any  $p \in [1,2]$,
there exists a constant $c = c(p)>0$
such that for
all integers $N\ge4$ and $s \in [4,N]$, every uniformly bounded $4s$-sparse Riesz system
$\mathcal D_N$ in $L_2(\Og, \mu)$
with constant $K \ge 16$, and all  parameters  $\lambda \ge e$ and
$\varepsilon \in (0,\tfrac12]$, the universal Marcinkiewicz discretization inequality~\eqref{2-3-1} holds  with probability
at least
$
1 - 2 e^{-c\lambda (\log\lambda)^{1/2}}
$
provided that the sample size $m$ satisfies
\[
m \ge \bigl(\lambda\,\varepsilon^{-1}\bigr)^2
(\log \varepsilon^{-1})\,
K s\,\log N\,\log s\,\log(K s)\,\log\log(Ks).
\]
\end{theorem}

Compared with~\eqref{eq-univ-m}, Theorem~\ref{Th-univ} reduces the logarithmic loss from
$
(\log(Ks))^2\,\log(Ks\log N)
$
to
$
\log s\,\log(Ks)\,\log\log(Ks),
$
thus
improving the overall dependence both on $s$ and $N$. Moreover, our bound differs from the best known $L_2$ sample complexity~\eqref{eq-HR} only by an additional factor of $\log\log s$.

\subsection{Notation}
Throughout the paper we set $\NN_0=\mathbb{N}\cup\{0\}$.
For $p\in(1,\infty)$, we denote by
$
p'=\frac{p}{p-1}
$
the conjugate exponent of $p$, and as usual we set $p'=\infty$ for $p=1$. We also adopt the convention that whenever an expression of the form
$
\frac{pp_0}{p-p_0}
$
appears with $p=p_0$, it is understood to be equal to $\infty$.

For a set of functions $F\subset B(\Omega)$ and $r\in(0,\infty)$, we define
\[
T_r(F):=\{|f|^r\colon f\in F\}.
\]
We identify a vector $x\in\CC^m$ with the function on $\{1,\dots,m\}$ given by $j\mapsto x(j)$. Accordingly, for $0<r<\infty$ and $G\subset\CC^m$, we write
\[
|x|^r := \bigl(|x(1)|^r,\ldots,|x(m)|^r\bigr)
\qquad\text{and}\qquad
T_r(G):=\{|x|^r\colon x\in G\}.
\]
For $1\le q\le \infty$, we denote by $\|\cdot\|_q$ the standard $\ell_q^m$ norm on $\CC^m$.
For a sample $\pmb\xi=(\xi^1,\dots,\xi^m)$, we define the associated normalized  discrete $\ell_q$ norm by
\begin{equation}\label{1-9-2025}
  \|f\|_{L_q(\pmb\xi)}
  :=
  \begin{cases}
    \Bigl(\f1m\sum\limits_{j=1}^m |f(\xi^j)|^q\Bigr)^{1/q}, & 1\le q<\infty,\\[1em]
    \max\limits_{1\le j\le m}|f(\xi^j)|, & q=\infty.
  \end{cases}
\end{equation}
For a set $F\subset B(\Omega)$, we define
\beq\label{1-16} F(\pmb\xi) := \bigl\{(f(\xi^1),\ldots,f(\xi^m))\colon f\in
F\bigr\}\subset\CC^m. \enq

Throughout the paper, $C,c>0$ denote positive absolute constants whose values may change from line to line. When several such constants appear in the same argument, we distinguish them by subscripts, writing $C_1,C_2$, etc. These subscripts are only labels and do not indicate dependence on parameters. Such dependence is indicated explicitly, for instance by writing $C(p)$ when dependence on the parameter $p$ is involved.

\subsection{Organization of the paper}

The remainder of the paper is organized as follows.
In Section~\ref{sec2}, we introduce the probabilistic framework, discuss the connection with empirical moments of random vectors, and explain the limitations of the sole Nikolskii assumption.
In Section~\ref{sec3}, we formulate our main abstract discretization theorem, Theorem~\ref{Th-main}, the key chaining estimate, Theorem~\ref{thmGammaBound}, and its application to $\theta$-convex indexing sets, Theorem~\ref{Th-Conv}.

Section~\ref{sec-prelim} collects preliminary material on entropy numbers and Talagrand's generic chaining theory.
In Section~\ref{sec:4}, we present the main steps in the proof of Theorem~\ref{Th-main}, reducing it to the technical chaining estimate of Theorem~\ref{thmGammaBound}.
Section~\ref{sec5} is devoted to the derivation of Theorems~\ref{Th-Conv} and~\ref{Th-Nik}, while Section~\ref{universal} contains the proof of Theorem~\ref{Th-univ}.

In Section~\ref{sec6}, we review the aspects of van Handel's approach to chaining needed in the sequel and prove a finite-dimensional version of the contraction principle.
Finally, Section~\ref{sec:proof-thmGammaBound} contains the proof of Theorem~\ref{thmGammaBound}, which is the main technical part of the paper.

\section{Probabilistic setting and connection with empirical moments}\label{sec2}

\subsection{Probabilistic framework}

Let $\pmb\xi:=(\xi^1,\ldots,\xi^m)$ be a sequence of independent random
points in $\Omega$ satisfying condition~\eqref{probab-cond}, that is,
\[
\mu(E)=\frac{1}{m}\sum_{j=1}^m \PP[\xi^j\in E],
\quad E\in\mathcal F.
\]
Given a subclass $F\subset B(\Omega)$, we define the discretization error by
\[
\Er_p(F,\pmb\xi)
:=
\sup_{f\in F}
\biggl|
\frac1m\sum_{j=1}^m |f(\xi^j)|^p
-
\|f\|_{L_p(\mu)}^p
\biggr|.
\]
Thus, for
\[
X_N^p:=\{f\in X_N\colon\|f\|_{L_p(\mu)}\le 1\},
\]
the Marcinkiewicz discretization inequality~\eqref{mz} holds if and only if
\[
\Er_p(X_N^p,\pmb\xi)\le \varepsilon.
\]
Our goal is to estimate the minimal sample size $m$ for which $\Er_p(X_N^p,\pmb\xi)$ is small, either in expectation or with high probability. More precisely, given $\varepsilon,\delta\in(0,1)$, one seeks conditions ensuring that
\[
\EE\,\Er_p(X_N^p,\pmb\xi)\le \varepsilon
\quad\text{or}\quad
\PP\bigl[\Er_p(X_N^p,\pmb\xi)>\varepsilon\bigr]\le \delta.
\]

\subsection{Random sampling and empirical moments}
\label{subsect-moments}

The probabilistic Marcinkiewicz discretization problem is closely related to the problem of approximating moments of a random vector by empirical moments. Let $\pmb\varphi\in\mathbb{R}^N$ be an $N$-dimensional random vector, and let $\pmb\varphi^1,\dots,\pmb\varphi^m$ be independent copies of $\pmb\varphi$. For a set $B\subset\RR^N$, consider
\[
V_p(B):=
\sup_{y\in B}
\biggl|
\frac1m\sum_{j=1}^m |\langle \pmb\varphi^j,y\rangle|^p
-
\EE |\langle \pmb\varphi,y\rangle|^p
\biggr|.
\]
The problem of controlling $V_p(B)$ with high probability has been extensively studied in high-dimensional probability. The most classical case corresponds to $p=2$ and $B$ being the Euclidean unit ball, which is equivalent to approximating the covariance matrix by the empirical covariance matrix in operator norm (see, for instance,~\cites{MP12, MP14, Tikh18, SV13, Versh12}). More general values of $p$ and more general index sets $B$ have been studied to a much lesser extent, with only a few works addressing this level of generality, typically under some additional assumptions (see, for instance,~\cites{Versh11, AV26, GM, GR, ALPT-J}).

The connection with probabilistic Marcinkiewicz discretization becomes immediate after choosing a basis of $X_N$. For simplicity, we assume that $X_N$ consists  solely  of real-valued functions.  Let $\varphi_1,\dots,\varphi_N$ be  a basis of $X_N$, and define the random vector
\[
\pmb\varphi := (\varphi_1,\dots,\varphi_N)\colon (\Omega,\mathcal F,\mu)\to\mathbb{R}^N.
\]
If $\xi^1,\dots,\xi^m$ are i.i.d.\ random {points in $\Og$}  with distribution $\mu$, then
\[
\pmb\varphi^j:=\bigl(\varphi_1(\xi^j),\dots,\varphi_N(\xi^j)\bigr),
\quad j=1,\dots,m,
\]
are independent copies of $\pmb\varphi$. Moreover,
\[
\Er_p(F,\pmb\xi)=V_p(B),
\]
where
\[
B:=\Bigl\{y=(y_1,\dots,y_N)\in\RR^N\colon \sum_{k=1}^N y_k\varphi_k\in F\Bigr\}.
\]
Thus, probabilistic Marcinkiewicz discretization is equivalent to the empirical
moment problem for random vectors associated with $X_N$, and results obtained
in either language can be readily translated into the other.

However, the standard assumptions used in the empirical moment literature are not well suited to the regime $p>2$ considered in this paper. Typically, one assumes an almost sure bound
\begin{equation}\label{diam}
  \|\pmb\varphi\|_{\ell_2^N}\le \sqrt{KN}
\end{equation}
together with a moment-comparison estimate of the form
\begin{equation}\label{moments}
  \bigl(\EE|\langle \pmb\varphi,y\rangle|^q\bigr)^{1/q}
  \le
  M\bigl(\EE|\langle \pmb\varphi,y\rangle|^2\bigr)^{1/2},
  \qquad \forall\, y\in\RR^N,
\end{equation}
for some $q\ge p$. When the basis $\{\varphi_k\}_{k=1}^N$ is orthonormal in $L_2(\mu)$, condition~\eqref{diam} corresponds to the usual $(2,\infty)$-Nikolskii inequality $X_N\in \mathrm{NI}_{2,\infty}(KN)$. By contrast, condition~\eqref{moments} translates into the norm-equivalence assumption
\[
\|f\|_{L_p(\mu)}\le\|f\|_{L_{q}(\mu)}\le M\|f\|_{L_2(\mu)},\quad \forall\, f\in X_N,
\]
which is too restrictive for the discretization problem when $p>2$: it fails in many classical spaces, such as spaces of trigonometric or algebraic polynomials, and, more importantly, it is precisely this condition that forces the required number of sampling points to grow polynomially in the dimension (see~\eqref{eq-lower}).

For this reason, we do not impose any moment-comparison assumptions. Instead, we work under the more flexible $(p,\infty)$-Nikolskii condition $X_N\in\mathrm{NI}_{p,\infty}(H)$ and seek high-probability discretization bounds in terms of the Nikolskii constant $H$. It is worth mentioning that this framework still contains the classical $(2,\infty)$ setting, since
\[
\mathrm{NI}_{2,\infty}(KN)\implies \mathrm{NI}_{p,\infty}\bigl((KN)^{p/2}\bigr).
\]

\subsection{Limitations of the \texorpdfstring{$(p,\infty)$}{(p,infinity)}-Nikolskii assumption}

The generality of the $(p,\infty)$-Nikolskii assumption comes with an intrinsic limitation: in this setting, one cannot in general obtain probabilistic Marcinkiewicz discretization with fewer than $\mathcal{O}(N \log N)$ random samples. This already appears in the model case of the uniform measure on an $N$-point set.

\begin{example}\label{limitations}
Let $\Omega := \{1,\ldots,N\}$, let $X_N := \{f\colon \Omega \to \mathbb{R}\} = \mathbb{R}^N$, and let $\mu(\{j\}) = \frac{1}{N}$ for all $j \in \Omega$. Then
\[
\|f\|_{L_p(\mu)}
= \Bigl(\frac{1}{N}\sum_{j=1}^N |f(j)|^p\Bigr)^{1/p}
\quad\text{and}\quad
\|f\|_\infty \le N^{1/p}\|f\|_{L_p(\mu)},
\]
so that $X_N \in \textnormal{NI}_{p,\infty}(N)$ for every $p \in [1,\infty)$.

Let $\xi^1,\xi^2,\ldots$ be independent random variables distributed according to $\mu$, and define $T$ to be the smallest positive integer $m$ such that
$\{1,\ldots,N\} \subset \{\xi^1,\ldots,\xi^m\}$.
Clearly, for every fixed $m\in\NN$,
\[
\Er_p(X_N^p,\pmb{\xi})
:= \sup_{\|f\|_{L_p(\mu)} \le 1}
\biggl| \frac{1}{m}\sum_{j=1}^m |f(\xi^j)|^p - \|f\|_{L_p(\mu)}^p \biggr|
\ge \ind_{\{T>m\}}.
\]
Hence, for every $\varepsilon \in (0,1)$,
\[
\PP\bigl[\Er_p(X_N^p,\pmb{\xi}) \le \varepsilon\bigr]
\le \PP(T \le m).
\]
This implies that for every $\delta \in (0,1)$ and every $m \le \delta N\log N$,
\[
\PP\bigl[\Er_p(X_N^p,\pmb{\xi}) \le \varepsilon\bigr]
\le \PP\bigl(|T-\EE( T)| \ge \EE (T) - m\bigr)
\le \tfrac{\operatorname{Var}(T)}{(\mathbb{E}(T) - \delta N\log N)^2}
\xrightarrow[N\to\infty]{} 0,
\]
where we used the classical coupon collector estimates
\[
\EE (T) = N\log N + O(N)
\quad\text{and}\quad
\operatorname{Var}(T) = O(N^2).
\]
Thus, under the sole assumption $X_N\in \textnormal{NI}_{p,\infty}(N)$, one cannot in general expect probabilistic Marcinkiewicz discretization with fewer than $\mathcal{O}(N \log N)$ random samples.
\end{example}

\section{The main abstract results}
\label{sec3}

\subsection{A general estimate in terms of the chaining functional}\label{subsect-main}

Our main abstract discretization theorem bounds the expected uniform discretization error $\Er_p(F,\pmb{\xi})$ of the $L_p$ norm  over an arbitrary function set $F$ containing $0$, in terms of Talagrand's chaining functional (see Definition~\ref{def-Tal})
\[
\gamma_{p,p_0}\Bigl(F,\|\cdot\|_{L_{p_1}(\pmb{\xi})}\Bigr)
:= m^{\f 1p-\f 1{p_0}}
\gamma_{p,p_0}\Bigl(F(\pmb{\xi}),\|\cdot\|_{p_1}\Bigr),
\]
where $p\ge 2$, $p_0\in(1, p]$, $p_1:=\f{pp_0}{p-p_0}$, and $F(\pmb{\xi})$ is defined in \eqref{1-16}.

To formulate the theorem, let $(\Omega,\mathcal F,\mu)$ be a probability space,
and let $\pmb{\xi}:=(\xi^1,\ldots,\xi^m)$ be a {collection} of $m$ independent
random elements {taking values in $\Omega$} and satisfying
condition~\eqref{probab-cond}.

\begin{theorem}\label{Th-main}
Let $2\le p<\infty$, and let $F$ be a class of bounded functions on $\Omega$
which contains $0$ and satisfies the following two conditions with parameters
$p_0\in(1,p]$, $H>0$, $m_0\in\mathbb N\cap[1,m]$, and $1\le q<\infty$:
\begin{equation}\label{1-11a-0}
\Bigl(\mathbb{E}\,
\Bigl(
\sup_{f\in F}\|f\|_{L_\infty( \pmb\xi)}^{p-p_0}
\cdot \sup_{g\in F}\|g\|_{L_{p_1}(\pmb{\xi})}^{p_0}
\Bigr)^q
\Bigr)^{1/q}
\le H^{p_0/p},
\end{equation}
 \begin{equation}\label{1-11a}
\mathbb{E}\,
\Bigl[\sup_{f\in F}\|f\|_{L_\infty(\pmb\xi)}^{p-p_0}\cdot \Bigl(\gamma_{p,p_0}\Bigl(F,\|\cdot\|_{L_{p_1}(\pmb{\xi})}\Bigr)\Bigr)^{p_0 }\Bigr]
\le (H m_0)^{p_0/p},
\end{equation}
where  $p_1:=\f {pp_0}{p-p_0}$.
Then there exists a constant $C=C(p)>0$ such that, whenever
\begin{equation}\label{1-11b}
m \ge \max\Bigl\{16m_0,\ e\bigl(e\log\log \tfrac{m}{m_0}\bigr)^{p/p_0'} Hm_0\Bigr\},
\end{equation}
one has
\[
\Bigl(\mathbb{E}\, \bigl|\Er_p(F,\pmb{\xi})\bigr|^q \Bigr)^{1/q}
\le \frac{C q}{1+\log q}\, \Theta,
\]
where
\[
\Theta :=
\Bl(\f {H m_0} m\Br)^{\f1p} \Bl(\log\log \frac{m}{m_0} + \log \f m{H m_0}\Br)^{1-\f1{p_0}}\Bigl(1+\sup\limits_{f \in F} \|f\|_{L_p(\mu)}^p\Bigr).
\]
Moreover, if~\eqref{1-11a-0} holds for all $ 1\le q<\infty$, then one  has
\begin{equation}\label{eq-probab-est-main}
\mathbb{P}\bigl[\Er_p(F,\pmb{\xi}) > C\,\Theta\, t\bigr]
\le \exp\Bigl(-\frac{t\log t}{e}\Bigr),
\quad \forall\, t \ge e.
\end{equation}
\end{theorem}

The results announced in Section~\ref{sec1} will be derived from this theorem.

\smallskip

Our approach relies on Talagrand's generic chaining~\cite{Tal}, combining ideas from~\cite{Tal12} with van Handel's more recent  approach ~\cites{VH, VH1}, which utilizes interpolation functionals and the contraction principle.
This framework  is particularly convenient in our setting, since it allows one to compare a given chaining functional with more tractable auxiliary ones.

Let $F \subset B(\Omega)$ be a bounded set. A standard symmetrization argument, combined with Talagrand's majorizing measure theorem, implies (see~\eqref{5-8b}) that
\[
\mathbb{E}\Er_p(F,\pmb{\xi})
\le \frac{C(p)}{m}\,
\mathbb{E}\,\Bl[
\gamma_{p,1}\bigl(T_p(F(\pmb\xi)),\|\cdot\|_{p'}\bigr)\Br],
\]
thereby reducing the problem to the estimation of
$
\gamma_{p,1}\bigl(T_p(G),\|\cdot\|_{p'}\bigr)
$
for the vector set
\[
G:=F(\pmb\xi)\subset \CC^m.
\]
The main difficulty is that the geometry of $T_p(G)$ is usually much more complicated than that of the original set $G$, making it hard to obtain   direct,  sharp  estimates of
$
\gamma_{p,1}\bigl(T_p(G),\|\cdot\|_{p'}\bigr)
$
in concrete situations. The key step in our argument is therefore to bound this functional in terms of more tractable auxiliary chaining quantities on $G$. This is formalized in the following abstract comparison theorem.

\begin{theorem}\label{thmGammaBound}
Let $2 \le p < \infty$, $p_0\in(1, p]$, and $p_1:=\f {pp_0}{p-p_0}$.
Then there exists a constant $C = C(p)>0$
such that for every nonempty set $G \subset \mathbb{C}^m$
and  any choice of parameters  $b > 0$, $\alpha_0 \in (0,1)$, $n_0 \in \NN$ and $m_0\in \NN \cap [1, m)$, one has
\begin{align}\label{5-7-0}
\ga_{p,1}(T_p(G),\|\cdot\|_{p'})
&\le C \Biggl[
\Bigl( b n_0 + b^2\cdot (\alpha_0)^{1/p}
+\frac{\log \frac{m}{m_0}}{(2^{n_0}\alpha_0)^{1/p}} \Bigr) \sup_{x \in G} \|x\|_{p}^p
\\
&+
\sup_{x\in G}\|x\|_\infty^{p-p_0}\cdot
\frac{m_0^{p_0/p}\bigl[\diam(G, \|\cdot\|_{p_1}) \bigr]^{p_0}+ \bigl[\ga_{p,p_0}(G, \|\cdot\|_{p_1})\bigr]^{p_0}}{b^{p_0-1}}
\Biggr].\nonumber
\end{align}
\end{theorem}

\begin{remark} In the case of  $p_0:=p$, Theorem~\ref{thmGammaBound}   improves upon the naive estimate
\[
\ga_{p,1}(T_p(G),\|\cdot\|_{p'})
\le 2 p \Bl(\sup_{x \in G} \|x\|_{p}^{p/p'}\Br)\ga_{p,1}(G,\|\cdot\|_{\infty}),
\]
since   $\ga_{p,p}(G,\|\cdot\|_{\infty})$ typically grows at least logarithmically more slowly in $m$ than $\ga_{p,1}(G,\|\cdot\|_{\infty})$ (see, for example, Lemma~\ref{lem-2-5} when $G\subset \RR^m$ is $p$-convex). By utilizing these superior estimates  for  $\ga_{p,p}(G,\|\cdot\|_{\infty})$, and optimizing the  parameters $b, \al_0$ and $n_0$ in   Theorem \ref{thmGammaBound}, one typically obtains notably   sharper  bounds for  $\ga_{p,1}(T_p(G),\|\cdot\|_{p'})$.
\end{remark}

\begin{remark}\label{rm-complex}
It suffices to prove Theorem~\ref{thmGammaBound} in the real case, that is, for sets $G \subset \mathbb{R}^m$. To see this,  let $G \subset \mathbb{C}^m$ and consider the set
\[
T_1(G):=\{(|z_1|,\ldots,|z_m|)\colon (z_1,\ldots,z_m)\in G\}\subset\mathbb{R}^m.
\]
Then, for every $q\ge 1$,
\[
\diam(T_1(G), \|\cdot\|_{q})
\le \diam(G, \|\cdot\|_{q}),
\]
and (see~\eqref{eq-gamma-equiv}), for any $\alpha>0$, $\beta\ge 1$,
\[
\gamma_{\alpha, \beta}(T_1(G),\|\cdot\|_{q})
\le C(\alpha)\,\gamma_{\alpha, \beta}(G,\|\cdot\|_{q}),
\]
while
\[
\ga_{p,1}(T_p(G),\|\cdot\|_{p'})
=
\ga_{p,1}\bigl(T_p\bigl(T_1(G)\bigr),\|\cdot\|_{p'}\bigr).
\]
Applying Theorem~\ref{thmGammaBound} to the real set $T_1(G)$ yields~\eqref{5-7-0} for $G\subset \mathbb{C}^m$.
\end{remark}

\subsection{Uniformly convex classes}
Theorem~\ref{Th-main} is  especially useful when the indexing sets $F$ are  $\theta$-convex for some $\ta\ge 2$ (see Definition \ref{def-4-8}).
It is well-known that Euclidean balls are $2$-convex. More generally, for any $1<p<\infty$,  and any finite-dimensional subspace $X_N$ of $L_p(\mu)$,
the unit ball
\[
X_N^p:=\{f\in X_N\colon \|f\|_{L_p(\mu)}\le1\}
\]
is $\max\{p, 2\}$-convex with a constant $\eta=\eta(p)$ depending only on $p$.

Gu\'edon and Rudelson~\cite{GR} established the following estimate for every
$p\ge \theta$ and every $\theta$-convex subset $F$ of a Euclidean ball
$D$, with $\theta$-convexity constant $\eta$:
\begin{equation}\label{GR-est}
\mathbb{E}\Er_p(F,\pmb{\xi})
\le C(p,\eta)\Bigl(A + A^{1/2} \sup_{f \in F} \|f\|_{L_p(\mu)}^{p/2}\Bigr),
\end{equation}
where
\[
A = \frac{(\log m)^{2(1-\frac{1}{\theta})}}{m}
\mathbb{E}\Bigl(
\sup_{f \in D}\|f\|_{L_\infty(\pmb{\xi})}^2
\sup_{g \in F}\|g\|_{L_\infty(\pmb{\xi})}^{p-2}
\Bigr).
\]
The specific case  where $F=D$ is  the Euclidean ball and $p=2$  was treated earlier by Rudelson~\cite{Rud99}. The estimate~\eqref{GR-est} was further improved in~\cite{Kos}*{Corollary~4.4}, where it was shown to hold with
\[
A =
\frac{1}{m}
\mathbb{E}\Bigl(
\sup_{f \in D}\|f\|_{L_\infty(\pmb{\xi})}^2
\sup_{g \in F}\|g\|_{L_\infty(\pmb{\xi})}^{p-2}
\Bigr)
+ \frac{\log m}{m}
\mathbb{E}\Bl(
\sup_{g \in F}\|g\|_{L_\infty(\pmb{\xi})}^{p}\Br).
\]
Taking $\theta=p\ge 2$,  $F = X_N^p$, and  $D = X_N^2$,  one obtains the Marcinkiewicz discretization
inequality~\eqref{mz} with
$m = C(p, \varepsilon) H^{p/2} \log H$ sampling points
for each subspace $X_N \in \textnormal{NI}_{2,\infty}(H)$.

To cover the case $X_N \in \textnormal{NI}_{p,\infty}(H)$ with $p>2$, the following counterpart of the estimate~\eqref{GR-est} was proved in~\cite{Kos}*{Corollary~4.7} for any $\ta$-convex subset $F$ of $B(\Og)$:
\[
\mathbb{E}\Er_p(F,\pmb{\xi})
\le C(p, \eta)\Bigl(A + A^{\frac{1}{\theta}} \sup_{f \in F} \|f\|_{L_p(\mu)}^{p(1-\frac{1}{\theta})}\Bigr),
\]
where
\[
A = \frac{(\log m)^\theta}{m}
\mathbb{E}
\sup_{g \in F}\|g\|_{L_\infty(\pmb{\xi})}^p.
\]
Taking here $F = X_N^p$ and $\theta=p$, one obtains the Marcinkiewicz discretization inequality~\eqref{mz} with
$
m = C(p, \varepsilon)\, H (\log H)^p
$
sampling points for every subspace $X_N \in \textnormal{NI}_{p,\infty}(H)$ (see~\cite{DT1}*{Theorem~2.2} for further refinements).

In this paper,  we deduce from Theorem~\ref{Th-main} the following general result on the discretization error over $\theta$-convex sets $F \subset B(\Omega;\mathbb{R})$, where $(\Omega,\mathcal{F}, \mu)$ denotes an arbitrary probability space, and
$B(\Omega;\mathbb{R})$ denotes the space of all bounded real-valued functions on $\Og$.

\begin{theorem}\label{Th-Conv}
Let $\pmb{\xi}:=(\xi^1,\ldots,\xi^m)$ be a collection of $m$ independent
random elements taking values in $\Omega$ and satisfying
condition~\eqref{probab-cond}.
Let $p  \ge 2$ and let  $F\subset B(\Omega;\mathbb{R})$ be a $p$-convex subset
with constant $\eta>0$ such that
\[
\sup_{f\in F}\|f\|_{L_p(\mu)}\le1.
\]
Assume that
for some constants $H\ge 16$ and $1\le q<\infty$,
\begin{equation}\label{3-7a}
\Bigl(\mathbb{E}\sup_{f \in F}\|f\|_{L_\infty(\pmb{\xi})}^{p q}\Bigr)^{1/q}
\le H.
\end{equation}
Then there exists a constant $C=C(p,\eta)>0$ such that, whenever
\[
m \ge (8 e p)^{2p}\, H\, (\log H)\, (\log\log H)^{p-1},
\]
one has
\[
 \Bigl(\mathbb{E}\,|\Er_p(F,\pmb{\xi})|^q\Bigr)^{1/q}
\le
\frac{C q}{1+\log q}
\Bigl(\frac{H \log m}{m}\Bigr)^{1/p}
\Bigl(\log \frac{m}{H}\Bigr)^{1-\frac{1}{p}}.
\]
Furthermore, if condition~\eqref{3-7a} is satisfied for every $q\ge 1$, and the sample size $m$ satisfies
\[
m \ge (\lambda\,\varepsilon^{-1}\log \varepsilon^{-1} )^p
H\, (\log H)\, (\log\log H)^{p-1}
\]
for some constants $\lambda\ge e$ and $\varepsilon \in (0,\tfrac12]$, then
\[
\mathbb{P}\bigl[\Er_p(F,\pmb{\xi})>\varepsilon\bigr]
\le 2  e^{-c\ld}
\]
for some constant $c=c(p,\eta)>0$.
\end{theorem}

Taking $F = X_N^p$, one deduces Theorem~\ref{Th-Nik} for a subspace
$X_N \in \textnormal{NI}_{p,\infty}(H)$ in the real case.

\begin{remark}
Using the equivalence between the two formulations described in
Subsection~\ref{subsect-moments}, the preceding theorem gives a counterpart
for uniform approximation of moments of a random vector $\pmb\varphi$ by
empirical moments based on its i.i.d.\!\! copies
$\pmb\varphi^1,\ldots,\pmb\varphi^m$, uniformly over $p$-convex sets
$B\subset\mathbb R^N$.
Under this equivalence, assumption~\eqref{3-7a} translates into the following
control of the dual norm:
\[
\bigl(\mathbb{E}\max_{1\le j\le m}
\|\pmb\varphi^j\|_{B^\circ}^{pq}\bigr)^{1/q}
\le H,
\]
where \[\|\mathbf{x}\|_{B^\circ}:=\sup_{\mathbf y \in B}|\la \mathbf x, \mathbf y\ra|,\  \ \  \mathbf x\in \RR^N.\]
A quantity of this type already appears in the result of
Gu\'edon and Rudelson~\cite{GR}.
\end{remark}

\section{Preliminaries: entropy, entropy numbers and chaining functionals}\label{sec-prelim}

In this section, we recall several basic facts concerning entropy numbers and Talagrand's generic chaining theory~\cite{Tal}. We begin by fixing some notation.

Let $(\mathbf{T},\varrho)$ be a metric space. For $x\in \mathbf{T}$ and $r>0$, we write
\[
B_\varrho(x,r):=\{y\in \mathbf{T}\colon \varrho(x,y)\le r\}
\]
for the closed ball of radius $r$ centered at $x$. For a set $A\subset \mathbf{T}$, we define
\[
\diam(A,\varrho):=\sup_{s,t\in A}\varrho(s,t)\  \ \text{and}\  \ \varrho(s, A):=\inf_{t\in A}\varrho(s,t)\  \ \text{for $s\in \mathbf{T}$}.
\]
If the metric $\varrho$ is induced by a norm $\|\cdot\|$, we write $\diam(A,\|\cdot\|)$ instead of $\diam(A,\varrho)$.

For a finite set $\Lambda$, we denote its cardinality by $|\Lambda|$. For $x\ge0$, we write $\lceil x\rceil$ for the smallest integer greater than or equal to $x$. Finally, we set
\[
N_0:=1,
\quad
N_n:=2^{2^n}, \quad n\in\NN.
\]

\subsection{Entropy and entropy numbers}

\begin{definition}
For $\varepsilon>0$, the covering number $\mathcal N_\varepsilon(A,\varrho)$ of a bounded set $A\subset \mathbf T$ is defined as the least positive integer $n$ for which there exist points $x_1,\dots,x_n\in A$ such that
\[
A\subset \bigcup_{j=1}^n B_\varrho(x_j,\varepsilon).
\]
The $\varepsilon$-entropy of $A$ with respect to $\varrho$ is defined by
\[
\mathcal H_\varepsilon(A,\varrho):=\log_2 \mathcal N_\varepsilon(A,\varrho).
\]
\end{definition}

\begin{definition}
Let $A\subset \mathbf T$ be bounded. For $n\in \mathbb N_0$, the $n$-th entropy number of $A$ with respect to $\varrho$ is defined by
\[
e_n(A,\varrho)
:=
\inf_{x_1,\dots,x_{N_n}\in A}
\sup_{x\in A}\min_{1\le j\le N_n}\varrho(x,x_j),
\]
or equivalently,
\[
e_n(A,\varrho)
=
\inf\Bigl\{\varepsilon>0\colon
A\subset \bigcup_{j=1}^{N_n} B_\varrho(x_j,\varepsilon)
\text{ for some } x_1,\dots,x_{N_n}\in A
\Bigr\}.
\]
\end{definition}
Clearly,
\[
e_n(A,\varrho)
=
\inf\bigl\{\varepsilon>0\colon\mathcal H_\varepsilon(A,\varrho)\le 2^n\bigr\}.
\]
One may also define entropy numbers without requiring the centers of the covering balls to belong to $A$, namely
\[
\tilde e_n(A,\varrho):=
\inf\Bigl\{\varepsilon>0\colon
A\subset \bigcup_{j=1}^{N_n} B_\varrho(x_j,\varepsilon)
\text{ for some } x_1,\dots,x_{N_n}\in \mathbf T
\Bigr\}.
\]
Then
\begin{equation}\label{3-1-0}
  \tilde e_n(A,\varrho)\le e_n(A,\varrho)\le 2\,\tilde e_n(A,\varrho).
\end{equation}
If the metric $\varrho$ is induced by a norm $\|\cdot\|$, we also write $e_n(A,\|\cdot\|)$ instead of $e_n(A,\varrho)$.

We shall use the following standard properties of entropy numbers in finite-dimensional spaces, which follow from \cite{TemBook}*{Corollary~7.2.2} and \cite{TemBook}*{Estimate~(7.1.6)}.

\begin{lemma}\label{lem-1-2a}
Let $(X_N,\|\cdot\|)$ be an $N$-dimensional real normed space, and let
\[
B:=\{x\in X_N\colon\|x\|\le 1\}
\]
be its unit ball. Then
\[
e_n(B,\|\cdot\|)\le 4N_n^{-1/N}
=4\cdot 2^{-2^n/N},
\quad n\in\mathbb N_0.
\]
Moreover, for every bounded set $A\subset X_N$, the sequence
\[
\{N_n^{1/N}e_n(A,\|\cdot\|)\}_{n=0}^\infty
\]
is almost decreasing, in the sense that for all $k,n\in\mathbb N_0$ with $k\le n$,
\begin{equation}\label{1-1-0}
N_n^{1/N}e_n(A,\|\cdot\|)
\le 3\,N_k^{1/N}e_k(A,\|\cdot\|).
\end{equation}
\end{lemma}
For further properties of entropy numbers, we refer to~\cite{TemBook}*{Chapter~7}.

\subsection{The Talagrand chaining functional}

\begin{definition}\label{def-Tal}
Let $(\mathbf T,\varrho)$ be a metric space. A sequence of partitions
$\{\mathcal A_n\}_{n=0}^\infty$ of $\mathbf T$ is called \emph{increasing} if, for every
$n\ge0$ and every $I\in\mathcal A_n$, $J\in\mathcal A_{n+1}$, one has either
$J\subset I$ or $J\cap I=\emptyset$.
It is called an \emph{admissible sequence of partitions} if it is increasing and
satisfies $|\mathcal A_n|\le N_n$ for all $n\ge0$.

For $x\in\mathbf T$, we denote by $\mathcal A_n(x)$ the unique element of
$\mathcal A_n$ containing $x$.
\end{definition}

\begin{definition}
Let $\alpha>0$ and $1\le \beta<\infty$. The \emph{chaining functional} of
$(\mathbf T,\varrho)$ is defined by
\[
\gamma_{\alpha,\beta}(\mathbf T,\varrho)
:=
\biggl(
\inf_{\{\mathcal A_n\}}
\sup_{x\in\mathbf T}
\sum_{n=0}^\infty
\bigl[2^{n/\alpha}\diam(\mathcal A_n(x),\varrho)\bigr]^\beta
\biggr)^{1/\beta},
\]
where the infimum is taken over all admissible sequences of partitions
$\{\mathcal A_n\}_{n=0}^\infty$ of $\mathbf T$. If the metric $\varrho$ is induced by a norm $\|\cdot\|$, we write
$\gamma_{\alpha,\beta}(\mathbf T,\|\cdot\|)$ instead of
$\gamma_{\alpha,\beta}(\mathbf T,\varrho)$.
\end{definition}

By definition,
\begin{equation*}
  \diam(\mathbf T,\varrho)
  =
  \sup_{x\in\mathbf T}\diam(\mathcal A_0(x),\varrho)
  \le
  \gamma_{\alpha,\beta}(\mathbf T,\varrho).
\end{equation*}
Thus, without loss of generality, we may always assume that
$\diam(\mathbf T,\varrho)<\infty$.

There is an equivalent description of the chaining functional in terms of approximating sets.

\begin{definition}
  Let $\alpha>0$, $1\le \beta<\infty$, and let $(\mathbf T,\varrho)$ be a metric space. Define
  \[
  \gamma_{\alpha,\beta}^*(\mathbf T,\varrho)
  :=
  \biggl(
  \inf_{\{T_n\}}
  \sup_{x\in\mathbf T}
  \sum_{n=0}^\infty
  \bigl[2^{n/\alpha}\varrho(x,T_n)\bigr]^\beta
  \biggr)^{1/\beta},
  \]
  where the infimum is taken over all sequences of subsets $T_n\subset\mathbf T$
  satisfying $|T_n|\le N_n$.
\end{definition}

It is known (see \cite{Tal}*{Theorem~2.3.1} and \cite{VH}*{Lemma~4.2}) that:
\begin{equation}\label{eq-gamma-equiv}
  \gamma_{\alpha,\beta}^*(\mathbf T,\varrho)
  \le
  \gamma_{\alpha,\beta}(\mathbf T,\varrho)
  \le
  C(\alpha)\,\gamma_{\alpha,\beta}^*(\mathbf T,\varrho).
\end{equation}

The chaining functional is also linked to entropy numbers through the estimate
\begin{equation}\label{1-2-2025}
  \gamma_{\alpha,\beta}(\mathbf T,\varrho)
  \le
  C(\alpha)
  \biggl(
  \sum_{n=0}^\infty
  \bigl[2^{n/\alpha}e_n(\mathbf T,\varrho)\bigr]^\beta
  \biggr)^{1/\beta}.
\end{equation}

The following theorem of Talagrand, which connects the supremum of a random process with the associated chaining functional, is a basic tool in what follows.

\begin{theorem}\label{thm-1-3}
Let $\{W_x\colon x\in\mathbf T\}$ be a random process indexed by a metric space
$(\mathbf T,\varrho)$. Assume that there exists $\alpha>0$ such that
\begin{equation}\label{exp}
\PP\bigl(|W_x-W_y|\ge u\,\varrho(x,y)\bigr)\le 2e^{-u^\alpha},
\quad \forall\, u>0,\ \forall\, x,y\in\mathbf T.
\end{equation}
Then there exists a constant $c=c(\alpha)>0$ such that for every $x_0\in\mathbf T$,
\[
\PP\bigl[\sup_{x\in\mathbf T}|W_x-W_{x_0}|
\ge \gamma_{\alpha,1}(\mathbf T,\varrho)\,u\bigr]
\le 2e^{-cu^\alpha},
\quad \forall\, u>0.
\]
In particular,
\[
\EE \sup_{x\in\mathbf T}|W_x-W_{x_0}|
\le C(\alpha)\,\gamma_{\alpha,1}(\mathbf T,\varrho).
\]
\end{theorem}

\subsection{Entropy and chaining functionals of \texorpdfstring{$\theta$}{theta}-convex sets}

Uniform convexity, and in particular $\theta$-convexity, plays an important role in the study of entropy numbers and Talagrand's chaining functionals in normed spaces. In this section, we collect several known estimates for $\theta$-convex sets that will be used later to derive a convenient upper bound for the chaining functional appearing in our main arguments.

We begin by recalling the definition of a $\theta$-convex set.

\begin{definition}\label{def-4-8}
Let $X$ be a real linear space and let $\theta \ge 2$.
A centrally symmetric convex set $F \subset X$ is called
\emph{$\theta$-convex with constant $\eta>0$}
if there exists a norm $\|\cdot\|_F$ on  a linear subspace $Y$ of $X$ such that $(Y,\|\cdot\|_F)$ is a Banach space,
$F=\{x\in Y\colon  \|x\|_F\le 1\}$   and
\[
\biggl\|\frac{f+g}{2}\biggr\|_F
\le
\max\{\|f\|_F,\|g\|_F\}
-\eta\,\|f-g\|_F^\theta,
\quad \forall\, f,g\in F.
\]
In this case, we also call $(Y, \|\cdot\|_F)$ a \emph{$\theta$-convex Banach space}
with constant $\eta>0$.
\end{definition}

Classical examples of $\theta$-convex Banach spaces are the spaces $L_p$, $1<p<\infty$, and their closed linear subspaces, with $\theta=\max\{p,2\}$.
\begin{remark}
It is known (see, for instance, \cite{Pi}*{Proposition~2.4}) that the above property is equivalent to the inequality
\[
\left\|\frac{f+g}{2}\right\|_F^\theta
+
\lambda
\left\|\frac{f-g}{2}\right\|_F^\theta
\le
\frac{1}{2}\bigl(\|f\|_F^\theta+\|g\|_F^\theta\bigr),
\quad \forall\, f,g \in Y,
\]
where $\lambda>0$ depends only on $\eta$ and $\theta$.
\end{remark}

\begin{remark}
The property of $\theta$-convexity is invariant under  linear mappings. More precisely, assume that  $F$ is a  $\theta$-convex subset of a real linear space $X$   with constant $\eta>0$ and
\[
F:=\{x\in X\colon  \|x\|_F\le 1\}.
\]
If $T: X\to Y$ is a  linear mapping onto another real linear space $Y$ such that $\ker(T)$ is closed with respect to the norm  $\|\cdot\|_F$, then the image $T(F) := \{Tx\colon x\in F\}$ is a $\theta$-convex subset of $Y$ with the same constant $\eta$ and  the associated norm on $Y$ given by
\[
\|y\|_{T(F)}:=\min\Bl\{ \|x\|_F\colon  x\in X,\  T(x)=y\Br\},\quad y\in Y.
\]
This invariance can be established using the fact that every uniformly convex
Banach space is reflexive.
\end{remark}

The entropy bound~\eqref{1-2-2025} provides a general estimate for chaining functionals in terms of entropy numbers. While this estimate is sufficiently strong for many purposes, it is often not optimal and may lose logarithmic factors. A substantial improvement is possible when the underlying set is uniformly convex (see, for instance, \cites{Tal,VH,VH1}). In particular, for $\theta$-convex sets one has the following refinement of~\eqref{1-2-2025}.

\begin{lemma}[{\cite{Tal}*{Theorem~4.1.4}, \cite{VH}*{Theorem~5.8}}]\label{lem-2-5}
Let $(X,\|\cdot\|)$ be a real Banach space, and let $\theta\ge2$. If $F\subset X$ is a $\theta$-convex set with constant $\eta>0$, then for every $\alpha\ge1$,
\begin{equation*}
\gamma_{\alpha,\theta}(F,\|\cdot\|)
\le
C(\alpha,\theta)\,\eta^{-1/\theta}\sup_{n\ge0} 2^{n/\alpha}e_n(F,\|\cdot\|).
\end{equation*}
\end{lemma}

For $\theta$-convex sets one also has the following estimate for entropy numbers in the sampled $\ell_\infty$ norm (see \cite{Tal}*{Lemma~16.5.4}, \cite{Kos}*{Corollary~4.2}, \cite{Tem22}).

\begin{lemma}\label{cor-8-2}
Let $X$ be a linear space of real-valued functions on a set $\Omega$, and let $\theta\ge2$. Suppose that $F\subset X$ is a $\theta$-convex set with constant $\eta>0$, and that the point-evaluation functionals are continuous on $F$ with respect to the norm $\|\cdot\|_F$. Let $\pmb\xi=\{\xi_1,\dots,\xi_m\}\subset\Omega$ be a finite sequence of points in $\Omega$. Then
\[
e_n(F,\|\cdot\|_{L_\infty(\pmb\xi)})
\le
C(\theta,\eta)\Bigl(\sup_{f\in F}\|f\|_{L_\infty(\pmb\xi)}\Bigr)
\Bigl(\frac{\log m}{2^n}\Bigr)^{1/\theta},
\quad \forall\, n\in\mathbb N_0.
\]
\end{lemma}

\section{Proof of Theorem \ref{Th-main}}\label{sec:4}

As explained in Subsection~\ref{subsect-main}, the proof of Theorem~\ref{Th-main} relies on the estimate provided by Theorem~\ref{thmGammaBound}. Since the proof of Theorem~\ref{thmGammaBound} is rather involved, we postpone it to Section~\ref{sec:proof-thmGammaBound}. Assuming Theorem~\ref{thmGammaBound}, we now proceed to prove  Theorem~\ref{Th-main}.

Recall that $\pmb\xi=(\xi^1,\dots,\xi^m)\in\Omega^m$, where $\xi^1,\dots,\xi^m$ are independent random points satisfying
\begin{equation}\label{5-1b}
  \mu(E)=\frac1m\sum_{j=1}^m \PP[\xi^j\in E],
  \quad \text{for every measurable } E\subset\Omega.
\end{equation}
For $1\le q\le\infty$, the seminorm $\|\cdot\|_{L_q(\pmb\xi)}$ is defined in~\eqref{1-9-2025}. For $1\le p<\infty$ and $F\subset B(\Omega)$, we write
\begin{equation}\label{4-1a}
\Er_p(F,\pmb\xi):=
\sup_{f\in F}
\Biggl|
\frac1m\sum_{j=1}^m |f(\xi^j)|^p
-
\|f\|_{L_p(\mu)}^p
\Biggr|.
\end{equation}

The proof of Theorem~\ref{Th-main} is divided into four steps, which are carried out in the next four subsections.
In Step 1, we combine a standard symmetrization argument with Talagrand's majorizing measure theorem to obtain
\begin{equation}\label{4-7-eq}
\EE \Er_p(F,\pmb\xi)
\le
\frac{C(p)}{m}\,
\EE_{\pmb\xi}\,
\Bl[\gamma_{p,1}\bigl(T_p(F(\pmb\xi)),\|\cdot\|_{p'}\bigr)\Br].
\end{equation}
In Step 2, we establish Theorem~\ref{Th-main} for the case $q=1$ by applying ~\eqref{4-7-eq} and optimizing the parameters in Theorem~\ref{thmGammaBound}.
In Step 3, we extend this result to $q>1$ via Talagrand's concentration inequality (Lemma~\ref{thm-9-3}), which bounds  the $L_q(\mu)$-norm of Banach space-valued random elements by their  $L_1(\mu)$-norm. Finally, in Step 4, we deduce the probability estimate~\eqref{eq-probab-est-main}.

\subsection{Step 1: Symmetrization}

Let $\pmb\va=(\va_1,\dots,\va_m)$ be a sequence of i.i.d.\!\! random variables, each taking the values $\pm1$ with probability $1/2$, and assume that $\pmb\va$ is independent of the random points $\xi^1,\dots,\xi^m$.

In this step, we use a standard Gin\'e--Zinn symmetrization argument (see, for instance, \cite{Tal}*{Lemma~9.1.11}, \cite{GR}, and \cite{Kos}*{Lemma~3.1}) to obtain the following estimate.

\begin{lemma}\label{lem-5-2}
For every subset $F\subset B(\Omega)$, one has
\[
\EE \Er_p(F,\pmb\xi)
\le \frac{2}{m}\,
\EE \sup_{f\in F}\biggl|\sum_{j=1}^m |f(\xi^j)|^p \va_j\biggr|.
\]
\end{lemma}

For completeness, we provide the proof below.

\begin{proof}
Let $\pmb\eta=(\eta^1,\dots,\eta^m)$ be an independent copy of $\pmb\xi=(\xi^1,\dots,\xi^m)$ that is also independent of $\pmb\va=(\va_1,\dots,\va_m)$. By~\eqref{5-1b}, for every $f\in B(\Omega)$,
\[
\frac{1}{m}\sum_{j=1}^m |f(\xi^j)|^p-\|f\|_{L_p(\mu)}^p
=
\frac{1}{m}\,\EE_{\pmb\eta}\Bigl[\sum_{j=1}^m\bigl(|f(\xi^j)|^p-|f(\eta^j)|^p\bigr)\Bigr].
\]
Therefore,
\begin{align*}
m\,\EE \Er_p(F,\pmb\xi)
&=
\EE_{\pmb\xi}\sup_{f\in F}
\biggl|
\EE_{\pmb\eta}\Bigl[\sum_{j=1}^m\bigl(|f(\xi^j)|^p-|f(\eta^j)|^p\bigr)\Bigr]
\biggr| \le
\EE_{\pmb\xi}\EE_{\pmb\eta}
\sup_{f\in F}
\biggl|
\sum_{j=1}^m\bigl(|f(\xi^j)|^p-|f(\eta^j)|^p\bigr)
\biggr|.
\end{align*}
By symmetry and independence,
\[
\EE_{\pmb\xi}\EE_{\pmb\eta}
\sup_{f\in F}
\biggl|
\sum_{j=1}^m\bigl(|f(\xi^j)|^p-|f(\eta^j)|^p\bigr)
\biggr|
=
\EE_{\pmb\va}\EE_{\pmb\xi}\EE_{\pmb\eta}
\sup_{f\in F}
\biggl|
\sum_{j=1}^m\bigl(|f(\xi^j)|^p-|f(\eta^j)|^p\bigr)\va_j
\biggr|.
\]
Hence,
\begin{align*}
m\,\EE \Er_p(F,\pmb\xi)
&\le
\EE_{\pmb\xi}\EE_{\pmb\eta}\EE_{\pmb\va}
\sup_{f\in F}
\biggl|
\sum_{j=1}^m\bigl(|f(\xi^j)|^p-|f(\eta^j)|^p\bigr)\va_j
\biggr| \le
2\,\EE
\sup_{f\in F}
\biggl|
\sum_{j=1}^m |f(\xi^j)|^p \va_j
\biggr|.
\end{align*}
This proves the lemma.
\end{proof}

By Lemma~\ref{lem-5-2},
\begin{equation}\label{5-2b}
  \EE \Er_p(F,\pmb\xi)
  \le
  \frac{2}{m}\,\EE_{\pmb\xi}\EE_{\pmb\va}
  \Bigl[\sup_{x\in T_p(F(\pmb\xi))}|W_x|\Bigr]
  =
  \frac{2}{m}\,\EE_{\pmb\xi}\EE_{\pmb\va}
  \Bigl[\sup_{x\in T_p(F(\pmb\xi))}|W_x-W_0|\Bigr],
\end{equation}
where
\[
W_x:=\sum_{j=1}^m \va_j x(j),
\quad x\in T_p(F(\pmb\xi))\subset\RR^m.
\]
Now we fix the points $\xi^1,\dots,\xi^m\in\Omega$ and apply Theorem~\ref{thm-1-3} to the process
\[
\{W_x\colon x\in T_p(F(\pmb\xi))\}.
\]
To do so, we need to verify the increment condition~\eqref{exp}. This follows from the following tail estimate for Bernoulli processes.

\begin{lemma}[{\cite{LedTal}*{Lemma~4.3}}]\label{TailsEst}
Let $\va_1,\dots,\va_m$
be a sequence of i.i.d.\!\! random variables, each taking the values $\pm1$ with probability $1/2$.
Then for every $p\in[2,\infty)$ there exists a constant $c(p)>0$ such that for every $\mathbf a=(a_1,\dots,a_m)\in\RR^m$ and every $t>0$,
\[
\PP\biggl(\biggl|\sum_{j=1}^m a_j\va_j\biggr|\ge t\|\mathbf a\|_{p'}\biggr)
\le 2e^{-c(p)t^p}.
\]
\end{lemma}

Therefore, by Theorem~\ref{thm-1-3} and Lemma~\ref{TailsEst}, there exists a constant $C(p)>0$ such that
\[
\EE_{\pmb\va}\Bigl[\sup_{x\in T_p(F(\pmb\xi))}|W_x-W_0|\Bigr]
\le
C(p)\,\gamma_{p,1}\bigl(T_p(F(\pmb\xi)),\|\cdot\|_{p'}\bigr).
\]
Combining this with~\eqref{5-2b}, we obtain
\begin{equation}\label{5-8b}
  \EE \Er_p(F,\pmb\xi)
  \le
  \frac{C(p)}{m}\,
  \EE_{\pmb\xi}\Bl[\gamma_{p,1}\bigl(T_p(F(\pmb\xi)),\|\cdot\|_{p'}\bigr)\Br].
\end{equation}

\subsection{Step 2: Deriving Theorem \ref{Th-main} for \texorpdfstring{$q=1$}{q=1} from Theorem \ref{thmGammaBound}}\label{step3}
Let $q=1$. For simplicity, we define
$m_1:=\frac{m}{m_0}$, and
\begin{equation}\label{4-5d}
  \delta:=\frac{H}{m_1\log m_1},
\end{equation}
where $H$ is the constant from the estimates~\eqref{1-11a} and~\eqref{1-11a-0} for $q=1$.
Since $m_1\ge 16$, by~\eqref{1-11b}, we obtain
\begin{equation}\label{4-5c}
  0<\delta\le \frac{1}{(\log m_1)\cdot e\,(e\log\log m_1)^{p/p_0'}}
  < e^{-1-p/p_0'}.
\end{equation}

We now apply Theorem~\ref{thmGammaBound} with the set $G = F(\pmb{\xi})$
for a fixed $\pmb{\xi}=(\xi^1,\dots,\xi^m)$.
Let $0<b<1$, $\alpha_0 \in (0,1)$, and $n_0 \in \mathbb{N} \cap [2,\infty)$ be parameters to be specified later.  Since $b\,\alpha_0^{1/p}\le n_0$,
Theorem~\ref{thmGammaBound} yields
\begin{align*}
&\ga_{p,1}\bigl(T_p( F(\pmb{\xi}) ), \|\cdot\|_{p'}\bigr)
\le C(p)\Biggl[ A_{m_1}(n_0, b, \al_0)\, m\,  \sup_{f \in F} \|f\|_{L_p( \pmb{\xi})}^p
\\
&+ \f m {m_1^{p_0/p}}\cdot \sup_{f\in F}\|f\|_{L_\infty(\pmb\xi)}^{p-p_0}\cdot
\f{\Bigl[{\rm diam}\bigl(F, \|\cdot\|_{L_{p_1}( \pmb{\xi})}\,\bigr)\Bigr]^{p_0}
+ \Bigl[m_0^{-1/p}\ga_{p,p_0}\,\bigl(F, \|\cdot\|_{L_{p_1}(\pmb{\xi})}\,\bigr)\Bigr]^{p_0}}{b^{p_0-1}}\Biggr],
\end{align*}
where
\[
A_{m_1}(n_0, b, \al_0)
:= bn_0 + (2^{n_0}\, \al_0)^{-1/p} \log m_1.
\]
Since $0\in F$, we have
\[
{\rm diam}\bigl(F,\|\cdot\|_{L_{p_1}(\pmb{\xi})}\bigr)
\le2
\sup_{g\in F}\|g\|_{L_{p_1}(\pmb{\xi})}.
\]
Thus, substituting  into \eqref{5-8b}, and using \eqref{1-11a} and \eqref{1-11a-0} with $q=1$, we deduce
\begin{align*}
&\EE \Er_p(F, \pmb{\xi}) \le {C_1(p)}\Bigg[  A_{m_1}(n_0, b, \al_0)\,\EE  \bigl[\sup_{f \in F} \|f\|_{L_p( \pmb{\xi})}^p\bigr]
+  \Bl(\f H {m_1}\Br)^{p_0/p}\cdot b^{1-p_0}\Bigg].
\end{align*}
Moreover, by  \eqref{4-1a},
\[
\sup_{f \in F} \|f\|_{L_p( \pmb{\xi})}^p=\sup_{f\in F} \frac{1}{m}\sum_{j=1}^m |f(\xi^j)|^p
\le \Er_p(F, \pmb{\xi}) + \sup_{f \in F} \|f\|_{L_p(\mu)}^p.
\]
Thus, setting $L := \sup\limits_{f \in F} \|f\|_{L_p(\mu)}^p + 1$ and using \eqref{4-5d}, we obtain
\begin{align}\label{5-6c}
\EE \Er_p(F, \pmb{\xi})
   &\le C_1(p) \Bigg[  A_{m_1}(n_0, b, \al_0)
      \Bigl(\EE \Er_p(F, \pmb{\xi})+ \sup_{f \in F} \|f\|_{L_p(\mu)}^p\Bigr)
      +\f{\da^{p_0/p} \cdot (\log m_1)^{p_0/p}}{ b^{p_0-1}}\Bigg]\\
  &\le C_1(p) A_{m_1}(n_0, b, \al_0)
\EE \Er_p(F, \pmb{\xi})+  C_1(p) L\, B_{m_1}(n_0, b, \al_0),\notag
\end{align}
where
\beq
B_{m_1}(n_0, b, \al_0)
  := b n_0 + (2^{n_0}\, \al_0)^{-1/p} \log m_1
     + \f{ \da^{p_0/p} \cdot (\log m_1)^{p_0/p}}{b^{p_0-1}}.\label{5-15c}
\enq

We will choose the parameters $n_0, b, \al_0$  so that
\begin{equation}\label{5-6b}
 A_{m_1}(n_0, b, \al_0)
  = bn_0
      + (2^{n_0}\, \al_0)^{-1/p} \log m_1
  \le \frac{1}{2C_1(p)},
\end{equation}
which,  combined with \eqref{5-6c},  will yield
\[
\EE \Er_p(F, \pmb{\xi})  \le 2 C_1(p) L\, B_{m_1}(n_0, b, \al_0).
\]

We now specify the parameters $b, \al_0, n_0$  by minimizing the function $B_{m_1}(n_0, b, \al_0)$ in   \eqref{5-15c},  subject to the constraint \eqref{5-6b}.
Let $c=c(p)\in (0, 1/2)$ be a small constant (depending only on~$p$) to be specified later.
First, choose
\beq
b := c\cdot \frac{\da^{1/p} (\log m_1)^{1/p}}{n_0^{1/p_0}}
\label{4-12a}
\enq
so that $
b n_0 =\f{c^{p_0}  \, \da^{p_0/p}\, (\log m_1)^{p_0/p}}{b^{p_0-1}}$,  and hence
\eq{
B_{m_1}(n_0, b, \al_0)
   &=(c+ c^{1-p_0})(\da \log m_1)^{1/p} n_0^{1/p_0'}
+ (2^{n_0}\, \al_0)^{-1/p} \log m_1\\
&\le 2 c^{-(p_0-1)} (\da \log m_1)^{1/p} n_0^{1/p_0'}
+ (2^{n_0}\, \al_0)^{-1/p} \log m_1.}
Now, setting $\al_0 := ( \log m_1)^{-1}$, we obtain
\[
B_{m_1}(n_0, b, \al_0)
   \le  (\log m_1)^{1/p}
        \left[ 2^{-n_0/p} \log m_1 + 2c^{1-p_0}\da^{1/p} n_0^{1/p_0'} \right].
\]
Second, to balance these last two terms in the brackets, we require that  $2^{-n_0/p} \log m_1 \sim \da^{1/p}$, that is,
$2^{n_0} \sim \da^{-1} (\log m_1)^p$. More precisely, choose  $n_0 \in \NN$ such that
\[
2^{n_0-1} \le c^{-p} \da^{-1} (\log m_1)^p \le 2^{n_0 }.
\]
Using  \eqref{4-5c}, we have \beq\label{5-21-c} n_0 \le C_2(p)  \Bl[ \log \da^{-1} + \log \log m_1+\log c^{-1}\Br]\le C_3(p) |\log c| \cdot|\log \da|,\enq
and, {since $p_0\le p$ and $\frac{1}{p_0'}\le \frac{1}{p'}$},
\begin{align*}
B_{m_1}(n_0, b, \al_0)
   &\le 3 c^{1-p_0} (\log m_1)^{1/p} \da^{1/p} n_0^{1/p_0'}\le C_4(p, c)\, |\log\da|^{ 1/p_0'}\,  (\da\cdot \log m_1)^{1/p}.
\end{align*}
Thus, assuming the chosen parameters satisfy \eqref{5-6b}, we obtain
\[
\EE \Er_p(F, \pmb{\xi})  \le C_5(p, c)L\cdot \bigl( \da\cdot \log m_1\bigr)^{1/p}\cdot | \log  \da|^{1/p_0'},
\]
which  combined with \eqref{4-5d}, gives us the desired result:
\[ \EE \Er_p(F, \pmb{\xi}) \le  C_5(p, c) \Bl(\f {Hm_0} m\Br)^{\f1p} \Bl(\log\log \frac{m}{m_0} + \log \f m{Hm_0}\Br)^{1-\f1{p_0}}\Bigl(1+\sup\limits_{f \in F} \|f\|_{L_p(\mu)}^p\Bigr).
\]

Thus, it remains to show that \eqref{5-6b} holds for a sufficiently small constant $c \in (0, 1/2)$.
 Indeed, from \eqref{4-12a} and \eqref{5-21-c},  we get
\begin{align*}
A_{m_1}(n_0, b,\al_0)&:=b n_0 + 2^{-n_0/p}\, (\log m_1)^{1+\f1p}\le c n_0^{1/p_0'}\da^{1/p}( \log m_1)^{1/p}
+  c \da^{1/p} (\log m_1)^{1/p} \\
   &\le 2c \da^{1/p}  (\log m_1)^{1/p}\, n_0^{1/p_0'}\le C_6(p) c |\log c|^{1/p_0'}
\cdot \Bl(  \da \cdot   |\log \da|^{p/p_0'}\, \log m_1\Br)^{1/p}.
\end{align*}
Since the function $x|\log x|^{p/p_0'}$ is increasing on $(0, e^{-p/p_0'})$,
it follows by \eqref{4-5c} that
\begin{equation*}
\da \cdot  |\log \da|^{p/p_0'} \le \f { (\log\log m_1 {+1+p} +p\log\log \log m_1)^{p/p_0'}} { (\log m_1) \cdot (\log\log m_1)^{p/p_0'}}\le \f {C_7(p)}{\log m_1}.
\end{equation*}
Hence, we have
\begin{equation*}
A_{m_1}(n_0, b,\al_0) \le  C_8(p) c  |\log c|^{1/p_0'}\le C_8(p) c  |\log c|^{1/p'}.
\end{equation*}
Finally, we choose $c = c(p) \in (0,1)$ sufficiently small so that \eqref{5-6b} holds.
This completes the proof of Theorem~\ref{Th-main} for $q = 1$.

\subsection{Step 3: Deriving Theorem \ref{Th-main} for \texorpdfstring{$q>1$}{q>1}}\label{step4}
Having established Theorem~\ref{Th-main} for $q = 1$, we now use this result to prove the case $q > 1$. For this purpose, we require the following lemma.
\begin{lemma}\label{lem-9-1}
Let  $1 \le p, q < \infty$, and let $\pmb{\xi}=(\xi^1,\ldots, \xi^m)$ be a {collection} of $m$ independent random elements in $ \Omega$. Assume that $F$ is a nonempty subset of $B(\Og)$ satisfying
\begin{equation}\label{9-1a}
  \Bigl( \EE \max_{1 \le j \le m} \sup_{f \in F} |f(\xi^j)|^{pq} \Bigr)^{1/q} \le H_0 < \infty.
\end{equation}
{There} exists a universal constant $C > 0$ such that
\begin{align*}
\left( \EE|\Er_p(F, \pmb{\xi})|^q \right)^{1/q}\le \frac{Cq}{1+\log q}\max\Bl\{  \EE \Er_p(F, \pmb{\xi}),\   \frac{H_0}{m}\Br\}.
\end{align*}
\end{lemma}

The proof of Lemma \ref{lem-9-1} relies on the following concentration inequality of Talagrand.

\begin{lemma}[{\cite{Tal89}*{Theorem 1}}]\label{thm-9-3}
Let $\{\vi_j\}_{j=1}^m$ be a sequence of independent, mean-zero random elements taking values in a Banach space $(V,\|\cdot\|_V)$ such that $\EE \|\vi_j\|_V<\infty$ for $1\le j\le m$. Then for every $q\ge 1$ we have
\[
\Biggl( \EE \Bl\|\sum_{j=1}^m \vi_j\Br\|_V^q \Biggr)^{1/q} \le \frac{Cq}{1+\log q}\Biggl (\EE \Bl\|\sum_{j=1}^m  \vi_j \Br\|_V+\left(\EE \max _{1\le j\le m}  \| \vi_j\|_V^q  \right)^{1/q}\Biggr),
\]
where $C$ is a universal constant.
\end{lemma}
\begin{proof}[Proof of Lemma \ref{lem-9-1}]
For brevity, given a real-valued random variable $X$ and $q \ge 1$, we denote
\[
\|X\|_{L_q} := (\EE |X|^{q})^{1/q}.
\]
Let $V$ denote the Banach space of all bounded functions $\vi\colon F\to \RR$ with norm $\|\vi\|_{V}:=\sup\limits_{f\in F} |\vi(f)|$. Then we may write
  \[\Er_p(F, \pmb{\xi})=\sup_{f\in F}  \Bl| \f 1m \sum_{j=1}^m  |f(\xi^j)|^p\, -\,\|f\|_{L_p(\mu)}^p \Br| =\Bl\|\sum_{j=1}^m \vi_{j}\Br\|_V,\]
where $\vi_j\colon F\to \RR$ is defined by
 \[ \vi_j (f): =\f 1m \Bl( |f(\xi^j)|^p-\EE |f(\xi^j)|^p\Br), \quad  f\in F.\]
Clearly, $\vi_1,\dots, \vi_m$ are independent $V$-valued random variables with mean zero  such that
\[\Bl\|\max_{1\le j \le m}\|\vi_j\|_V\Br\|_{L_q}=\frac{1}{m}\bigg\|\max_{1\le j \le m}\sup_{f\in F}\Bl| |f(\xi^j)|^p - \EE|f(\xi^j)|^p\Br| \bigg\|_{L_q} \le \frac{2H_0}{m}.\]
Thus, using Lemma \ref{thm-9-3}, we obtain
\begin{align*}
\|\Er_p(F, \pmb{\xi})\|_{L_q}&=\bigg\|\Bl\|\sum_{j=1}^m \vi_{j} \Br\|_V\bigg\|_{L_q} \le \frac{Cq}{1+\log q} \bigg(\EE \,  \Bl\|\sum_{j=1}^m \vi_j\Br\|_V +  \Bl\| \max_{1\le j\le m} \|\vi_j\|_V\Br\|_{L_q}  \bigg)\\
&\le \frac{Cq}{1+\log q} \bigg(\EE \Er_p(F, \pmb{\xi}) +  \frac{H_0}{m}\bigg)\le \frac{Cq}{1+\log q}\max\Bl\{  \EE \Er_p(F, \pmb{\xi}),\ \frac{H_0}{m}\Br\},
\end{align*}
which proves the claim.
\end{proof}

We are now in a position to prove the claim of Theorem~\ref{Th-main} for all $q > 1$. First, note  that
\begin{equation*}
\max_{1 \le j \le m} \sup_{f \in F} |f(\xi^j)|^p
\le m^{1- {\f {p_0}{p}}} \sup_{f \in F}\|f\|_{L_\infty(\pmb\xi)}^{p-p_0}\cdot \sup_{g \in F}\|g\|_{L_{p_1}( \pmb\xi)}^{p_0}.
\end{equation*}
Thus, \eqref{1-11a-0} implies \eqref{9-1a}
with
\[
H_0 =  H^{p_0/p}m^{1-\frac{p_0}{p}},
\]
and hence, by Lemma~\ref{lem-9-1}, we obtain
\begin{equation}\label{9-2b}
\left( \EE |\Er_p(F, \pmb{\xi})|^q \right)^{1/q}
\le \frac{Cq}{1+\log q}
\max\Bl\{ \EE \Er_p(F, \pmb{\xi}),\ \Bigl(\frac{H}{m}\Bigr)^{p_0/p} \Br\}.
\end{equation}
Next, by H\"older's inequality,
\[
\mathbb{E}\,
\Bigl(\sup_{f\in F}\|f\|_{L_\infty(\pmb\xi)}^{p-p_0}
\sup_{g\in F}\|g\|_{L_{p_1}(\pmb{\xi})}^{p_0}\Bigr)
\le H^{p_0/p},
\]
so we may apply the results from Step~2 (Subsection~\ref{step3}).
For all $m$ satisfying \eqref{1-11b},  we then obtain
\begin{equation} \label{9-2a}
\EE \Er_p(F, \pmb{\xi}) \le C_9(p) \Bl(\f {Hm_0} m\Br)^{\f1p} \Bl(\log\log \frac{m}{m_0} + \log \f m{Hm_0}\Br)^{1-\f1{p_0}}\Bigl(1+\sup\limits_{f \in F} \|f\|_{L_p(\mu)}^p\Bigr).
\end{equation}
On the other hand, \eqref{1-11b} clearly implies that
\[\Bl(\frac{H}{m}\Br)^{p_0/p}\le \Bl(\frac{Hm_0}{m}\Br)^{1/p}\le \text{RHS of \eqref{9-2a}}.
\]
This combined with \eqref{9-2b} and  \eqref{9-2a} then yields
\begin{equation*}
\left( \EE|\Er_p(F, \pmb{\xi})|^q \right)^{1/q}\le \frac{C_{10}(p)q}{1+\log q}\Bl(\f {Hm_0} m\Br)^{\f1p} \Bl(\log\log \frac{m}{m_0} + \log \f m{Hm_0}\Br)^{1-\f1{p_0}}\Bigl(1+\sup\limits_{f \in F} \|f\|_{L_p(\mu)}^p\Bigr),
\end{equation*}
which proves Theorem \ref{Th-main} for any $q>1$.

\subsection{Step 4: Deriving the probability estimate (\ref{eq-probab-est-main})}

We will apply the following simple lemma.

\begin{lemma} \label{lem-9-5}
Let $Z$ be a real valued random variable satisfying
\begin{equation*}
(\EE|Z|^q)^{1/q} \le \frac{q}{1+\log q}, \quad \forall\, q \ge 1.
\end{equation*}
Then we have
\begin{equation}\label{9-5-1b}
\PP(|Z| \ge t) \le \exp\biggl(-\frac{t\log t}{e}\biggr),\quad \forall\, t \ge e.
\end{equation}
\end{lemma}
\begin{proof} Given any $t\ge e$, choose $q=q_t\ge1$ such that $t=\frac{eq}{1+\log q}$. By Chebyshev's inequality,  we obtain
\begin{equation}\label{9-5-2}
\PP(|Z|\ge t) \le \frac{\EE|Z|^q}{t^q}\le \Bl(\frac{q}{t(1+\log q)}\Br)^q =  e^{-q}.
\end{equation}
Since
\[
\log t = 1 + \log q - \log(1+\log q) \le \frac{e q}{t},
\]
it follows  that
\[
q \ge \frac{t \log t}{e}.
\]
Substituting into \eqref{9-5-2},  we prove   the desired estimate \eqref{9-5-1b}.
\end{proof}

Now, if~\eqref{1-11a-0} holds for every $q>1$, then, by the results of Subsection~\ref{step4}, we have
\[
\left( \EE |\Er_p(F, \pmb{\xi})|^q \right)^{1/q}
\le \frac{C(p) q}{1+\log q}\,\Theta,
\quad \forall\, q \ge 1,
\]
where
\[
\Theta :=
\Bl({\f{Hm_0}{m}}\Br)^{\f1p}
\Bl(\log\log \frac{m}{m_0} + \log \f{m}{Hm_0}\Br)^{1-\f1{p_0}}
\Bigl(1+\sup\limits_{f \in F} \|f\|_{L_p(\mu)}^p\Bigr).
\]
Applying Lemma~\ref{lem-9-5} to the random variable
\[
Z = \f{\Er_p(F,\pmb\xi)}{C(p)\Theta},
\]
we obtain the estimate
\[
\mathbb{P}\bigl[\Er_p(F,\pmb{\xi}) > C(p)\Theta\, t\bigr]
\le \exp\Bigl(-\frac{t \log t}{e}\Bigr),
\quad \forall\, t \ge e.
\]

\section{Proofs of Theorems~\ref{Th-Conv} and~\ref{Th-Nik}}\label{sec5}

Let $(\Omega,\mathcal{F}, \mu)$ be a probability space, and   $\pmb{\xi}:=(\xi^1,\ldots,\xi^m)\in\Og^m$ {be} a {collection} of $m$ independent  random elements   satisfying condition~\eqref{probab-cond}. Let $p \ge 2$,   $p_0\in(1, p]$, and $p_1:=\f {pp_0}{p-p_0}$.
We begin with the following intermediate lemma, which allows us to pass from Theorem~\ref{Th-main} to the statement of Theorem~\ref{Th-Conv}, and which will also be useful in the next section for obtaining results on universal discretization.

\begin{lemma}\label{lem-log}
Let $F$ be a set of bounded functions on $\Og$ with  $0\in F$ and
\[
\sup_{f\in F}\|f\|_{L_p(\mu)}\le1.
\]
Assume that for some constants $R, H\ge 16$, $\alpha\in [0, p]$, $q \ge 1$, and an integer $m_0\in  [1, m]$,  the set $F$ satisfies
\begin{equation}\label{q-cond-0}
\Bigl(\mathbb{E}\,
\Bigl[\Bigl(
\sup_{f\in F}\|f\|_{L_\infty(\pmb\xi)}^{p-p_0}
\cdot \sup_{g\in F}\|g\|_{L_{p_1}(\pmb{\xi})}^{p_0}
\Bigr)^q\Bigr]
\Bigr)^{1/q}
\le  R\cdot \Bl(H (\log \tfrac{m}{m_0})^\alpha\Br)^{\f {p_0}p},
\end{equation}
and
\begin{equation}\label{q-cond}
\mathbb{E}\,
\Bigl[\sup_{f\in F}\|f\|_{L_\infty( \pmb\xi)}^{p-p_0}\cdot\bigl|\gamma_{p,p_0}\bigl(F,\|\cdot\|
_{L_{p_1}( \pmb{\xi})}\bigr)\bigr|^{p_0 }\Bigr]
\le R\cdot \Bl(H m_0 (\log \tfrac{m}{m_0})^\alpha\Br)^{\f {p_0}p}.
\end{equation}
Then there exists a constant $C := C(p)>0$  such that, whenever
\begin{equation}\label{5-11a}
m \ge  (8ep)^{2p}\, Hm_0\,  (\log H)^{\alpha}\, (\log\log H)^{p/p_0'},
\end{equation}
one has
\begin{equation}\label{6-4b}
\Bigl(\mathbb{E}\,|\Er_p(F,\pmb{\xi})|^q\Bigr)^{1/q}
\le
\frac{C R q}{1+\log q}
\Bigl(\frac{Hm_0 (\log \tfrac{m}{m_0})^{\alpha}}{m}\Bigr)^{1/p}
\Bigl(\log\log H+\log \frac{m}{Hm_0}\Bigr)^{1-\frac{1}{p_0}}.
\end{equation}
Furthermore, if {\eqref{q-cond-0} holds}  for all $q\ge 1$, and if for some constants $\lambda\ge e$ and $\varepsilon \in (0,\tfrac12]$,
\begin{equation}\label{6-9a}
m \ge (\lambda\,\varepsilon^{-1})^p(\log \varepsilon^{-1})^{\alpha+\frac{p}{p_0'}}
Hm_0\, (\log H)^{\alpha}\, (\log\log H)^{p/p_0'},
\end{equation}
then
\begin{equation}\label{eq-probab-est}
\mathbb{P}\bigl[\Er_p(F,\pmb{\xi})>\varepsilon\bigr]
\le 2  \exp\biggl(-\f { c \lambda}{ {R\log R}\cdot (\log \lambda)^{\frac{\alpha}{p}  -\frac{1}{p_0} }}\biggr)
\end{equation}
{for some constant  $c:=c(p)>0$.}
\end{lemma}

\begin{proof}
We  will apply
Theorem~\ref{Th-main} to the scaled  function class
\[
\wt {F}:=R^{-1/p} F=\{ R^{-1/p} f\colon f\in F\}.
\]
Clearly, $\wt F$ satisfies conditions \eqref{1-11a-0} and \eqref{1-11a} with  constant
$
\wt H := H(\log \tfrac{m}{m_0})^{\alpha}$. We also need to   verify condition \eqref{1-11b} with constant $\wt H$;  namely,
\begin{equation}\label{6-6b}
m \ge \max\Bigl\{16m_0,\  e\bigl(e\log\log \tfrac{m}{m_0}\bigr)^{p/p_0'}  \wt Hm_0\Bigr\}.\end{equation}
To see this, consider the function
\[
\psi(x) := \f{x}{(\log x)^{\alpha}\,e(e \cdot \log\log x)^{p/p_0'}},\   \ x> e.
\]
Since for $x\ge e^{2p}$,
$$\frac{d}{dx} \Bl[\log \psi(x) \Br]= \frac{1}{x \log x} \left( \log x - \alpha - \frac{p/p_0'}{\log \log x} \right)\ge  \frac{1}{x \log x} \left( \log x - 2p\right)\ge 0,  $$
$\psi$
is an increasing function on  $[e^{2p},\infty)$. Thus,   \eqref{5-11a} implies that  $m\ge e^{2p} m_0$ and
\beq\label{6-8a}
\psi\bigl(\tfrac{m}{m_0})
= \f{\tfrac{m}{m_0}}{(\log \tfrac{m}{m_0})^{\alpha}e(e \cdot \log\log \tfrac{m}{m_0})^{p/p_0'}}
\ge \psi(x_*),
\enq
where
\[
x_* := (8 e p)^{2p} H (\log H)^{\alpha} (\log\log H)^{p/p_0'}.
\]
A straightforward calculation shows that $\psi(x_*)\ge H$ {and \eqref{6-6b} then follows from \eqref{6-8a}}.

Now applying Theorem~\ref{Th-main} to $\wt F$, and  recalling that
\[
{\sup_{f\in F}\|f\|_{L_p(\mu)}\le1,}
\]
 we obtain
\[R^{-1} \cdot \left( \EE |\Er_p(F, \pmb{\xi})|^q \right)^{1/q}
=\left( \EE |\Er_p(\wt F, \pmb{\xi})|^q \right)^{1/q}
\le \frac{C(p) q}{1+\log q}\, \wt \Theta,
\]
where
\[
 \wt \Theta :=
\Bigl(\frac{\wt Hm_0 }{m}\Bigr)^{1/p}
\Bl(\log\log \frac{m}{m_0} + \log \f{m}{\wt Hm_0}\Br)^{1-\frac{1}{p_0}}.
\]
Since $H\le \wt H$ and
\[ \log\log \f m{m_0} =\log\Bl[ \log \f {m}{Hm_0 } +\log H\Br] \le 2 \log \f {m}{Hm_0 }+2\log\log H,\]
it follows that
\[  \left( \EE |\Er_p(F, \pmb{\xi})|^q \right)^{1/q}
\le \frac{C(p) R q}{1+\log q}\, \wt \Theta\le \frac{3C(p) R q}{1+\log q}\, \Theta,
\]
where
\[
\Theta :=
\Bigl(\frac{ H m_0 (\log \f m{m_0})^\al  }{m}\Bigr)^{1/p}
\Bl(\log\log H + \log \f{m}{ Hm_0}\Br)^{1-\frac{1}{p_0}}.
\]
This proves \eqref{6-4b}.

\medskip

Finally, assuming that
{\eqref{q-cond-0} holds} for all $q\ge 1$, we  prove the probability estimate
\eqref{eq-probab-est} under  the sample size condition \eqref{6-9a}.
For simplicity, we set $\beta := 1 + \frac{\alpha}{p} - \frac{1}{p_0}$. Clearly, $0<\beta\le 1+\f 1{p'}$.
Without loss of generality, we may assume that $\lambda \ge C_* R(\log R)^\be$ for some sufficiently large constant  $C_*=C_\ast(p)$   that will be specified later.
Indeed, if  ${e\le} \lambda < C_* R(\log R)^\be$, then \eqref{eq-probab-est}    holds trivially for any constant   $0<c \le {\frac{\log 2}{C_*(3+\log C_*)} }$, using   monotonicity of the function $x (\log x)^{1-\be}$ on {$[e, \infty)$}.

By applying Theorem~\ref{Th-main} to the scaled  function class  $\wt F$, we obtain
\begin{equation}\label{6-0}
  \PP\Bl[\Er_p( F,\pmb\xi) > C_1(p) R\Theta \cdot t \Br]
  \le \exp\left[-\frac{t \log t}{e}\right],
  \qquad \forall\, t \ge e.
\end{equation}
We will invoke \eqref{6-0} with
 $t=t_0:=\f{\varepsilon}{C_1(p) R\Theta}$.
To this end, we claim that for some constant $C(p) > 0$,
\begin{equation}\label{6-8}t_0 \ge \f 1 {C(p)} \frac{\lambda}{R (\log \lambda)^\beta}.
\end{equation}
For the moment, we assume  the estimate \eqref{6-8} and proceed with the proof of
 the desired probability estimate
\eqref{eq-probab-est}.

By  monotonicity of the function $x(\log x)^{-\be}$ on the interval $[e^\be, \infty)$, and recalling $\ld\ge C_\ast R (\log R)^\be$, we obtain from \eqref{6-8} that
\[
t_0\ge \f 1 {C(p)}\f { C_\ast   (\log R)^\be} { ( \log C_\ast +\log R +\be \log\log R)^\be }\ge \f 1 {C(p)} \f {C_\ast} {  ({3+}\log C_\ast)^2}\ge e,
\]
provided that $C_\ast=C_\ast(p)$ is  large enough.
Thus, we may apply  \eqref{6-0} with $t=t_0$ to obtain
 \begin{equation}\label{6-0-c}
  \PP\Bl[\Er_p( F,\pmb\xi) > \va \Br]
  \le \exp\left[-\frac{t_0 \log t_0}{e}\right].
\end{equation}
Using \eqref{6-8} and monotonicity of the function $x\log x$   on $[e, \infty)$, we  have
\eq{
t_0\log t_0 \ge \frac{c(p) \lambda}{R (\log \lambda)^\beta}\Bl[ \log  \ld -\log R-\beta \log\log \ld-\log C(p)\Br].
}
Given that $\lambda \ge C_\ast R (\log R)^\beta$ and that the function $(1-\frac{1}{\log R}) y - \beta \log y$ is increasing on $[\log ({4}R), \infty)$, a direct calculation shows that
\[
\log\lambda - \log R - \beta \log\log \lambda - \log C(p) \ge \frac{\log \lambda}{\log R}
\]
for any sufficiently large constant $C_\ast = C_\ast(p)$. It follows that
\[ t_0\log t_0\ge \f {c(p) \ld} {R\log R (\log\ld)^{\be-1}}.\]
Substituting into \eqref{6-0-c}, we prove the  probability estimate
\eqref{eq-probab-est}.

It remains to prove \eqref{6-8}.  Let
\[
\varphi(x) := \frac{x^{1/p}}{(\log x+\log\log H)^{1/p_0'}(\log x + \log H)^{\alpha/p}},
\quad x > 1.
\]
A straightforward calculation shows that for $x \ge e^{2p}$,
\begin{align*}
 \f d{dx}(\log \vi(x))
&= \frac{1}{x}
\Bigl[
\frac{1}{p}
- \frac{1}{p_0'} \cdot \frac{1}{\log x+\log\log H}
- \frac{\alpha}{p} \cdot \frac{1}{\log x + \log H}
\Bigr] \\
&> \frac{1}{x}
\Bigl[
\frac{1}{p}
- \frac{2}{\log x}
\Bigr]
\ge 0,
\end{align*}
so $\varphi$ is increasing on $[e^{2p},\infty)$.
Since we may assume that $C_* \ge e^2$,  \eqref{6-9a} implies
\[
\frac{m}{Hm_0}
\ge (\lambda\varepsilon^{-1})^p(\log \varepsilon^{-1})^{\alpha+\frac{p}{p_0'}}
(\log H)^{\alpha}(\log\log H)^{p/p_0'}
> e^{2p}.
\]
It follows that
\begin{align*}
t_0&:=\f{\varepsilon}{C_1(p) R\Theta}
= \frac{\varepsilon}{C_1(p)R}\cdot
\varphi\Bl(\frac{m}{Hm_0}\Br) \\
&\ge
\frac{\varepsilon}{C_1(p) R}\cdot
\varphi\Bl(
(\lambda\varepsilon^{-1})^p(\log \varepsilon^{-1})^{\alpha+\frac{p}{p_0'}}
(\log H)^{\alpha}(\log\log H)^{p/p_0'}
\Br) \\
&\ge\f 1 {C_2(p) R}\cdot
\frac{\lambda\,
(\log \varepsilon^{-1})^{\frac{1}{p_0'}+\frac{\alpha}{p}}
(\log H)^{\alpha/p}
(\log\log H)^{1/p_0'}}
{
(\log \lambda + \log \varepsilon^{-1} + \log\log H)^{1/p_0'}
(\log H + \log \lambda + \log \varepsilon^{-1})^{\alpha/p}} \\
&\ge\f 1 {C(p)R}\cdot
\frac{\lambda}
{(\log \lambda)^{\frac{1}{p_0'}+\frac{\alpha}{p}}}=\f 1 {C(p)}\cdot
\frac{\lambda}
{R(\log \lambda)^\be}.
\end{align*}
This proves the claim \eqref{6-8}.
\end{proof}

For the convenience of later applications, we record the following consequence of Lemma \ref{lem-log}.
\begin{remark}\label{rem-probab-est}
Under the conditions of Lemma~\ref{lem-log}, assume further that for some $L\ge e$,
\[
Hm_0\,  (\log H)^{\alpha}\, (\log\log H)^{p/p_0'}
\le L\cdot H_0.
\]
If conditions~\eqref{q-cond} and \eqref{q-cond-0} hold for every $q > 1$, and if
for some $\lambda \ge e$ and $\varepsilon \in (0,\tfrac12]$,
\[
m \ge (\lambda\,\varepsilon^{-1})^p(\log \varepsilon^{-1})^{\alpha+\frac{p}{p_0'}} H_0,
\]
 then we have
\[
\mathbb{P}\bigl[\Er_p(F,\pmb{\xi}) > \varepsilon\bigr]
\le 2 \exp\biggl(-\frac{c(p) \lambda}{(R\log R)( L^{1/p} \log L)  \cdot(\log \lambda)^{\f \al p-\f 1{p_0}}}\biggr).
\]

To see this,  set $\be:=1+\frac{\alpha}{p}  - \frac{1}{p_0}$.
If  $\lambda\in [e, eL^{1/p}]$,  then
\[ \ld (\log \ld)^{1-\be}\le e L^{1/p} \Bl( 1+\f 1p \log L\Br)^{1-\be}\le  e p L^{1/p}(\log L)^{1-\be}\le e p L^{1/p}(\log L),\]
which implies
\[
\mathbb{P}\bigl[\Er_p(F,\pmb{\xi}) > \varepsilon\bigr]
\le 1\le
2 \exp\biggl(-\frac{c(p)  \lambda}{(R \log R) (L^{1/p}\log L)\cdot (\log \lambda)^{\be-1}}\biggr).
\]
If $\lambda\ge eL^{1/p}$, we use Lemma~\ref{lem-log} with $\ld L^{-1/p}$ in place of $\ld$ to obtain
\eq{
\mathbb{P}\bigl[\Er_p(F,\pmb{\xi}) > \varepsilon\bigr]
&\le 2 \exp\biggl(-\frac{{c_1(p)}\lambda L^{-1/p}}{R\log R\cdot (\log (\lambda L^{-1/p}))^{\be-1}}\biggr)\\
&\le
2 \exp\biggl(-\frac{{c_2(p)}  \lambda  }{(R\log R)(L^{1/p} \log L)\cdot (\log \lambda)^{\be-1}}\biggr),
}
where the last step treats the cases $\beta \ge 1$ and $0 < \beta < 1$ separately, and  the latter is further split into $\lambda \ge e L^{2/p}$ and $e L^{1/p} \le \lambda < e L^{2/p}$.
\end{remark}

\begin{proof}[Proof of Theorem~\ref{Th-Conv}]
Since $F\subset B(\Og, \RR)$ is $p$-convex, we may combine
 Lemmas~\ref{lem-2-5} and  \ref{cor-8-2} to obtain that  for any $q\ge 1$,
\[
\Bigl(\mathbb{E}
\bigl|\ga_{p,p}(F, \|\cdot\|_{L_\infty(\pmb\xi)})\bigr|^{pq}\Bigr)^{1/q}
\le C(p,\eta)\,(\log m)\,
\Bigl(\mathbb{E}\sup_{f\in F}\|f\|_{L_\infty(\pmb\xi)}^{pq}\Bigr)^{1/q}.
\]
Together with condition \eqref{3-7a}, this implies
\[ \Bigl(\mathbb{E}
\bigl|\ga_{p,p}(F, \|\cdot\|_{L_\infty(\pmb\xi)})\bigr|^{pq}\Bigr)^{1/q}\le  C(p,\eta)\,(\log m) H,\]
which ensures   that the conditions of Lemma~\ref{lem-log}
are satisfied with $p_0=p$ and  $\al=m_0=1$.  Furthermore,   the assumption on the sample  size estimate  implies
$ \log \f m H\ge c\log\log H$ for some constant $c=c(p,\eta)>0$.  Theorem~\ref{Th-Conv} then   follows directly from Lemma~\ref{lem-log}.
\end{proof}

\begin{proof}[Proof of Theorem~\ref{Th-Nik}]
The case where $X_N \subset B(\Omega; \mathbb{R})$ consists of real-valued functions follows immediately from Theorem~\ref{Th-Conv} applied to   the $p$-convex set
\[
F=X_N^p:=\bigl\{f\in X_N\colon \|f\|_{L_p(\mu)}\le 1\bigr\}.
\]
It remains to establish the complex case $X_N \subset B(\Omega)$. We  apply   Lemma~\ref{lem-log} to $F=X_N^p$ with parameters  $m_0=\al=1$ and  $p_0=p$.   Following the conditions of the lemma, it suffices to  verify  that for any $\pmb\xi=(\xi^1,\cdots, \xi^m)\in \Og^m$,
\beq\label{6-13}
\ga_{p,p}(X_N^p, \|\cdot\|_{L_\infty(\pmb\xi)}) \le  C(p)\,(\log m)^{1/p} H^{1/p}.
\enq

To prove \eqref{6-13},  consider the product  probability space
$(\widetilde{\Omega}, \wt {\mathcal{F}}, \wt \mu):=(\{0,1\},  \nu)\times (\Omega,\mathcal{F}, \mu)$, where $\nu$ is the uniform probability measure on  $\{0, 1\}$. Explicitly, for any $A\in\wt{\mathcal{F}}$,
\[
\widetilde{\mu}(A):=
\frac{1}{2}\Bigl[\mu\bigl(\big\{x\in\Og\colon (0,x)\in A\big\}\bigr) + \mu\bigl(\big\{x\in\Og\colon (1,x)\in A\big\}\bigr)\Bigr].
\]
Given $\pmb\xi=(\xi^1,\cdots, \xi^m)\in \Og^m$,  we define
\[
\widetilde{\pmb\xi}:=((0,\xi^1),\ldots,(0,\xi^m),(1,\xi^1),\ldots,(1,\xi^m))\in\widetilde\Omega^{2m}.
\]
For any   $f\in B(\Omega)$, we define its real-valued counterpart  $\wt f:\wt\Og\to \RR$ by
\[
\widetilde{f}(0,x):= {\rm Re}\, f(x)\quad \text{ and }\quad
\widetilde{f}(1,x):= {\rm Im}\, f(x),\  \  x\in\Og.
\]
Then  the following norm equivalence holds for each $f \in B(\Omega)$,
\beq\label{6-14a}  \|\wt f\|_{L_\infty(\wt {\pmb{\xi}})}\le
\|f\|_{L_\infty(\pmb\xi)}\le \sqrt{2}\,
\|\widetilde{f}\|_{L_\infty( \widetilde{\pmb{\xi}})},
\enq
where
 \[
\|\widetilde{f}\|_{L_\infty(\widetilde{\pmb{\xi}})}
:=
\max_{1\le j\le 2m} |\wt f(\wt {\xi}^j)|=\max_{1\le j\le m} \max\Bl\{ |{\rm Re}\, f(\xi^j)|, |{\rm Im}\, f(\xi^j)|\Br\} .
\]
Furthermore,  for $p\ge 2$,   we also have
\begin{align}\label{6-15}
\|\widetilde{f}\|_{L_p(\widetilde{\mu})}^p
&= \frac{1}{2}\Bigl(
\int_\Omega |{\rm Re}\, f|^p\, d\mu
+ \int_\Omega |{\rm Im}\, f|^p\, d\mu
\Bigr)
\le \f 12 \int_\Omega \bigl(|{\rm Re}\, f|^2+ |{\rm Im}\, f|^2\bigr)^{p/2}\, d\mu\\
&=\f 12 \|f\|_{L_p(\mu)}^p\le 2^{\f p2-1} \|\widetilde{f}\|_{L_p(\widetilde{\mu})}^p.
\notag
\end{align}

Finally, we define $Y_{N}:=\{\widetilde{f}\colon f\in X_N\}$, which is a real linear subspace of $B(\widetilde{\Omega}; \mathbb{R})$ with dimension at most $2N$.
{The estimate} \eqref{6-14a}  implies  that
\[
\ga_{p,p}(X_N^p, \|\cdot\|_{L_\infty(\pmb\xi)})
\le \sqrt{2}\ga_{p,p}\Bigl(\widetilde{X_N^p}, \|\cdot\|_{L_\infty( \widetilde{\pmb{\xi}})}\Bigr),\]
where  $\widetilde{X_N^p}:=\{\widetilde{f}\colon f\in X_N^p\}$.
However, by \eqref{6-15},
 \[
\widetilde{X_N^p} \subset  Y_N^p:=\bigl\{\widetilde{f}\in Y_N\colon \|\widetilde{f}\|_{L_p(\widetilde{\mu})}\le 1\bigr\}.
\]
It  follows that
\beq \label{6-16}
\ga_{p,p}(X_N^p, \|\cdot\|_{L_\infty(\pmb\xi)})
\le \sqrt{2}\ga_{p,p}\Bigl(Y_N^p, \|\cdot\|_{L_\infty( \widetilde{\pmb{\xi}})}\Bigr).
\enq
On the other hand, since $X_N \in \textnormal{NI}_{p,\infty}(H)$,  using   \eqref{6-15}, we obtain  that for any $f\in X_N$,
\[
\|\widetilde{f}\|_{L_\infty(\widetilde{\pmb{\xi}})}
\le \|f\|_{L_\infty(\pmb{\xi})}
\le H^{1/p}\|f\|_{L_p(\mu)}
\le \sqrt{2}H^{1/p}\|\widetilde{f}\|_{L_p(\widetilde{\mu})},\]
which implies  \[
\sup_{\wt f\in Y_N^p} \|\wt f\|_{L_\infty(\widetilde{\pmb{\xi}})}\le \sqrt{2}H^{1/p}.
\]
Proceeding as in the proof  of Theorem~\ref{Th-Conv},  using the $p$-convexity of $Y_N^p$,  and applying Lemmas~\ref{lem-2-5} and \ref{cor-8-2} to the finite subset $\{\widetilde{\xi}^1, \dots, \widetilde{\xi}^{2m}\}\subset \wt\Og$, we deduce
\begin{align*}
\ga_{p,p}(Y_N^p,
\|\cdot\|_{L_\infty(\widetilde{\pmb{\xi}})})&\le C_p \sup_{n\ge 0} 2^{\f np} e_n(Y_N^p, \|\cdot\|_{L_\infty(\widetilde{\pmb{\xi}})})\le C_p \Bl(\sup_{\wt f\in Y_N^p} \|\wt f \|_{L_\infty(\wt{\pmb{\xi}})}\Br) (\log m)^{1/p}\\
&\le  C_p\,(\log m)^{1/p}\, H^{1/p}.
\end{align*}
This combined with \eqref{6-16} confirms  the estimate  \eqref{6-13}. Thus,  applying Lemma \ref{lem-log} completes the proof of  Theorem~\ref{Th-Nik} in the complex case.
\end{proof}

\begin{remark}
Let $X_N \subset B(\Omega)$ be an $N$-dimensional linear space of bounded functions.
For $1 \le p < \infty$ and $x \in \Omega$, define the \emph{$L_p$ Christoffel function} associated with $X_N$ by
\[
\lambda_p(x)\equiv \lambda(X_N,p;x)
:=
\inf\Bigl\{\int_\Omega |f(z)|^p\,\mu(dz)\colon f\in X_N,\ |f(x)|=1\Bigr\}.
\]
Then $\lambda_p(x)>0$ for every $x\in\Omega$.
Define
\[
(Uf)(x):=\lambda_p(x)^{1/p}f(x), \quad f\in X_N,\ x\in\Omega,
\]
and set
\[
H:=\int_\Omega \lambda_p(x)^{-1}\,\mu(dx).
\]
Assume that $H<\infty$, and define a probability measure $\nu$ on $\Omega$ by
\[
\nu(dx)=\frac{1}{H\lambda_p(x)}\,\mu(dx).
\]
By the definition of $\lambda_p(x)$, for every $f\in X_N$ and every $x\in\Omega$ one has
\[
|(Uf)(x)|^p=\lambda_p(x)|f(x)|^p\le \int_\Omega |f(z)|^p\,\mu(dz)
= H\int_\Omega |Uf(z)|^p\,\nu(dz).
\]
Hence,
\[
\|Uf\|_\infty \le H^{1/p}\,\|Uf\|_{L_p(\nu)},
\quad \forall\, f\in X_N;
\]
namely,
\[
Y_N:=\{Uf\colon f\in X_N\}\in \textnormal{NI}_{p,\infty}(H).
\]
Applying Theorem~\ref{Th-Nik} to $Y_N$ and $\nu$, we conclude that for every $2<p<\infty$ there exists a constant $C(p)>0$ such that, whenever
\[
m \ge C(p)\,H\,\log H\,(\log\log H)^{p-1},
\]
there exist points $x_1,\dots,x_m\in\Omega$ satisfying
\[
\frac12\,\|f\|_{L_p(\mu)}^p
\le
\sum_{j=1}^m \frac{H\lambda_p(x_j)}{m}\,|f(x_j)|^p
\le
\frac32\,\|f\|_{L_p(\mu)}^p,
\quad \forall\, f\in X_N.
\]

We note that $L_p$ Christoffel functions have been extensively studied in the classical setting when $X_N$ is a space of algebraic or trigonometric polynomials. In many such cases, sharp pointwise estimates for $\lambda_p(x)$ imply that
$H = K N$ with a uniformly bounded constant $K$.
See, for instance, \cite{Ma-To}*{Section~4.3}.
\end{remark}

\section{Proof of Theorem \ref{Th-univ}}\label{universal}

In this section, we deduce Theorem \ref{Th-univ} from Lemma~\ref{lem-log}. Before delving into the technical details, we outline the primary strategy of the proof. We define
\[
F:=\Sigma_s^p(\mathcal{D}_N):=\bigl\{f\in \Sigma_s(\mathcal{D}_N)\colon \|f\|_{L_p(\mu)}\le1\bigr\}.
\]
Then, observing that $\Er_p(F,\pmb{\xi})= \Er_2(T_{p/2}(F),\pmb{\xi})$ for any $\pmb\xi=(\xi^1,\cdots, \xi^m)\in \Og^m$, we apply Lemma \ref{lem-log} with the parameter $p=2$ to the modified function class $T_{p/2}(F)$ rather than directly to $F$. The proof of Theorem \ref{Th-univ} proceeds in two stages.

First, in Subsection \ref{subsec-7-1}, we establish two technical lemmas. Lemma \ref{lem-6-2} will imply that for any  $1\le p_0<2$, we have
\eq{& \  \ga_{2, p_0} (T_{p/2}(F), \|\cdot\|_{L_{p_1}(\pmb\xi)})
 \le  C(p)  \Bl[ \ga_{p, \f {pp_0}2} (F, \|\cdot\|_{L_{q}(\pmb\xi)})\Br]^{\f {p}2 },
}
where $p_1:=\f {2p_0}{2-p_0}$ and $q:=\f {pp_0} {2-p_0}$. Lemma \ref{lem-univ-gamma-est}  then provides bounds for the  chaining functionals $  \ga_{p, \f {pp_0}2} (F, \|\cdot\|_{L_{q}(\pmb\xi)})$ for any $\pmb\xi=(\xi^1,\cdots, \xi^m)\in\Og^m$. This bound constitutes the core technical component of the proof of Theorem~\ref{Th-univ}.

Second, in Subsection \ref{subsection-1-2}, we
use the estimates from Lemmas \ref{lem-6-2} and \ref{lem-univ-gamma-est} to verify the conditions of Lemma \ref{lem-log} for the function class $T_{p/2}(F)$ and the parameter $p=2$. Applying  Lemma \ref{lem-log} then allows us  to estimate the probability $\PP[ \Er_2(T_{p/2}(F),\pmb{\xi})>\va]$ for  any $\va\in (0, \f12]$, which concludes the proof of Theorem \ref{Th-univ}.

Throughout the remainder of this section,  $(\Omega, \mathcal{F}, \mu)$ denotes our base probability space. Any supplementary probability measure $\nu$ on $\Omega$ is assumed to be defined on an extended $\sigma$-algebra $\mathcal{F}_1$ containing $\mathcal{F}$.
Recall that we define
$$\|f\|_\infty := \sup_{x\in\Omega} |f(x)| \quad \text{for all } f\in B(\Omega).$$
In our applications, $\nu$ will be a uniform probability measure supported on a finite  sequence of points in $\Og$.

\subsection{Technical lemmas}\label{subsec-7-1}

We start with the following simple observation for the $\gamma$-functionals.
\begin{lemma}\label{lem-6-2}
Assume that  $\lambda \in (0,1]$, $\alpha>0$, and $\beta\ge 1/\ld$.
Then there exists a constant $C(\al)>0$ such that for any $p_1 \in [\lambda^{-1},\infty]$, any $F \subset B(\Omega)$, and any probability measure $\nu$ on $\Og$ ,
\beq\label{7-1-0}
\gamma_{\alpha,\beta}\bigl(T_\lambda(F), \|\cdot\|_{L_{p_1}(\nu)}\bigr)
\le
C(\alpha)
\bigl[\gamma_{\lambda\alpha, \lambda\beta}\bigl(F,\|\cdot\|_{L_{\lambda p_1} (\nu)}\bigr)\bigr]^\lambda.
\enq
\end{lemma}
\begin{proof} To begin, we recall the elementary  inequality for any $0 < \lambda \le 1$ and $a, b \in \mathbb{C}$:$$||a|^\lambda - |b|^\lambda| \le |a - b|^\lambda.$$
Applying this inequality  pointwise to functions $f_1, f_2 \in B(\Omega)$ yields
\[
\Bl\||f_1|^\ld-|f_2|^\ld\Br\|_{L_{p_1}(\nu)}\le \|f_1-f_2\|_{L_{\lambda p_1}(\nu)}^{\lambda}.
\]
It follows that for every $f\in F$ and every subset $A\subset F$, {
\[ \varrho_{L_{p_1}(\nu)} \big(|f|^\ld,   T_\ld(A)\big)\le  \varrho_{L_{\ld p_1}(\nu)} \big(f,   A\big)^\ld,\]}
where
\[\varrho_{L_{r}(\nu)} (f,  A)=\inf_{g\in A}\|f-g\|_{L_{r}(\nu)},\   \ 1\le r\le \infty.\]
 Thus,  for any  sequence $\{A_n\}_{n=0}^\infty$ of  finite subsets  of $F$  with $|A_n| \le N_n$ for all $n\ge 0$, we have {
\[
\sup_{f\in F}\sum_{n=0}^\infty\Bigl[2^{n/\alpha} \varrho_{L_{p_1}(\nu)} \big(|f|^\ld,   T_\ld(A_n)\big)\Bigr]^\beta
\le \sup_{f\in F}\sum_{n=0}^\infty\Bigl[2^{n/(\lambda\alpha)} \varrho_{L_{\ld p_1}(\nu)} \big(f, A_n\big)\Bigr]^{\lambda\beta }.
\]}
Taking the infimum over all such sequences of subsets $\{A_n\}$, and applying the equivalence \eqref{eq-gamma-equiv} between the $\gamma$ and $\gamma^*$ functionals, we arrive at the desired estimate \eqref{7-1-0}.
\end{proof}

The main technical component of the proof of Theorem~\ref{Th-univ} is contained in the following lemma.

\begin{lemma}\label{lem-univ-gamma-est}\  Assume that   $s, N\in\mathbb{N}$ and  $4 \le s \le N$. Let   \[
\mathcal{D}_N := \{\varphi_1,\dots,\varphi_N\} \subset B(\Omega)
\] be a  uniformly bounded $4s$-sparse Riesz system in $L_2(\Og, \mu)$
with constant $K\ge 16$.
  Then,  given any $p \in [1,2]$,  there exist a constant $C:=C(p) > 0$ and a number
$p_0 := p_0(s,K) \in [\f 32, 2)$ such that for any probability measure $\nu$ on $\Og$, and  for $q:=\f {pp_0}{2-p_0}\ge 3$, we have
\beq\label{7-3a}
\sup_{f\in \Sigma_s^p(\mathcal{D}_N)}\|f\|_{\infty}^{p-\frac{pp_0}{2}}\ \cdot
\ \Bigl[\gamma_{p,\frac{pp_0}{2}}\bigl(\Sigma_s^p(\mathcal{D}_N),
\|\cdot\|_{L_{q}( \nu)}\bigr)\Bigr]^{\frac{pp_0}{2}}
\le
C \bigl(Ks\, \log s\, \log(Ks)\, \log N \bigr)^{\f {p_0}2}
\enq
and
\beq\label{7-3b}
\sup_{f\in \Sigma_s^p(\mathcal{D}_N)}\|f\|_{\infty}^{p-\frac{pp_0}{2}}
\cdot \sup_{g\in \Sigma_s^p(\mathcal{D}_N)}\|g\|_{L_{q} (\nu)}^{\frac{pp_0}{2}}
\le C(Ks)^{\f {p_0}2},
\enq
where
\[
\Sigma_s^p(\mathcal{D}_N)
:=
\bigl\{f\in \Sigma_s(\mathcal{D}_N)\colon \|f\|_{L_p(\mu)}\le 1\bigr\}.
\]
\end{lemma}

\begin{proof}
For clarity, we divide the proof of this lemma into several steps.\\

\noindent  \textbf{Step 1 (Reduction).} By adapting the argument from the proof of Theorem~\ref{Th-Nik}, we  show that it suffices to establish the lemma for real-valued functions and uniformly bounded $2s$-sparse (rather than $4s$-sparse) real Riesz systems.\\

To this end, we  consider  the product space $\widetilde\Omega:=\{0,1\}\times\Omega$ equipped  with the probability  measure
\[
\widetilde{\mu}(A):=\frac12\Bigl(\mu\bigl(\{x\in\Og\colon  (0,x)\in A\}\bigr)+\mu\bigl(\{x\in\Og\colon (1,x)\in A\}\bigr)\Bigr),\quad A\subset \wt \Og.
\]
For any $f\in B(\Omega)$,  define its real-valued counterpart $\widetilde f\in B(\widetilde\Omega; \mathbb{R})$ by
\[
\widetilde f(0,x):={\rm Re}\, f(x),
\quad
\widetilde f(1,x):={\rm Im}\, f(x),\  \ x\in\Og.
\]
We  write $f^{\sim}$ for $\wt f$ whenever it is more convenient. By definition,  for any $f\in B(\Og)$ and $a, b\in\RR$,
\beq\label{7-3} \bigl((a+ib) f\bigr)^{\sim } =a\wt f + b (i f)^{\sim}.\enq
Furthermore, since for any $r\ge 2$,
\[
|{\rm Re}\, f|^r+|{\rm Im}\, f|^r\le |f|^r \le  2^{\f r2-1} \Bl(|{\rm Re}\, f|^r+|{\rm Im}\, f|^r\Br),
\]
it follows that
\beq \label{7-4}
2^{1/r}\|\widetilde f\|_{L_r(\widetilde\mu)} \le \|f\|_{L_r(\mu)}
\le \sqrt{2} \,\|\widetilde f\|_{L_r(\widetilde\mu)},\  \ \forall\, f\in B(\Og).
\enq

Next, we define  the  system of real-valued functions
\[
\widetilde{\mathcal D}_{2N}:=\bigl\{\widetilde{\varphi_1},\ldots,\widetilde{\varphi_N},
(i\varphi_1)^\sim,\ldots,(i\varphi_N)^\sim\bigr\}
\subset B(\widetilde\Omega; \mathbb{R}).
\]
Note that  $\|\widetilde{\varphi_j}\|_\infty\le 1$ and $\|(i\varphi_j)^\sim\|_\infty\le 1$ for all $j$.
Furthermore,  given a  $4s$-sparse real vector   $(a_1, \dots, a_N, b_1, \dots, b_N ) \in \mathbb{R}^{2N}$,  the vector $(a_1+ib_1, \ldots, a_N+ib_N)\in \mathbb{C}^N$ is $4s$-sparse in $\mathbb{C}^N$, thus,    using \eqref{riesz},  \eqref{7-3} and \eqref{7-4},
we obtain
\begin{align*}
\sum_{j=1}^N(|a_j|^2+|b_j|^2)
&\le K\Bigl\|\sum_{j=1}^N(a_j+ib_j)\varphi_j\Bigr\|_{L_2(\mu)}^2\le  2K\Bigl\|\sum_{j=1}^N\big((a_j+ib_j)\varphi_j\big)^\sim\Bigr\|_{L_2(\wt\mu)}^2
\\
&=2K
\Bigl\|\sum_{j=1}^Na_j\widetilde{\varphi_j}
+ \sum_{j=1}^Nb_j(i\varphi_j)^\sim\Bigr\|_{L_2(\widetilde{\mu})}^2.
\end{align*}
This means  that   $\widetilde{\mathcal{D}}_{2N}$ is a uniformly bounded real system on $\wt \Og$  satisfying a one-sided $4s$-sparse Riesz inequality \eqref{riesz} with constant $2K$.

Third, for another probability measure $\nu$ on $\Og$, we also define
\[
\widetilde{\nu}(A):=\frac12\Bigl(\nu\bigl(\{x\in\Og\colon (0,x)\in A\}\bigr)+\nu\bigl(\{x\in\Og\colon  (1,x)\in A\}\bigr)\Bigr),\quad A\subset \wt \Og.
\]
Then $\wt\nu$ is a probability measure on $\wt\Og$ satisfying
\beq\label{7-6-0}
\|f\|_{L_q(\nu)}\le  \sqrt{2} \|\widetilde f\|_{L_q(\widetilde{\nu})},\  \ \forall\, f\in B(\Og).
\enq
This implies that for  $\al:=p$ and  $\be:=\f {pp_0}2$,
\beq\label{7-6}
\gamma_{\al,
\beta}\bigl({\Sigma_s^p(\mathcal D_N)},\|\cdot\|_{L_q(\nu)}\bigr)
\le \sqrt{2}\,
\gamma_{\al,\beta}\Bigl(\big(\Sigma_s^p(\mathcal D_N)\big)^\sim,\|\cdot\|_{L_q(\widetilde{\nu})}\Bigr),
\enq
where
\[
\bigl(\Sigma_s^p(\mathcal D_N)\bigr)^\sim
:=\Bigl\{\widetilde f\colon    f\in \Sigma_s^p(\mathcal D_N)\Bigr\}\subset B(\wt \Og; \RR).
\]
However,
if $f=\sum\limits_{j\in J}(a_j+ib_j)\varphi_j \in \Sigma_s (\mathcal D_N)$ with $|J|\le s$, then using  \eqref{7-3}, we have
\[
\widetilde f=\sum_{j\in J}a_j\widetilde{\varphi_j}+\sum_{j\in J}b_j(i\varphi_j)^\sim\in \Sigma_{2s} (\wt{\mathcal D}_{2N}),
\]
which  combined  with \eqref{7-4} implies that
\beq\label{7-8}
\bigl(\Sigma_s^p(\mathcal D_N)\bigr)^\sim
\subset \Sigma_{2s}^p(\widetilde{\mathcal D}_{2N}).
\enq
Thus, using \eqref{7-6}, we obtain
\beq \label{7-10a}
\gamma_{\alpha,\beta}\bigl(\Sigma_s^p(\mathcal D_N),\|\cdot\|_{L_q(\nu)}\bigr)
\le \sqrt {2}\,
\gamma_{\alpha,\beta}\bigl(\Sigma_{2s}^p(\widetilde{\mathcal D}_{2N}),\|\cdot\|_{L_q(\widetilde{\nu})}\bigr).
\enq
Using  \eqref{7-6-0} and \eqref{7-8}, we also have
\beq \label{7-10b}\sup_{f\in \Sigma_s^p(\mathcal{D}_N)}
\|f\|_{L_q(\nu)}
\le \sqrt {2}
\sup_{g\in \Sigma_{2s}^p(\widetilde{\mathcal{D}}_{2N})}
\|g\|_{L_q(\widetilde{\nu})}.
\enq

Now combining the above observations, using \eqref{7-10a} and \eqref{7-10b},  we conclude that  the complex case of Lemma \ref{lem-univ-gamma-est} follows by applying the real-valued result  to the uniformly  bounded $4s$-sparse  real Riesz system $\widetilde{\mathcal D}_{2N}$ and its corresponding class  $\Sigma_{2s}( \wt {\mathcal D}_{2N})$. This completes the reduction.

For the remainder of the proof, we  assume without loss of generality that   $\mathcal{D}_N$ is  a uniformly  bounded, real-valued $2s$-sparse (rather than $4s$-sparse)  Riesz system, and that
$\Sigma_s(\mathcal{D}_N)$ is restricted to real coefficients.\\

\noindent \textbf{Step 2.} {In this step we prove} that for any $1\le p\le 2$, there exists a  constant $C=C(p)>0$ such that for  any $3\le q <\infty$,
\beq \label{7-14}
e_n(\Sigma^p_{s}(\mathcal{D}_N), \|\cdot\|_{L_q(\nu)})
\le
C\Bl(\f{Ks\cdot q
\log N}{2^n} \Br)^{\f 1 {2\theta}}, \quad \forall\, n\in \NN_0,
\enq
where {$\theta:=\frac{\frac{1}{2}-\frac{1}{q}}{\frac{1}{p}-\frac{1}{q}}.$}\\

First, we prove that there exists a universal constant $C>0$ such that   for any $2<q<\infty$,
\beq \label{7-14b}
e_n(\Sigma^2_{2s}(\mathcal{D}_N), \|\cdot\|_{L_q(\nu)})
\le
C\Bl(\f{Ks\cdot q
\log N}{2^n} \Br)^{\f 1 {2}}, \quad \forall\, n\in \NN_0.
\enq
Note that this bound  is slightly stronger than \eqref{7-14} for $p=2$.

The proof of \eqref{7-14b}   relies on the following known result, which is a direct consequence of Theorems~7.4.3 and~8.6.6, and Remark~8.6.10 (see also Theorem~9.2.1) in \cite{TemBook}; see also \cite{Tem13}.

\begin{lemma}\label{lem-ent-est}
There exists a universal  constant $C > 0$ such that for any $2\le q <\infty$,  any  measure $\nu$  on $\Og$, and any system of  functions
$\mathcal{D}_N = \{\varphi_1,\ldots,\varphi_N\} \subset B(\Omega; \mathbb{R})$ satisfying
\[
\max_{1\le j\le N} \|\varphi_j\|_{L_q(\nu)} \le 1,
\]
we have
\[
e_n\bigl(\mathcal{A}_1(\mathcal{D}_N), \|\cdot\|_{L_q(\nu)}\bigr)
\le C \sqrt{ \f { q\log N}{2^n}},
\qquad \forall\, 0\le n \le \log_2 N,
\]
where $\mathcal{A}_1(\mathcal{D}_N)$ denotes the absolute convex hull of $\mathcal{D}_N$, defined as
\[
\mathcal{A}_1(\mathcal{D}_N) := \biggl\{
\sum_{j=1}^N a_j \varphi_j
\colon  a_j\in\RR,\    \sum_{j=1}^N |a_j| \le 1
\biggr\}.
\]
\end{lemma}

By \eqref{riesz}, we get  $\Sigma^2_{2s}(\mathcal{D}_N)\subset \sqrt{2Ks}\cdot \mathcal{A}_1(\mathcal{D}_N)$. Furthermore,
\[
\|\varphi_j\|_{L_q(\nu)}\le \sup_{x\in \Og}|\vi_j(x)|\le 1,\  \ \forall\, 1\le j\le N.
\]
Thus, using Lemma~\ref{lem-ent-est}, we obtain \eqref{7-14b}
 for $0\le n\le  k_0:= \lceil \log_2 N\rceil$. For $n> k_0$, \eqref{7-14b} can be deduced  by  applying  \eqref{1-1-0} and \eqref{7-14b} for the already proven case $n=k_0$:
\eq{
e_n\bigl(\Sigma^2_{2s}(\mathcal{D}_N), \|\cdot\|_{L_q(\nu)}\bigr)
&\le \f{3\cdot 2^{2^{k_0}/N}}{ 2^{2^n/N}}e_{k_0}\bigl(\Sigma^2_{2s}(\mathcal{D}_N), \|\cdot\|_{L_q(\nu)}\bigr)
\le C \sqrt{\f {q  Ks\log N} N} 2^{-2^n/N}\\
&\le
C \sqrt{\f {q Ks\log N} { 2^{n}}} \sup_{t\ge 1/2} 2^{-t} \sqrt{t} \le C \sqrt{\f {q Ks\log N} { 2^{n}}}.
}

Next,  combining  estimate \eqref{7-14b} with Lemma \ref{lem-ent-trans} below, and taking into account that \eqref{7-14b} holds for an arbitrary  probability measure $\nu$ on $\Og$, we  deduce estimate  \eqref{7-14}  for $1\le p<2$ and $3\le q<\infty$.

\begin{lemma}\label{lem-ent-trans}
Let $\mathcal{D}_N=\{\varphi_1, \ldots, \varphi_N\}\subset B(\Omega; \mathbb{R})$ be a system  of bounded, real-valued functions on $\Omega$.
Assume that  for some  $q\in[3, \infty)$ and integer $s\in [0, N]$,  there exists a constant $B>0$ such that
\beq \label{7-14c}
e_n\bigl(\Sigma^2_{2s}(\mathcal{D}_N), \|\cdot\|_{L_q(\mu)}\bigr)+   e_n\bigl(\Sigma^2_{2s}(\mathcal{D}_N), \|\cdot\|_{L_q(\nu)}\bigr)
\le B2^{-n/2}, \quad \forall\, n\ge 0.
\enq
Then for any  $p\in[1, 2)$, there exists  a constant $C(p)>0$ such that
\beq\label{7-15c}
e_n\bigl(\Sigma^p_s(\mathcal{D}_N), \|\cdot\|_{L_q(\nu)}\bigr)
\le C(p)(B 2^{-n/2})^{1/\theta},\  \ \forall\, n\in\NN, \  \ \text{where}\  \ \theta:=\tfrac{\frac{1}{2}-\frac{1}{q}}{\frac{1}{p}-\frac{1}{q}}.
\enq
\end{lemma}

Lemma \ref{lem-ent-trans} for the case $\mu=\nu$ follows directly from Lemma 3.1 in \cite{DTM1}. In the general case, we use the following inequality:
\begin{equation}\label{7-16c}
e_{n+1}\bigl(\Sigma_s^p(\mathcal{D}_N), L_q(\nu)\bigr) \le 2 e_n\bigl(\Sigma_s^p(\mathcal{D}_N), L_2(\mu)\bigr) \cdot e_n\bigl(\Sigma_{2s}^2(\mathcal{D}_N), L_q(\nu)\bigr).
\end{equation}
Using (3.3) of \cite{DTM1} and the estimate \eqref{7-14c} for $e_n(\Sigma_{2s}^2(\mathcal{D}_N), \|\cdot\|_{L_q(\mu)})$, we obtain
\[
e_n\bigl(\Sigma_s^p(\mathcal{D}_N), \|\cdot\|_{L_2(\mu)}\bigr) \le C_1(p) (B2^{-n/2})^{\frac{1-\ta}{\ta}},
\]
which,  combined with \eqref{7-16c} and \eqref{7-14c}, yields the desired estimate \eqref{7-15c}.\\

\noindent \textbf{Step 3.} {In this step we prove} that for any $1\le p\le 2$ and $n\ge \log_2 \bigl(s\log_2\tfrac{eN}{s}\bigr)$, we have
\begin{equation}\label{2-ent-est-2}
e_{n+1}\bigl(\Sigma^p_{s}(\mathcal{D}_N), \|\cdot\|_{\infty}\bigr)
\le 4(Ks)^{1/p} 2^{-2^n/s},
\end{equation}
where we recall that $\|f\|_\infty=\sup\limits_{x\in\Og}|f(x)|$ for each $f\in B(\Og)$.\\

For each  $f\in \Sigma_{s}(\mathcal{D}_N)$, we have
\[
\|f\|_\infty
\le \sqrt{Ks}\|f\|_{L_2(\mu)}
\le \sqrt{Ks}\|f\|_{L_p(\mu)}^{\f p2}\|f\|_\infty^{1-\f p2},
\]
implying
\begin{equation}\label{eq-rad-2}
 \|f\|_{\infty}\le K_0:=(Ks)^{1/p},\  \ \forall\, f\in \Sigma^p_{s}(\mathcal{D}_N).
\end{equation}

For each $J\subset\{1, \ldots, N\}$ with  $|J|=s$, define  $V_J:=  {\rm span}\{\varphi_j\colon j\in J\}$,
\[
V_{J,\mu}^p:=\{f\in V_J\colon \|f\|_{L_p(\mu)}\le 1\}\quad \text{and}\quad
V_{J}^\infty:=\{f\in V_J\colon \|f\|_{\infty}\le 1\}.
\]
Then  \eqref{eq-rad-2} implies that
$ V_{J,\mu}^p \subset K_0\cdot V_J^\infty$.
It follows   that  (see \cite{TemBook}*{Theorem~7.2.1 and Corollary~7.2.2}) for any $u>0$,
\[
\mathcal{N}_u(V_{J,\mu}^p, \|\cdot\|_\infty )\le \mathcal{N}_u\bigl(K_0\cdot  V_J^\infty,  \|\cdot\|_{\infty}\bigr)
\le  \Bigl(1+\tfrac{2K_0}{u}\Bigr)^s.
\]
Since
\[ \Sigma_s^p(\mathcal{D}_N)=\bigcup_{|J|=s} V_{J,\mu}^p,\]
we deduce that for any $u>0$,
\begin{equation*}
\mathcal{N}_u(\Sigma_s^p(\mathcal{D}_N), \|\cdot\|_\infty )\le \tfrac{N!}{s!(N-s)!}\Bigl(1+\tfrac{2K_0}{u}\Bigr)^s
\le \bigr(\tfrac{eN}{s}\bigl)^s\Bigl(1+\tfrac{2K_0}{u}\Bigr)^s.
\end{equation*}
 Setting $u=u_n:= 4K_02^{-2^n/s}$  and assuming that  $n\ge \log_2 (s\log_2\frac{eN}{s})\ge \log_2s$,
 we obtain
\[
\mathcal{N}_{u_n}(\Sigma_s^p(\mathcal{D}_N), \|\cdot\|_{\infty})
\le 2^{s\log_2(\frac{eN}{s})}\Bigl(1+\tfrac{2^{\frac{2^n}{s}}}{2}\Bigr)^s\le 2^{2^{n+1}},
\]
which implies \eqref{2-ent-est-2}. \\

\noindent \textbf{Step 4 (final step).} {In this step we prove}  estimates \eqref{7-3a} and \eqref{7-3b}.\\

For simplicity, set $F:=\Sigma^p_s(\mathcal{D}_N)$.
Let  $p_0 = \frac{2}{1+\delta}\in [3/2, 2)$ for some constant  $\delta\in (0, 1/3)$ to be specified later. Then  $q:= \frac{pp_0}{2-p_0}=\f p\da\ge 3$, and
\[
\theta:=\tfrac{\frac{1}{2}-\frac{1}{q}}{\frac{1}{p}-\frac{1}{q}} = \frac{\frac{p}{2}-\delta}{1-\delta}
\in \bigl[\tfrac{1}{4}, \tfrac{p}{2}\bigr].
\]
By \eqref{1-2-2025}, we have
\begin{align}\Bl[\gamma_{p,\f{p_0p} 2}(F,\|\cdot\|_{L_q(\nu)})\Br]^{\f {p_0p}2}&\le  C_1(p)\cdot S,\label{7-12}
\end{align}
where
$$S:= \sum_{n=0}^\infty \Bl(2^{n/p} e_n(F,\|\cdot\|_{L_q( \nu)})\Br)^{\f {pp_0}2}.$$
We  split  the sum $S$ into three parts:
\[ S=\sum_{n=0}^{n_0}+ \sum_{n=n_0+1}^{n_1+{3}}+\sum_{n=n_1+{4}}^{\infty}:= S_1+S_2+S_3,\]
where
\[
n_0:= \bigl\lceil\log_2(2\log_2(eN))\bigr\rceil+2\text{ and } n_1:= \bigl\lceil\log_2 (2s\log_2(eN))\bigr\rceil+2.
\]
For the first sum, we apply \eqref{eq-rad-2} to obtain
\eq{S_1&\le  \sum_{n=0}^{n_0} \Bl(2^{n/p} e_n(F,\|\cdot\|_{\infty})\Br)^{\f {pp_0}2}\le  (Ks)^{\f {p_0}2} \sum_{n=0}^{n_0} 2^{\f {n p_0} 2}\le  C_2(p) (Ks\ \log N)^{\f {p_0}2}.
}
For the second sum, we apply \eqref{7-14} to obtain
\eq{S_2&\le C_3(p) \big(Ks\da^{-1}\big)^{\f {pp_0}{4\ta}}\sum_{n=n_0+1}^{n_1+{3}} 2^{\f {np_0}2} \Bl( \f { \log N}{2^n} \Br) ^{\f {p}{2\ta} \cdot \f {p_0} 2}.
}
Since $2^n\ge \log N$ for $n> n_0$ and since $2\ta\le p$, it follows that
\eq{
S_2\le C_3(p) \big(Ks\da^{-1}\big)^{\f {pp_0}{4\ta}}\sum_{n=n_0+1}^{n_1+{3}} 2^{\f {np_0}2} \Bl( \f { \log N}{2^n} \Br) ^{ \f {p_0} 2}\le   C_4(p) \big(Ks\da^{-1} \big)^{\f {pp_0}{4\ta}} (\log N)^{\f {p_0}2}\log s.
}
For the sum $S_3$, we apply \eqref{2-ent-est-2} to obtain
\eq{S_3&\le \sum_{n=n_1+{4}}^\infty \Bl(2^{\f np} e_n(F,\|\cdot\|_\infty)\Br)^{\f {pp_0}2}\le C_5(p) (Ks)^{\f {p_0}2} \sum_{n=n_1+3}^\infty2^{-(\f{2^n}s-\f n p) \f {pp_0}2}.
}
For $n\ge n_1+3$, we have
\eq{
\Bl(\frac{2^{n}}{s}-\f np\Br)-\Bl(\frac{2^{n_1}}{s}-\f {n_1}p\Br)
&=\frac{2^{n_1}}{s} (2^{n-n_1} - 1) - \f {n-n_1}p\\
&\ge (2^{n-n_1} - 1) - \f {n-n_1} p \ge  n-n_1.
}
It follows that
\eq{
S_3 &\le C_5(p) (Ks)^{\f {p_0}2}  2^{-(\f{2^{n_1} }s-\f {n_1} p) \f {pp_0}2}\sum_{k=3}^\infty 2^{-k/2}
\le C_6(p) (Ks)^{\f {p_0}2},
}
where the last step uses the fact that $\f{2^{n_1} }s\ge n_1$.
Putting the above estimates together, and recalling $\f p {2\ta}\ge 1$, we obtain
\[
S \le C_7(p)  (Ks\da^{-1})^{\f {pp_0}{4\ta}} (\log N)^{p_0/2} \log s.
\]
Substituting into \eqref{7-12}, we then obtain
\[
\Bl[\gamma_{p,\f{p_0p} 2}(F,\|\cdot\|_{L_q(\nu)})\Br]^{\f {p_0p}2}
\le C_8(p)  (Ks\da^{-1})^{\f {pp_0}{4\ta}} (\log N)^{p_0/2} \log s.
\]
This combined with  \eqref{eq-rad-2} implies
\begin{align*}
&\sup_{f\in F}\|f\|_{\infty}^{p-\frac{pp_0}{2}}\ \cdot  \
\Bl[\gamma_{p,\frac{pp_0}{2}}(F,\|\cdot\|_{L_q(\nu)})\Br]^{\frac{pp_0}{2}}\le C_8(p) (Ks)^{\f p {2\ta}} \da^{-\f {pp_0}{4\ta}} (\log N)^{\f {p_0}2} \log s\\
&=
C_8(p) \Bigl((Ks)^{1+\delta\cdot\frac{2-p\delta}{p-2\delta}}(\log s)^{1+\delta}\cdot \delta^{-1 - \delta\cdot \frac{2-p}{p-2\delta}} \log N \Bigr)^{\f {p_0}2}
\\
&\le
C_9(p) \Bigl((Ks)^{1+6\delta}(\log s)^{1+\delta}\cdot \delta^{-1} \log N \Bigr)^{p_0/2}.
\end{align*}
Finally, setting  $\delta:= \frac{1}{3\log Ks}$, and recalling $p_0=\f 2 {1+\da}$,  we obtain
\eq{&\sup_{f\in F}\|f\|_{\infty}^{p-\frac{pp_0}{2}}\ \cdot  \
\Bl[\gamma_{p,\frac{pp_0}{2}}(F,\|\cdot\|_{L_q(\nu)})\Br]^{\frac{pp_0}{2}}\le
C_{10}(p) \Bigl(Ks \  \log s \  \log(Ks)\ \log N \Bigr)^{p_0/2},}
and \eq{
&\sup_{f\in F}\|f\|_{\infty}^{p-\frac{pp_0}{2}}
\cdot \sup_{g\in F}\|g\|_{L_q(\nu)}^{\frac{pp_0}{2}} \le \sup_{f\in F}\|f\|_{\infty}^{p}\le Ks\le C (Ks)^{p_0/2}.
}
This completes the proof of the lemma.
\end{proof}

\subsection{Proof of Theorem \ref{Th-univ}}\label{subsection-1-2}
For simplicity, we again set $F:=\Sigma_s^p(\mathcal{D}_N)$.
Fix the random points $\xi^1, \cdots, \xi^m$, and let $\nu$ be the uniform distribution on the finite set  $\{\xi^1,\cdots, \xi^m\}\subset \Og$. Applying Lemma \ref{lem-univ-gamma-est} with $\nu$, we  find a number $p_0=p_0(s, K) \in [\f 32, 2)$ such that the following two inequalities hold  for $q:=\f {pp_0} {2-p_0} \ge 3$:
\begin{equation}\label{7-18}
\sup_{f\in F} \|f\|_\infty^{p-\f {pp_0}2} \  \cdot\  \Bl[ \ga_{p, \f {pp_0}2}(F, \|\cdot\|_{L_q(\pmb\xi)})\Br]^{\f {pp_0}2} \le C(p) \bigl( Ks \ \log s \ \log (Ks) \  \log N\bigr)^{\f {p_0}2},
\end{equation}
\begin{equation}\label{7-19}
\sup_{f\in F} \|f\|_\infty^{p-\f {pp_0}2} \  \cdot \ \sup_{g\in F} \|g\|_{L_q(\pmb\xi)}^{\f {pp_0}2}\le C(p) (Ks)^{\f {p_0}2}.
\end{equation}

We will apply  Lemma \ref{lem-log}  to the function class $T_{p/2}(F)$ rather than to $F$. To this end, we first note  that
\[ \sup_{g\in T_{p/2}(F)}\|g\|_{L_2(\mu)}=\sup_{f\in F}\|f\|_{L_p(\mu)}^{\f p2}\le 1,\]
and
\[ \Er_p(F,\pmb{\xi})= \Er_2(T_{p/2}(F),\pmb{\xi}).\]
Second, invoking  Lemma \ref{lem-6-2} with $\ld=\f p2$,  we obtain  that for  $p_1:=\f {2p_0}{2-p_0}$ and $q:=\f {pp_0} {2-p_0}$,
\eq{&\sup_{g\in T_{p/2}(F)} \|g\|_\infty^{2-p_0}\cdot\Bl[ \ga_{2, p_0} (T_{p/2}(F), \|\cdot\|_{L_{p_1}(\pmb\xi)})\Br]^{p_0}
\le  C_1(p)   \sup_{f\in F} \|f\|_\infty^{p-\f {pp_0}2 }\cdot\Bl[ \ga_{p, \f {pp_0}2} (F, \|\cdot\|_{L_{q}(\pmb\xi)})\Br]^{\f {pp_0}2 },
}
which, using  \eqref{7-18}, is estimated above by
\[ \le C_2(p) \bigl( Ks \ \log s \ \log (Ks) \  \log N\bigr)^{\f {p_0}2}.\]
Finally,  using \eqref{7-19}, we obtain
\eq{\sup_{g\in T_{p/2}(F)} \|g\|_\infty^{2-p_0}\  \cdot \sup_{g\in T_{p/2}(F)}  \|g\|_{L_{p_1}(\pmb\xi)}^{p_0}&= \sup_{f\in F} \|f\|_\infty^{p-\f {pp_0}2} \  \cdot \ \sup_{g\in F} \|g\|_{L_q(\pmb\xi)}^{\f {pp_0}2}\le C(p) (Ks)^{\f {p_0}2}.
}

The above estimates show that  Lemma~\ref{lem-log} is applicable  to the function class $T_{p/2}(F)$ with the   parameters $p=2$, $H=Ks$, $\al=0$ and $m_0\sim \log s\, \log (Ks)\, \log N$.
Recalling   $1<p_0<2$,  we observe that
\[ H m_0 (\log\log H)^{2/p_0'} \le C Ks   \log s\, \log (Ks)\, \log N\,  \log\log (Ks). \]
 Taking Remark \ref{rem-probab-est} into account, we then obtain that for any  $\lambda \ge e$ and $\varepsilon \in (0,\tfrac12]$,
\eq{\mathbb{P}\bigl[\Er_p(F,\pmb{\xi})>\varepsilon\bigr]=
\mathbb{P}\Bl[\Er_2(T_{p/2}(F),\pmb{\xi})>\varepsilon\Br]
\le 2 \exp\bigl(-c(p)\lambda(\log\lambda)^{1/p_0}\bigr)\le 2 \exp\bigl(-c(p)\lambda\sqrt{\log\lambda}\bigr)
}
provided that
\[
m \ge (\lambda\,\varepsilon^{-1})^2(\log \varepsilon^{-1})\,
Ks\, \log N\, \log s \, \log(Ks)\, \log\log(Ks).
\]
This proves  Theorem \ref{Th-univ}.
\qed

\section{van Handel's approach: the contraction principle and \texorpdfstring{$K$}{K}-functionals}\label{sec6}

Talagrand's generic chaining provides sharp estimates for the expected supremum of random processes in terms of chaining functionals (see, for instance, Theorem~\ref{thm-1-3}). In concrete applications, however, these functionals are often difficult to estimate directly. Van Handel~\cites{VH,VH1} developed a powerful approach to bounding them by combining a contraction principle with interpolation techniques based on $K$-functionals. In this section, we recall the main ingredients of this method and derive a finite-dimensional version of the contraction principle that is particularly convenient for our later applications.

We begin with van Handel's contraction principle, which bounds chaining functionals in terms of suitable local entropy estimates. Throughout this section, unless stated otherwise, $\mathbf{T}=(\bT,\varrho)$ denotes a metric space.

\begin{lemma}[{\cite{VH}*{Theorem 3.1}}]\label{thm-1-2-1b}
Assume that there exist a constant $a\ge 0$ and a sequence of nonnegative functions
\[s_n:\bT\to [0,\infty),\quad n\in \NN_0,\]
such that for every set $A\subset \bT$ and every $n\in \NN_0$,
\begin{equation}\label{1-2-1}
e_n (A,\varrho)\le a \cdot \diam(A,\varrho) +\sup_{x\in A} s_n(x).
\end{equation}
Then, for any $\al>0$ and $1\le \be<\infty$,
\begin{align*}
&\gamma_{\alpha, \be}(\bT,\varrho) \le C(\al)\left[  a\cdot \gamma_{\alpha, \be}(\bT,\varrho)+\sup _{x \in \bT}\Bl(  \sum_{n=0}^\infty\left(2^{n / \alpha} s_n(x)\right)^\be\Br)^{1 / \be}\right].
\end{align*} In particular, if $0<a\le \f 1{2C(\al)}$, then
\[
\gamma_{\alpha, \be}(\bT,\varrho) \le 2C(\al)\ \sup _{x \in \bT}\Bl(  \sum_{n=0}^\infty\left(2^{n / \alpha} s_n(x)\right)^\be\Br)^{1 / \be}.\]
\end{lemma}

We will mostly work in a finite-dimensional setting, where $\bT$ is a subset of a finite-dimensional normed space $(X,\|\cdot\|)$ and $\varrho$ is the metric induced by $\|\cdot\|$. Combining Lemma~\ref{thm-1-2-1b} with Lemma~\ref{lem-1-2a}, we obtain the following finite-dimensional version of Lemma~\ref{thm-1-2-1b}, involving only finitely many functionals $s_n\colon \bT\to[0,\infty)$.

\begin{lemma}\label{cor-1-7} Let $\bT$ be a subset of an  $m$-dimensional real  normed space $X=(X,\|\cdot\|)$.   Assume that there exist a constant $a\ge 0$  and  a finite sequence of  nonnegative functions
\[
s_n\colon \bT\to [0,\infty),\  \ n=0,1,\dots, \lceil \log_2 m \rceil,
\]
such that for every subset $A\subset \bT$ and every integer $n\in [0, \log_2 m]$,
\beq\label{1-2-6}
e_{n}(A,\|\cdot\|) \le a \cdot\diam(A,\|\cdot\|)+\sup _{x \in A} s_n(x).
\enq
Then,  for any $\al>0$ and  $1\le \be<\infty$,
\[
\gamma_{\alpha, \be}(\bT,\|\cdot\|) \le C(\alpha)\left[ a\cdot  \gamma_{\alpha, \be}(\bT,\|\cdot\|)+\sup _{x \in \bT}\Bl( \sum_{n=0}^{\lceil \log_2 m \rceil}\big(2^{n / \alpha} s_n(x)\big)^\be\Br)^{1 / \be}\right].
\]
In particular, if $0<a\le \f 1 {2C(\al)}$, then
\[\ga_{\al,\be}(\bT,\|\cdot\|)\le 2C(\al)\ \sup _{x \in \bT} \Bl(\sum_{n=0}^{\lceil \log_2 m \rceil}\big(2^{n / \alpha} s_n(x)\big)^\be\Br)^{1 / \be}.\]
\end{lemma}

\bp Let $n_*=\lceil \log_2 m \rceil$.  Using  \eqref{1-1-0}, we obtain that for any $A\subset \bT$, and any integer  $n>n_*$,
\begin{align*}
e_n(A, \|\cdot\|)&\le 3\cdot  2^{-2^n/m}  2^{2^{n_*}/m} e_{n_*}(A, \|\cdot\|)\le 12\cdot  2^{-2^n/m}  e_{n_*}(A, \|\cdot\|).
\end{align*}
Applying \eqref{1-2-6} with $n_*$ in place of $n$ then yields
\beq e_n(A, \|\cdot\|) \le 12\cdot a \cdot\diam(A,\|\cdot\|)+12 \cdot 2^{-2^n/m} \sup _{x \in A} s_{n_\ast} (x),\  \ \forall\, n>n_\ast.\label{3-3a} \enq
Next, define
\[s_n(x):=12\cdot  2^{-2^n/m} s_{n_*} (x),\   \  n=n_*+1, n_*+2,\ldots.\]
Combining \eqref{3-3a} with \eqref{1-2-6}, we get
\[ e_{n}(A,\|\cdot\|) \le 12a \cdot\diam(A,\|\cdot\|)+\sup _{x \in A} s_n(x),\  \ \forall\, n\in \NN_0.\]
It then follows by Lemma \ref{thm-1-2-1b}  that
\begin{align*}\ga_{\al,\be}(\bT, \|\cdot\|) &\le C(\al)\left[ 12a\cdot  \gamma_{\alpha, \be}(\bT,\|\cdot\|) +  \Bl(\sup _{x \in \bT} \sum_{n=0}^{n_*}\big(2^{n / \alpha} s_n(x)\big)^\be\Br)^{1 / \be}+  \right.\\
&\left.\   \hspace{2cm} +\sup _{x \in \bT} s_{n_*} (x)\cdot \Bl(\sum_{n=n_*+1}^\infty\big(2^{n / \alpha} \cdot 2^{-2^n/m} \big)^\be\Br)^{1 / \be}\right].
\end{align*}
However, it is readily seen (see, for instance,~\cite{Kos}*{Lemma 2.11}) that
\[
\Bigl(\sum_{n=n_*+1}^\infty\bigl(2^{n/\alpha}\cdot 2^{-2^n/m}\bigr)^\beta\Bigr)^{1/\beta}
\le C_1(\alpha)\,2^{n_*/\alpha}.
\]
Thus,
\[ \ga_{\al,\be}(\bT, \|\cdot\|)\le C_2(\al)\left[ a\cdot  \gamma_{\alpha, \be}(\bT,\|\cdot\|) + \Bl(\sup _{x \in \bT} \sum_{n=0}^{n_*}\big(2^{n / \alpha} s_n(x)\big)^\be\Br)^{1 / \be}\right],\]
which completes the proof.
\ep

In applications of Lemma~\ref{thm-1-2-1b}, the main difficulty is to construct, for a sufficiently small constant $a>0$, a sequence of functionals $s_n\colon\bT\to[0,\infty)$ such that the entropy estimate~\eqref{1-2-1} holds for every subset $A\subset\bT$, while the quantity
\[
\sup_{x\in\bT}\sum_{n\ge0}\bigl(2^{n/\alpha}s_n(x)\bigr)^\beta
\]
remains under control. Van Handel~\cite{VH} proposed an elegant way to do this by means of interpolation $K$-functionals. We briefly recall this approach below, in a slightly generalized form. Since only the case $\beta=1$ will be used later, we restrict ourselves to this case throughout the rest of the section.

\begin{definition}\label{def-3-2}
Fix $\alpha>0$. Let
\[
u_n\colon\bT\times[0,\infty)\to[0,\infty), \quad n\in\NN_0,
\]
be a sequence of nonnegative functionals such that for some constant $u^\ast>0$,
\[
0\le u_n(x,r)\le u_{n+1}(x,r)\le u^\ast,
\quad \forall\, x\in\bT,\ \forall\, r\ge0,\ \forall\, n\in\NN_0.
\]
For each $n\in\NN_0$, define the associated $K$-functional by
\[
K_n(x,t):=\inf_{r\ge0}\bigl[tr+u_n(x,r)\bigr],
\quad x\in\bT,\ t\ge0.
\]
Given a constant $a>0$, let $r_n^a\colon\bT\to[0,\infty)$, $n\in\NN_0$, be any sequence of functions satisfying
\[
t_n r_n^a(x)+u_n(x,r_n^a(x))
\le
K_n(x,t_n)+2^{-n}u^\ast,
\quad \forall\, x\in\bT,\ \forall\, n\in\NN_0,
\]
where
$
t_n:=a\,2^{n/\alpha}$.

\end{definition}

By definition, it is clear that
\beq\label{3-3b}
0\le  K_n(x,t)\le K_{n+1}(x, t)\le  u_{n+1}(x, 0)\le u^\ast,\   \ \forall\, x\in \bT,\  \ \forall\, n\in \mathbb{N}_0.
\enq
We will adhere to  the notation and assumptions introduced in Definition~\ref{def-3-2} throughout the remainder of this section.

The next two lemmas record the basic properties of the quantities introduced in Definition~\ref{def-3-2}. Since we will use them repeatedly, we include their  proofs for completeness.

\begin{lemma}\label{lem-1-7-1} Let $a>0$, $x\in \bT$ and $n_0\in \NN$.  Define
\beq \label{3-8-2}\Delta_{n,n_0}(x):=K_{n+ n_0}(x, t_{n+ n_0})-K_{n}(x, t_{n})+\f{ u^\ast}{2^{n+n_0}},\  \ n\in \NN_0,\enq
where $t_n=a 2^{n/\al}$.
 Then \beq \label{ineq:delta-lower} \ \Delta_{n,n_0}(x)\ge (1-2^{-\f1\al})a 2^{(n+n_0)/\al} r_{n+n_0}^a(x)\ge 0,\  \ \ \forall\, n\in\NN_0, \enq and
 \beq \label{ineq:delta-sum} \sum_{n=0}^\infty \Delta_{n, n_0}(x)\le  (n_0+1) u^\ast.\enq
 In particular,
\beq \label{3-14b} \sum_{n=1}^\infty   2^{n/\al} \cdot r_n^a(x) \le\f{ C(\al) u^\ast}a.\enq
\end{lemma}

\bp
For brevity, write  $r_k^a:=r_k^a(x)$ for $k\in\NN_0$. By definition,  for each integer $n\ge 0$, we have
\begin{align*}
\Delta_{n, n_0}(x)&\ge t_{n+n_0} r_{n+n_0}^a + u_{n+n_0}(x, r_{n+n_0}^a) -\inf_{r\ge 0} \Bl[ t_{n} r+ u_{n}(x, r)\Br]\notag\\
&\ge (t_{n+n_0}-t_{n}) r_{n+n_0}^a + u_{n+n_0}(x, r_{n+n_0}^a) -u_{n}(x, r_{n+n_0}^a)\\
&\ge (t_{n+n_0}-t_{n}) r_{n+n_0}^a
\ge (1-2^{-\f1\al}) t_{n+n_0} r_{n+n_0}^a,
\end{align*}
which proves \eqref{ineq:delta-lower}.
To establish \eqref{ineq:delta-sum}, we use \eqref{3-3b} to obtain
\eq{\sum_{n=0}^\infty \Delta_{n, n_0}(x)\le u^\ast + \sup_{n\in\NN_0}\sum_{j=1}^{n_0} K_{n+j} (x, t_{n+j})\le (n_0+1) u^\ast.
}
Finally, applying \eqref{ineq:delta-lower} and \eqref{ineq:delta-sum} with $n_0=1$, we obtain
\begin{align*}
a\sum_{n=1}^\infty 2^{n/\al} r_n^a(x)&=\sum_{n=0}^\infty t_{n+1} r_{n+1}^a(x)\le C(\al) \sum_{n=0}^\infty \Delta_{n,1}(x)\le C(\al) u^\ast,
\end{align*}
which proves \eqref{3-14b}.
\ep

\begin{lemma}\label{lem-1-7-2}
Under the notation of Lemma~\ref{lem-1-7-1}, for every $r\ge0$,  one has
\beq \label{ineq:u-diff} u_{n+ n_0}(x, r_{n+ n_0}^a(x)) -u_{n}(x, r)\le \Delta_{n,n_0}(x) + a 2^{n/\al}  r.\enq
\end{lemma}
\bp By Definition \ref{def-3-2},  for any $r\ge 0$,
we have
\begin{align*} u_{n+ n_0} (x, r_{n+ n_0}^a( x))&\le t_{n+ n_0} r_{n+n_0}^a( x) +u_{n+ n_0} (x, r_{n+ n_0}^a( x)) \\
&\le K_{n+ n_0} (x, t_{n+ n_0} ) +\f {u^\ast} {2^{n+n_0}} =K_{n}(x, t_{n}) +\Delta_{n,n_0}(x)\\
&\le \Delta_{n,n_0}(x) + t_{n} r + u_{n}(x, r).\end{align*}
Rearranging terms yields the desired inequality \eqref{ineq:u-diff}.
\ep

We now combine Lemma~\ref{cor-1-7} with Lemmas~\ref{lem-1-7-1} and~\ref{lem-1-7-2} to derive an estimate for the chaining functional $\gamma_{\alpha,1}$ in terms of the growth of the associated $K$-functionals in a finite-dimensional setting. This estimate will be the main tool in the next section.

\begin{theorem}\label{thm-2-4}
Let $\bT$ be a subset of an $m$-dimensional real normed space $X=(X,\|\cdot\|)$, and let $\al>0$, $b_1,b_2>0$, and $n_0,n_1, \ell\in\NN$ with $n_1\le \lceil \log_2 m\rceil$ be given parameters.
Set
$
a:=\frac{1}{4b_2 C(\alpha)},$
where $C(\alpha)$ is the constant from Lemma~\ref{cor-1-7}. Assume that there exists a finite sequence of nonnegative functions
\[
\widetilde s_k\colon\bT\to[0,\infty),
\quad
k=0,1,\dots,\lceil\log_2 m\rceil,
\]
such that for every subset $A\subset\bT$ and every integer $n \in \NN_0$ satisfying $n+\ell\le
\lceil\log_2 m\rceil$, one has
\begin{align}\label{2-4}
e_{n+\ell}(A,\|\cdot\|)& \le \f 1{4C(\al)}\cdot  \diam(A,\|\cdot\|) +\sup_{x\in A}\wt s_{n}(x)+ b_1 \cdot 2^{n_0/\al}\cdot \sup_{x\in A} r_{n+n_0}^a(x)\\
&+
\f {b_2}{2^{n/\al} }\cdot\sup_{x\in A} \Bl[ u_{n+ n_0}(x, r_{n+ n_0}^a(x)) -u_{n}(x, r_A)\Br],\nonumber
\end{align}
where
\[
r_A=2^{n_0/\al}\cdot \sup\limits_{x\in A} r_{n+n_0}^a(x)+\diam(A, \|\cdot\|).
\]
Then
\beq \label{3-17-d} \ga_{\al, 1}(\bT, \|\cdot\|)\le  C(\al,\ell)  \Bigg[ (b_2n_0+b_1b_2) u^\ast+2^{n_1/\alpha}\diam(\bT,\|\cdot\|)+\sup_{x\in \bT}\sum_{n=n_1}^{\lceil \log_2 m \rceil}  2^{n/\al} \wt s_{n}(x)\Bigg]. \enq
\end{theorem}

In applications of Theorem~\ref{thm-2-4}, the parameters $b_1,b_2,n_0$, and $n_1$ must typically be optimized. By contrast, the parameter $\ell$ is inessential, and the implicit constants are allowed to depend  on it.

\bp We may assume that $n_1+\ell \le \lceil \log_2 m \rceil$, since otherwise, combining \eqref{1-2-2025} with \eqref{1-1-0}, one readily obtains
\[
\gamma_{\alpha,1}(T,\|\cdot\|)\le C(\alpha)\, 2^{(n_1+\ell)/\alpha}\diam(T,\|\cdot\|).
\]
We apply Lemma \ref{cor-1-7}.
For $0\le k\le n_1+\ell$, define $s_k(x):=\diam(\bT,\|\cdot\|)$. For $n_1+\ell< k \le \lceil \log_2 m \rceil +\ell$, define
\begin{align} s_k(x)&:=b_2\cdot 2^{-k/\al} \Delta_{k-\ell,n_0}(x) + \wt s_{k-\ell}(x)+ b_1 \cdot 2^{n_0/\al}\cdot  r_{k+n_0-\ell}^a(x),\label{3-17c}\end{align}
where $\Delta_{k-\ell, n_0}(x)$ is defined in \eqref{3-8-2}.
We claim that for every subset  $A\subset \bT$ and every integer $0\le k\le \lceil \log_2 m\rceil$, the following estimate holds:
\beq\label{2-5}
e_k(A,\|\cdot\|)\le  \f 1{2C(\al)}\cdot \diam(A,\|\cdot\|) +C(\al,\ell)\sup_{x\in A} s_k(x).
\enq

By the definition of $s_k$, the estimate  \eqref{2-5} holds trivially if $0\le k\le n_1+\ell$. Now suppose  $k=n+\ell$ for some integer  $n>n_1$. Invoking Lemma \ref{lem-1-7-2} with $r=r_A$, we obtain
\begin{align*} &\sup_{x\in A} \Bl[ u_{n+ n_0}(x, r_{n+ n_0}^a(x)) -u_{n}(x, r_A)\Br]\\
&\le 2\sup_{x\in A} \Bl[  \Delta_{n,n_0}(x) + a 2^{(n+n_0)/\al} r_{n+n_0}^a(x) \Br] + a 2^{n/\al}\cdot\diam(A, \|\cdot\|)\\
&\le C_1(\al) \sup_{x\in A}   \Delta_{n,n_0}(x)+a 2^{n/\al}\cdot\diam(A, \|\cdot\|),\end{align*}
where the last inequality follows from \eqref{ineq:delta-lower}. Substituting this into the assumption  \eqref{2-4}, we find
\begin{align*}
&e_{k}(A,\|\cdot\|)=e_{n+\ell}(A,\|\cdot\|)\\
 &\le \Bl( \f 1{4C(\al)}+a b_2\Br)\cdot  \diam(A,\|\cdot\|)
 +C_1(\al)\cdot \sup_{x\in A} \left[ \wt s_n(x)+ b_1 2^{n_0/\al} r_{n+n_0}^a(x) +\f{b_2 \Delta_{n, n_0}(x)}{2^{n/\al}}\right] \\
& \le \f 1{2C(\al)}\cdot \diam(A,\|\cdot\|) +C_2(\al,\ell)\cdot \sup_{x\in A} s_k(x).
\end{align*}
This proves  the claim \eqref{2-5}.

Finally,  applying  Lemma  \ref{cor-1-7} yields
\begin{align}
\ga_{\al,1}(\bT,\|\cdot\|)&\le C_3(\al, \ell) \sup_{x\in \bT}
\sum_{k=0}^{\lceil \log_2 m \rceil} 2^{k/\al} s_k(x)\notag\\
&\le C_4(\al,\ell)\ 2^{n_1/\alpha} \diam(\bT,\|\cdot\|)+C_3(\al, \ell) \sup_{x\in \bT}
\sum_{k=n_1+\ell}^{\lceil\log_2 m\rceil} 2^{k/\al} s_k(x).\label{3-20a}\end{align}
Using \eqref{3-17c} and  Lemma \ref{lem-1-7-1},  for any $x\in \bT$, we have
\eq{\sum_{k=n_1+\ell}^{\lceil\log_2 m\rceil}2^{k/\al} s_k(x)&\le C_5(\al,\ell) \sum_{n=n_1}^{\lceil\log_2 m\rceil} \Bl[ b_2 \Delta_{n, n_0}(x)+2^{n/\al} \wt s_n(x) + b_1 2^{(n+n_0)/\al} r_{n+n_0}^a(x) \Br]\\
&\le C_6(\al,\ell)\Bl[ b_2 n_0 u^\ast+ b_1 b_2 u^\ast\Br]+ C_5(\al, \ell) \sum_{n=n_1}^{\lceil\log_2 m\rceil} 2^{n/\al} \wt s_n(x).
}
Substituting this last estimate into \eqref{3-20a} yields the desired estimate \eqref{3-17-d}.
\ep

\section{Proof of  Theorem \ref{thmGammaBound}}\label{sec:proof-thmGammaBound}

In this section,  we prove Theorem~\ref{thmGammaBound}, which is restated below for convenience. Throughout the paper,   $\|\cdot\|_q$ denotes the  norm $\|\cdot\|_{\ell_q^m}$ for any $1\le q\le \infty$.\\

\noindent {\bf Theorem \ref{thmGammaBound}.}
{\it Let $2 \le p < \infty$, $p_0\in(1, p]$, and $p_1:=\f {pp_0}{p-p_0}$.
Then there exists a constant $C = C(p)>0$
such that for every nonempty set $G \subset \mathbb{C}^m$
and  any choice of parameters  $b > 0$, $\alpha_0 \in (0,1)$, $n_0 \in \NN$ and $m_0\in \NN \cap [1, m)$, one has
\begin{align}
\ga_{p,1}(T_p(G),\|\cdot\|_{p'})
&\le C \Biggl[
\Bigl( b n_0 + b^2\cdot (\alpha_0)^{1/p}
+\frac{\log \frac{m}{m_0}}{(2^{n_0}\alpha_0)^{1/p}} \Bigr) \sup_{x \in G} \|x\|_{p}^p\tag{\ref{5-7-0}}
\\
&+
\sup_{x\in G}\|x\|_\infty^{p-p_0}\cdot
\frac{m_0^{p_0/p}\bigl[\diam(G, \|\cdot\|_{p_1}) \bigr]^{p_0}+ \bigl[\ga_{p,p_0}(G, \|\cdot\|_{p_1})\bigr]^{p_0}}{b^{p_0-1}}
\Biggr].\nonumber
\end{align}
}\\

By Remark~\ref{rm-complex}, we only need to  prove Theorem \ref{thmGammaBound} for  the real case $G\subset\RR^m$.
The proof is based on Theorem~\ref{thm-2-4}, applied to the set
\[
\bT:=T_p(G)=\{|g|^p\colon g\in G\},
\]
together with a suitable sequence of functionals $\widetilde s_k:\bT\to[0,\infty)$ constructed from an admissible sequence of partitions of $G$. A central step is to verify a local entropy estimate of the form~\eqref{2-4} for an appropriate increasing sequence of functionals
\[
u_n\colon\bT\times[0,\infty)\to[0,\infty),
\quad n\in\NN_0.
\]
This requires two technical lemmas, Lemmas~\ref{lem-3-3} and~\ref{lem-3-4}, presented in Subsection~\ref{sec-4-1}.  We prove  Lemma~\ref{lem-3-3}  in that  subsection, but  the proof of the more involved  Lemma~\ref{lem-3-4} is deferred to Subsection~\ref{sec-4-3}. Finally, the proof of Theorem~\ref{thmGammaBound} is completed  in Subsection~\ref{sec-4-2} using Lemma~\ref{lem-3-4}.

Throughout this section, we fix $2\le p<\infty$
and assume  $G\subset\mathbb R^m$ is a nonempty set satisfying
\begin{equation}\label{4-5-25}
  \sup_{g\in G}\|g\|_p^p=:u^\ast<\infty.
\end{equation}
Let $\mathbf{T}:=T_p(G)$. For brevity, we also write   $q$ in place of $p_0$.

For any index set   $I\subset\{1,2,\dots,m\}$, let
\[
\mathbb R^I:=\spn\{e_j\colon j\in I\}\subset\mathbb R^m,
\]
and denote by
$
R_I\colon \mathbb R^m\to\mathbb R^m
$
the orthogonal projection onto $\mathbb R^I$, given by
\[
R_Ix=\sum_{j\in I}x_je_j,
\quad x=(x_1,\dots,x_m)\in\mathbb R^m.
\]
Denoting the complement of $I$ by
$
I^c:=\{1,2,\dots,m\}\setminus I,$
we define the threshold set
 for any $\tau\ge0$ and $x\in \RR^m$ as
\[
I^c(x;\tau):=\{j\in I^c\colon x(j)\ge\tau\}.
\]

\subsection{Key lemmas}\label{sec-4-1}

We start with the following elementary inequality.
\begin{lemma}\label{lem-6-4-0-Ta} For any  $f, g\in \RR^m$,   $I\subset \{1,2,\dots, m\}$, $r\in[1, \infty)$, and   $\tau\ge 0$, we have
\begin{align}\label{4-1b-15}
\Bl\||f|^p-|g|^p\Br\|_{p'} &\le \Bl\| |R_I f|^p-|R_Ig|^p \Br\|_{p'} +\tau^{\f 1p}\Bl( \|R_{I^c}f\|_p^{p-1} + \|R_{I^c}g\|_p^{p-1}\Br)+
\\
&+ 2p\bigl(S^r(\tau)\bigr)^{\f 1 {rp'}}\|R_{I^c} (f-g)\|_{p'r'},
\notag
\end{align}
where
\[
S^r(\tau)=\max\Biggl\{\  \sum_{j\in I^c(|f|^p;\tau)}\   |f(j)|^{pr},\  \sum_{j\in I^c(|g|^p;\tau)}\   |g(j)|^{pr}\Biggr\}.
\]
\end{lemma}

\bp First, we show that for any  $a, b\ge 0$ and $\eta\ge 0$,
\begin{equation}\label{4-7-a}
  |a^p-b^p|^{p'} \le  p^{p'} ( a_\eta^p+b_\eta^p)|b-a|^{p'} + \eta^{p'} (a^p+b^p),
\end{equation}
where $u_\eta =u\cdot  \ind_{[\eta,\infty)} (u)$ for $u\ge 0$.  Without loss of generality, we may assume that $a\ge b$.
If $a<\eta$, then
\[  |a^p-b^p|^{p'} \le (a^p)^{p'} =(a^p)^{p'-1} a^p\le
( \eta^p)^{p'-1} (a^p+b^p) =\eta^{p'} (a^p+b^p). \]
If $a\ge \eta$, then $a=a_\eta$ and
\begin{align*}
  |a^p-b^p|^{p'} &\le ( p a^{p-1} |b-a|)^{p'} = p^{p'} a^{p} |b-a|^{p'}
   \le  p^{p'} ( a_\eta^p +b_\eta^p) |b-a|^{p'}.
\end{align*}
Thus, in either case,   we prove   \eqref{4-7-a}.

 Next, we prove \eqref{4-1b-15}.  Let $f, g\in \RR^m$ and $I\subset \{1,2,\dots, m\}$. Then
\begin{align}
\Bl\||f|^p-|g|^p\Br\|_{p'}&\le
\Bl\| |R_I f|^p-|R_Ig|^p \Br\|_{p'}+\Bl\| |R_{I^c} f|^p-|R_{I^c}g|^p \Br\|_{p'}.\label{4-3-15}
\end{align}
Using \eqref{4-7-a} with $\eta=\tau^{1/p}$, we obtain
\begin{align*}
\Bl\| &|R_{I^c} f|^p-|R_{I^c}g|^p \Br\|_{p'}^{p'}=\sum_{j\in I^c} \Bl||f(j)|^p-|g(j)|^p\Br|^{p'}\notag\\
 &\le p^{p'} \sum_{j\in I^c} |f(j)-g(j)|^{p'} \Bl( |f(j)|^p\cdot \ind_{\{|f(j)|\ge \eta\}}(j)+|g(j)|^p \cdot\ind_{\{|g(j)|\ge \eta\}}(j)\Br)\notag\\
&\   \    \    \   +\eta^{p'} \sum_{j\in I^c} \Bl( |f(j)|^p +|g(j)|^p\Br)\notag\\
&\le  \tau^{\f 1 {p-1}}(\|R_{I^c}f\|_p^p+\|R_{I^c} g\|_p^p) + 2p^{p'} \bigl(S^r(\tau)\bigr)^{1/r}\| R_{I^c} (f- g)\|_{p'r'}^{p'}.
\end{align*}
This combined with \eqref{4-3-15} yields  \eqref{4-1b-15}.
\ep

We now apply Lemma~\ref{lem-6-4-0-Ta} to establish a key local entropy estimate, which provides control of the form~\eqref{1-2-6}.

\begin{lemma}\label{lem-3-3}
Let $\ell\ge 1$ be a fixed, inessential integer, and let  $(\cA_k)_{k \ge 0}$ be an arbitrary admissible sequence of partitions of $G$. Then,
for any $q\in(1, p]$, any integer $n\ge 0$,  any   parameters $\beta > 0$ and $\tau> 0$ (possibly depending on $n$), and every subset $A \subset \bT:=T_p(G)$,  the following estimate holds with $p_1:=\f {pq}{p-q}$:
\begin{align}\label{4-5-15}
&e_{n+\ell+1}(A, \|\cdot\|_{p'})
\le 16 \cdot 2^{-2^{\ell}} \cdot \diam(A, \|\cdot\|_{p'}) + 8 \tau^{1/p} (u^\ast)^{1/p'}+ \\
&\  \ + 8p \beta\cdot \min_{|I| \le 2^n} S_{I^c}(A; \tau)
+ 8 p\beta^{-(q-1)} \cdot
\sup_{f\in G}\|f\|_\infty^{p-q}
\cdot
\sup_{x \in A}\ \left[ \diam\, (\cA_{n+\ell}(g_x), \|\cdot\|_{p_1})\right]^{q},\notag
\end{align}
where $I\subset\{1,2,\dots, m\}$ and
\[
S_{I^c}(A; \tau) := \sup_{x \in A}\  \sum_{j\in I^c(x;\tau)}\  x(j).
\]
Here, for each $ x \in A $, we denote by $g_x\in G$ a vector such that $x = |g_x|^p$, and by
$\cA_{n+\ell}(g_x)$  the unique cell of the partition $\cA_{n+\ell}$ that contains $g_x$.
\end{lemma}
\bp  Let $G_A=T_p^{-1}(A):=\{g\in G\colon |g|^p\in A\} $,  so that  $A=\{|g|^p\colon g\in G_A\}$. For convenience, define $D:=\diam(A, \|\cdot\|_{p'})$. Without loss of generality, we may assume that
\beq \label{4-6-15}
e_{n+\ell+1} (A, \|\cdot\|_{p'})> 16\cdot 2^{-2^{\ell}} D,
\enq
as otherwise \eqref{4-5-15} holds trivially. Fix an index set $I\subset \{1,2,\dots, m\}$ such that $|I|\le 2^n$.
Note that  for each  $x=|g_x|^p\in A$, we may decompose
\[ x=R_I x+R_{I^c } x=|R_I g_x+R_{I^c} g_x|^p=|R_I g_x|^p +|R_{I^c} g_x|^p.\]
Let $\Ld_1, \Ld_2\subset A$ be arbitrary subsets with  $|\Ld_1|, |\Ld_2|\le N_{n+\ell}$, and define
\[
\Ld:=\Bl\{ R_I y+R_{I^c} z\colon  y\in \Ld_1, z\in \Ld_2\Br\}\subset \RR^m.
\]
Then $|\Ld|\le |\Ld_1||\Ld_2|\le N_{n+\ell+1}$, and hence,
\beq \label{4-7-15}
e_{n+\ell+1} (A, \|\cdot\|_{p'})\le 2 \sup_{x\in A} \min_{u\in\Ld} \|x-u\|_{p'}.
\enq
The factor $2$ appears because of~\eqref{3-1-0}, since in general $\Lambda\not\subset A$.

Using Lemma \ref{lem-6-4-0-Ta} with $r = \frac{q'}{p'}\ge1$, for any  $x\in A$ and $u=R_I y+R_{I^c} z\in\Ld$ with $y\in\Ld_1$ and $z\in\Ld_2$, we have
\eq{\|x-u\|_{p'}&=\bigl\| |g_x|^p - |R_I g_y +R_{I^c} g_z|^p\bigr\|_{p'} \\
\le& \| R_I x  -R_I y\|_{p'}  +2\tau^{\f1p} (u^\ast)^{\f1{p'}}+ 2p \bigl(S^r_{x,z}(\tau)\bigr)^{\f 1 {q'}} \|g_x-g_z\|_{\frac{pq}{p-q}},
}
where we have
\eq{
S^r_{x,z}(\tau)&:=\max\Biggl\{ \sum_{j\in I^c(|g_x|^p;\tau)} |g_x(j)|^{pr},\ \ \sum_{j\in I^c(|g_z|^p;\tau)} |g_z(j)|^{pr}\Biggr\},\\
&=\max\Biggl\{ \sum_{j\in I^c(x;\tau)} |x(j)|^r,\ \ \sum_{j\in I^c(z;\tau)}
|z(j)|^r\Biggr\}\le \sup_{f\in G}\|f\|_\infty^{\frac{q'(p-q)}{q}}\cdot S_{I^c}
(A; \tau),
}
with the inequality following from $r-1=\frac{q'}{qp}(p-q)$. Substituting  into \eqref{4-7-15} yields
\eq{&e_{n+\ell+1} (A, \|\cdot\|_{p'})\\
\le&2 \sup_{x\in A} \min_{y\in\Ld_1} \bigl\| R_I x -R_I y\bigr\|_{p'}  +4\tau^{\f1p} (u^\ast)^{\f1{p'}}+ 4p \bigl(S_{I^c} (A; \tau)\bigr)^{1/q'}
\sup_{f\in G}\|f\|_\infty^{\frac{p-q}{q}}
\sup_{x\in A}\min_{z\in\Ld_2}\|g_x-g_z\|_{\frac{pq}{p-q}}\\
\le & 2\sup_{x\in A} \min_{y\in\Ld_1} \bigl\| R_I x -R_I y\bigr\|_{p'}  +4\tau^{\f1p} (u^\ast)^{\f1{p'}}+4p\be  S_{I^c} (A; \tau)+\\
&+ 4p\be^{-(q-1)}
\sup_{f\in G}\|f\|_\infty^{p-q}
\sup_{x\in A}\min_{z\in\Ld_2} \|g_x-g_z\|_{\frac{pq}{p-q}}^{q},
}
where we used Young's inequality in the last step.
Taking infimum over all such  subsets $\Ld_1, \Ld_2\subset A$, we deduce
\begin{align}\label{4-8-15}
e_{n+\ell+1} (A, \|\cdot\|_{p'})&\le 2 e_{n+\ell}(R_I A, \|\cdot\|_{p'})
+4\tau^{\f1p} (u^\ast)^{\f1{p'}}+4p\be  S_{I^c} (A; \tau)+\\
&+4 p\be^{-(q-1)}\
\sup_{f\in G}\|f\|_\infty^{p-q}
\inf_{\substack{S\subset G_A \\|S|\le N_{n+\ell}}}\  \sup_{x\in A} \ \min_{g\in S} \|g_x-g\|_{\frac{pq}{p-q}}^{q}.
\notag
\end{align}

Since
\[R_I A\subset \{ x\in \RR^I\colon\|x\|_{p'}\le D\},\]
and $|I|\le 2^n$, we apply Lemma \ref{lem-1-2a} to obtain
\begin{align*}
e_{n+\ell}(R_I A,\|\cdot\|_{p'}) &\le D\cdot e_{n+\ell}(B_{p'}^{2^n}, \|\cdot\|_{p'})
 \le 4\cdot  2^{-2^{\ell}}\cdot D< \f14 e_{n+\ell+1} (A, \|\cdot\|_{p'}),
\end{align*}
where $B_{p'}^k:=\{x\in \RR^k\colon \|x\|_{p'}\le 1\}$, and we used \eqref{4-6-15} in the last step.
Substituting  into \eqref{4-8-15}, we then deduce
\begin{align}\label{4-9-15}
e_{n+\ell+1} (A, \|\cdot\|_{p'})&\le
8\tau^{\f1p} (u^\ast)^{\f1{p'}}+8 p\be  S_{I^c} (A; \tau)\\
&+8 p\be^{-(q-1)}\
\sup_{f\in G}\|f\|_\infty^{p-q}
\inf_{\substack{S\subset G_A \\|S|\le N_{n+\ell}}}\  \sup_{x\in A} \ \min_{g\in S} \|g_x-g\|_{p_1}^{q},
\notag
\end{align}
where $p_1:=\frac{pq}{p-q}$.

To estimate the last term,  we use the given admissible sequence of partitions of $G$.
Choose  a set $S_n\subset G_A$  such   that  for each  cell $F$ from the partition $\cA_{n+\ell}$ with $F\cap G_A\neq \emptyset$,  $F\cap S_n$ contains exactly one point from $F\cap G_A$.  Then   $|S_n|\le N_{n+\ell}$, and moreover,
\begin{align}
\label{4-10-15}
\inf_{\substack{S\subset G_A\\|S|\le N_{n+\ell}}}\  \sup_{x\in A} \ \min_{g\in S} \|g_x-g\|_{p_1}^{q}
&\le   \sup_{x\in A} \ \min_{g\in S_n} \|g_x-g\|_{p_1}^{q}\le \sup_{x\in A} \min_{g\in S_n\cap \mathcal{A}_{n+\ell} (g_x) } \|g_x-g\|_{p_1}^{q}\\
&\le  \sup_{x\in A}\left[ \diam(\cA_{n+\ell}(g_x), \|\cdot\|_{p_1})\right]^{q}.\notag
\end{align}
Substituting \eqref{4-10-15} into \eqref{4-9-15} completes the proof of  \eqref{4-5-15}.
\ep

To prove Theorem \ref{thmGammaBound}, we will apply Theorem \ref{thm-2-4} with the parameter $\al=p$. The crucial part is  to establish a local entropy number estimate of the form \eqref{2-4}. Lemma~\ref{lem-3-3} provides the entropy estimate \eqref{4-5-15}, which can be rewritten as
\begin{align}
e_{n+\ell+1}(A, \|\cdot\|_{p'})
\le 16 \cdot 2^{-2^{\ell}}  \diam(A, \|\cdot\|_{p'}) + \sup_{x\in A} \wt s_{n}(x) + 8p \beta \cdot \min_{|I| \le 2^n} S_{I^c}(A; \tau),\label{4-12-25}\end{align}
where $\wt s_n(x)=\wt s_{n,1}(x)+\wt s_{n,2}(x)$, $\wt s_{n,1}(x):=  8\tau^{1/p} (u^\ast)^{1/p'}$ and
\begin{align}\label{4-13b}
\wt s_{n,2}(x):=8p
\beta^{-(q-1)} \cdot
\sup_{f\in G}\|f\|_\infty^{p-q}
\cdot
\left[ \diam\, (\cA_{n+\ell}(g_x), \|\cdot\|_{\frac{pq}{p-q}})\right]^{q}.
\end{align}
For the term $\wt s_{n,1}$,  choosing  $\tau:=(\al_0 2^{n+n_0-1}) ^{-1} u^\ast$ for some parameters $\al_0\in (0, 1)$ and $n_0\in\NN$, we obtain
\beq \sup_{x\in A}\sum_{n=n_1}^{\lceil\log_2 m\rceil} 2^{\f np} \wt{s}_{n,1}(x)=\f{ 8u^\ast }{\al_0^{1/p} }\sum_{n= n_1}^{\lceil\log_2 m\rceil} 2^{-\f {n_0-1}p}\le \f {C\cdot (\log_2 m - n_1) \cdot u^\ast}{
 ( 2^{n_0}\al_0)^{1/p} },\label{4-13-25}
\enq
for any $n_1\in [1,\lceil \log_2 m\rceil]\cap \NN$. For the term $\wt s_{n,2}$, we may  choose  $\be=b 2^{-\f np}$ for some parameter $b>0$, and optimize the admissible sequence so that
\begin{align}\label{4-14-25}
\sup_{x\in A}\sum_{n=0}^\infty 2^{\f np} \wt s_{n,2}(x) &= 8p
b^{-(q-1)} \cdot
\sup_{f\in G}\|f\|_\infty^{p-q}
\sup_{x\in A}\sum_{n=0}^\infty \Bl[ 2^{\f np}\diam\, (\cA_{n+\ell}(g_x), \|\cdot\|_{\frac{pq}{p-q}})\Br]^{q}\\
&\le 16 pb^{-(q-1)}
\sup_{f\in G}\|f\|_\infty^{p-q}
\Bl[\ga_{p, q} (G, \|\cdot\|_{\frac{pq}{p-q}})\Br]^{q}.\notag
\end{align}
The main difficulty comes from the  term involving $\min\limits_{|I| \le 2^n} S_{I^c}(A; \tau)$.
To establish a local entropy number bound of the form \eqref{2-4}, it is necessary to construct   an increasing sequence of bounded nonnegative functions  $u_n\colon {\bf T}\times [0,\infty)\to [0,\infty)$, $n\in \NN_0$  satisfying the estimate \eqref{4-16-25} below. The existence of such a sequence is ensured by the following lemma.

\begin{lemma}\label{lem-3-4}
Let $p\in[2, \infty)$.
Let $\al_0\in (0, 1)$ and $n_0\in\NN\cap [2,\infty)$ be given parameters. Define, for each $n\in\NN_0$,
\[\al_n:=\al_0 2^n\  \ \text{ and}\  \  \tau:=\tau_n=(\al_{n+n_0-1}) ^{-1} u^\ast, \]
where $u^\ast>0$ is the constant given in \eqref{4-5-25}.
Then there exists an increasing sequence of bounded, nonnegative functions
\[u_n\colon {\bf T}\times [0,\infty)\to [0,\infty), \  \  n=0,1,2,\dots,\lceil \log_2 m \rceil,\]
such that
\[ 0\le u_n(x, r)\le u_{n+1}(x, r)\le u^\ast,\   \ \forall\, x\in {\bf T},\  \ \forall\, r\ge 0,\   \ \forall\, n\in\NN_0,\]
and the following growth condition holds:

\noindent
for every subset $A\subset {\bf T}$, every constant $a>0$  and every integer $1\le n\le \lceil \log_2 m \rceil$, one has
\begin{align}\label{4-16-25}
\sup_{x\in A}& \Bl[ u_{n+n_0} (x, r_{n+n_0}^a(x)) -u_n(x, r_A)\Br] +(\al_{n+n_0-1})^{\f 1p} \cdot \sup_{x\in A} r_{n+n_0}^a(x)\\
& \ge \min_{|I|\le \al_n} \sup_{x\in A} \ \sum_{j\in I^c(x,\tau)} |x(j)|=\min_{|I|\le \al_n} S_{I^c}(A,\tau){\ge \min_{|I| \le 2^n} S_{I^c}(A; \tau)},
\nonumber
\end{align}
where $r_n^a(x)\ge 0$ is as defined in Definition \ref{def-3-2}, and
\[r_A= 2^{n_0/p}\sup_{x\in A} r_{n+n_0}^a (x) +\diam(A, \|\cdot\|_{p'}).\]
\end{lemma}

The proof of Lemma~\ref{lem-3-4} will be deferred to Subsection~\ref{sec-4-3}.
For now, we accept it as given and proceed to the proof of Theorem~\ref{thmGammaBound} in the next subsection.

\subsection{Proof of Theorem \ref{thmGammaBound} (under the assumption of  Lemma~\ref{lem-3-4}) }\label{sec-4-2}
Let $q:=p_0\in(1, p]$, $b>0$, $\al_0\in (0, 1)$ and $n_0, m_0\in\NN$ be the parameters given in  Theorem \ref{thmGammaBound}. Let $\ell\in\NN$ be  a positive integer depending only on $p$, to  be specified later.

First, as we discussed above, we may apply Lemma~\ref{lem-3-3} with the choices of parameters \[\alpha = p,\  \beta = 2^{-n/p} b, \  \ \text{and}\  \tau = (\alpha_0 2^{n+n_0-1})^{-1} u^\ast.\]
This  yields the entropy estimate \eqref{4-12-25} for every subset $A \subset {\bf T}$.

 Next, we invoke Lemma \ref{lem-3-4} to estimate the term $\min\limits_{|I|\le 2^n} S_{I^c} (A, \tau)$ in \eqref{4-12-25}. Specifically, substituting the bound from \eqref{4-16-25} into
\eqref{4-12-25}, we obtain,  for any $a>0$,
 \begin{align*}&e_{n+\ell+1} (A, \|\cdot\|_{p'})\le  16\cdot 2^{-2^{\ell}} \cdot \diam(A, \|\cdot\|_{p'}) + \sup_{x\in A} \wt s_{n}(x)+\\
 &+8p b\cdot (2^{n_0} \al_0)^{1/p} \cdot \sup_{x\in A}r_{n+n_0}^a(x)+\f{8pb}{ 2^{n/p} }\cdot \sup_{x\in A} \Bl[ u_{n+n_0} (x, r_{n+n_0}^a(x)) -u_n(x, r_A)\Br],
 \end{align*}
 where
 $\wt s_n=\wt s_{n,1}+\wt s_{n,2}$, as  defined in \eqref{4-13b}, and \[r_A= 2^{n_0/p}\sup_{x\in A} r_{n+n_0}^a (x) +\diam(A, \|\cdot\|_{p'}). \]
 We now specify the constants $\ell$ and $a$. Let $C(p)>1$ denote the constant $C(\al)$  from Theorem~\ref{thm-2-4}, evaluated at  $\al=p$.  Define  $a:=\f 1{  32 p C(p) b}$, and  choose  $\ell=\ell(p)$ to be the smallest positive integer such that
 \[16\cdot 2^{-2^{\ell}}\le \f 1 {4C(p)}.\]
Invoking   Theorem \ref{thm-2-4} with the parameters
\[
\al=p,\  \ b_1=8p b\cdot (\al_0)^{1/p}, \  \  b_2=8p b,\ \ \text{and}\  \  n_1=\lceil \log_2 m_0\rceil,
\]
we obtain
\[
\ga_{p,1}(\mathbf{T}, \|\cdot\|_{p'})
\le C_1(p) \Bigg[  \big(bn_0+b^2 (\al_0)^{1/p}\big)u^\ast+2^{n_1/p}\diam({\bf T},\|\cdot\|_{p'})+\sup_{x\in {\bf T}}\sum_{n=n_1}^{\lceil \log_2 m \rceil} 2^{n/p}  \wt s_n(x)\Bigg].
\]

To estimate the last term, we take the infimum over all admissible sequences $(\cA_n)_{n\ge 0}$ of partitions of $G$.  Using  the estimates \eqref{4-13-25} and \eqref{4-14-25},
and taking $n_1= \lceil \log_2 m_0 \rceil$
we then obtain
\eq{
\sup_{x\in T}\sum_{n=n_1}^{\lceil \log_2 m \rceil} 2^{n/p}  \wt s_n(x)\le C_2(p)\Bigg[ \f { (\log_2 \frac{m}{m_0}) \cdot u^\ast}{
(2^{n_0}\al_0)^{1/p} } +
\sup\limits_{f\in G}\|f\|_\infty^{p-q}\cdot \frac{
\Bl[\ga_{p, q} (G, \|\cdot\|_{\frac{pq}{p-q}})\Br]^q}{b^{q-1}}\Bigg].
}

We now estimate the diameter  $\diam({\bf T}, \|\cdot\|_{p'})$ of ${\bf T}=T_p(G)$ in the norm $\|\cdot\|_{p'}$. For any $f, g\in G$, we have
\eq{\Bl\||f|^p-|g|^{p}\Br\|_{p'}&\le p \Bl\| (|f|^{p-1}+|g|^{p-1})|f-g|\Br\|_{p'}\\
&\le 2 p \|f-g\|_{\frac{pq}{p-q}}\cdot \sup_{h\in G}\|h\|_\infty^{\frac{p-q}{q}}\cdot (u^\ast)^{1/q'}.
}
Applying Young's inequality, we then obtain, for all $f, g\in G$,
\[ \Bl\||f|^p-|g|^{p}\Br\|_{p'}\le 2p bm_0^{-1/p}u^\ast +2p (bm_0^{-1/p})^{-(q-1)}\cdot \sup_{h\in G}\|h\|_\infty^{p-q}\cdot \|f-g\|_{\frac{pq}{p-q}}^q.\]
Taking the supremum over all $f, g\in G$, we conclude
\[
m_0^{1/p}\diam({\bf T},\|\cdot\|_{p'})
\le 2p b u^\ast + 2pb^{-(q-1)}m_0^{q/p}\, \sup_{f\in G}\|f\|_\infty^{p-q} \cdot \bigl[\diam(G, \|\cdot\|_{\frac{pq}{p-q}}) \bigr]^q.
\]
Finally, combining all of the above estimates,  we obtain the desired bound:
\begin{align*}
\ga_{p,1}(\mathbf{T}, \|\cdot\|_{p'})\le& C_3(p)  \Biggl[ \left(
bn_0+b^2 \cdot (\al_0)^{1/p} +\f { \log_2 \frac{m}{m_0}}{
(2^{n_0}\al_0)^{1/p}} \right)u^\ast
\\
&+
\sup_{f\in G}\|f\|_\infty^{p-q}\cdot
\f{m_0^{q/p}\bigl[\diam(G, \|\cdot\|_{\frac{pq}{p-q}}) \bigr]^q+ \bigl[\ga_{p,q}(G, \|\cdot\|_{\frac{pq}{p-q}})\bigr]^q}{b^{q-1}}\Biggr].
\end{align*}
Theorem~\ref{thmGammaBound} is now proved up to Lemma~\ref{lem-3-4}. \qed

\subsection{Proof of Lemma \ref{lem-3-4}}\label{sec-4-3}
 We consider the metric space ${\bf T}=(\mathbf{T}, \|\cdot\|_{p'})$, where
\[
{\bf T}=T_p(G)=\{|f|^p\colon  f\in G\}.\]
Recall that  \[  u^\ast=\sup_{g\in G}\|g\|_p^p=\sup_{x\in {\bf T}}\|x\|_1<\infty.\]
For $x\in {\bf T}$ and $r\ge 0$, let $B_{\bf T}(x,r):=\{y\in {\bf T}\colon  \|x-y\|_{p'}\le r\}$.
We need to estimate the quantity
\[
{\min_{|I|\le \al_n} S_{I^c} (A, \tau),}
\]
where
\[
 S_{I^c} (A, \tau) = \sup_{x\in A} \sum_{j\in I^c(x;\tau)} \  x(j).
\]
Throughout the proof, the letters   $I, J$ always denote subsets   of $\{1,2,\dots, m\}$.

We temporarily fix an index set $I\subset \{1,2,\cdots, m\}$ with $|I|\le \al_n$.
For any  $x\in \mathbf{T}$, we have
\[
|I^c(x;\tau)|\le \f {\|x\|_1}{\tau}\le \al_{n+n_0-1}.
\]
Consequently,
\eq{ S_{I^c} (A, \tau) &= \sup_{x\in A} \|R_{I^c(x;\tau)} x\|_1 \le  \sup_{x\in A} \max_{\sub{J:\ I\cap J=\emptyset\\ |J|\le \al_{n+n_0-1}}} \|R_J x\|_1.
}

Let $s\colon {\bf T}\to [0,\infty)$ be any   nonnegative function.
If $x\in {\bf T}$,  $y\in B_{\bf T}(x, s(x))$  and $|J|\le \al_{n+n_0-1}$, then
\eq{ \Bl| \|R_J x\|_1-\|R_J y\|_1\Br|\le \bigl\|R_J (x-y)\bigr\|_1 \le |J|^{\f1p} \|x-y\|_{p'}\le s(x) \cdot (\al_{n+n_0-1})^{\f1p},
}
 which implies
\[
\|R_J x\|_1 \le  (\al_{n+n_0-1})^{\f1p}\cdot s(x)+\inf_{y\in B_{\bf T}(x, s(x))} \|R_J y\|_1.
\]
Hence,
\beq  S_{I^c} (A, \tau) \le S_1(I; A, \tau)+ (\al_{n+n_0-1})^{1/p} \sup_{x\in A} s(x),\label{4-17-25}
\enq
where
\[ S_1(I; A, \tau):= \sup_{x\in A} \max_{\sub{J:\ I\cap J=\emptyset\\
|J|\le \al_{n+n_0-1}}} \inf_{y\in B_{\bf T}(x, s(x))}\|R_J y\|_1.\]
Observe  that
\eq{
&\sup_{x\in A} \max_{\sub{J:\  I\cap J=\emptyset\\
|J|\le \al_{n+n_0-1}}}\Bigg[  \inf_{y\in B_{\bf T}(x, s(x))}\|R_J y\|_1+ \inf_{y\in B_{\bf T}(x, s(x))}\|R_I y\|_1\Bigg]\\
&\le \sup_{x\in A} \max_{\sub{J: \ I\cap J=\emptyset\\
|J|\le \al_{n+n_0-1}}}  \inf_{y\in B_{\bf T}(x, s(x))}\|R_{I\cup J} y\|_1\le
\sup_{x\in A} \max_{|J|\le \al_{n+n_0}} \inf_{y\in B_{\bf T}(x, s(x))}\|R_{J} y\|_1,
}
where the last step uses the inequality $\al_{n+n_0-1}+\al_n\le \al_{n+n_0}$.
For any $x, z\in A$, we have
\[ B_{\bf T} (x, s(x))\subset B_{\bf T} (z, s_A),\  \ \text{where}\  \ s_A:= \sup_{z\in A} s(z) + \diam(A, \|\cdot\|_{p'}),\]
implying
\eq{ \inf_{y\in B_{\bf T}(x, s(x))}\|R_I y\|_1\ge  \sup_{z\in A} \inf_{y\in B_{\bf T}(z, s_A)}\|R_I y\|_1,\  \ \forall\, x\in A.}
It  follows that
\eq{ &S_1(I; A, \tau)+
{ \sup_{z\in A} \inf_{y\in B_{\bf T}(z, s_A)} }\|R_I y\|_1\le
\sup_{x\in A} \max_{|J|\le \al_{n+n_0}} \inf_{y\in B_{\bf T}(x, s(x))}\|R_{J} y\|_1.
}
This combined  with \eqref{4-17-25} yields
\eq{& S_{I^c} (A, \tau)+ \sup_{x\in A} \inf_{y\in B_{\bf T}(x, s_A)}\|R_I y\|_1\\
&\le \sup_{x\in A} \max_{|J|\le \al_{n+n_0}} \inf_{y\in B_{\bf T}(x, s(x))}\|R_{J} y\|_1+ (\al_{n+n_0-1})^{1/p} \sup_{x\in A} s(x).
}
Taking the maximum over all $I\subset \{1,2,\cdots, m\}$ with $|I|\le \al_n$, we obtain
\begin{align*}
&\min_{|I|\le \al_n} S_{I^c} (A, \tau)+ \sup_{x\in A} \max_{|I|\le \al_n}   \inf_{y\in B_{\bf T}(x, s_A)}\|R_I y\|_1\notag\\
&\le \max_{|I|\le \al_n}\Bigg[S_{I^c} (A, \tau)+ \sup_{x\in A} \inf_{y\in B_{\bf T}(x, s_A)}\|R_I y\|_1\Bigg]\notag\\
&\le \sup_{x\in A} \max_{|I|\le \al_{n+n_0}} \inf_{y\in B_{\bf T}(x, s(x))}\|R_{I} y\|_1+ (\al_{n+n_0-1})^{1/p} \sup_{x\in A} s(x).
\end{align*}
Thus, it follows that
\begin{align}\label{9-18}
&\min_{|I|\le \al_n} S_{I^c} (A, \tau)- (\al_{n+n_0-1})^{1/p} \sup_{x\in A} s(x)
\\ &\le \Bigg[\sup_{x\in A} \max_{|I|\le \al_{n+n_0}} \inf_{y\in B_{\bf T}(x, s(x))}\|R_{I} y\|_1- \sup_{x\in A} \max_{|I|\le \al_n}   \inf_{y\in B_{\bf T}(x, s_A)}\|R_I y\|_1\Bigg]\notag\\
&\le \sup_{x\in A}\Bigg[  \max_{|I|\le \al_{n+n_0}} \inf_{y\in B_{\bf T}(x, s(x))}\|R_{I} y\|_1-  \max_{|I|\le \al_n}   \inf_{y\in B_{\bf T}(x, s_A)}\|R_I y\|_1\Bigg].
\notag
\end{align}

Now we  define
a sequence of functionals $u_n\colon {\bf T}\times [0,\infty)\to [0,\infty)$, $n\in \NN_0$ as follows:
\[
u_n(x,r)=\max_{|I|\le \al_n} \inf_{y\in B_{\bf T}(x,r)} \|R_I y\|_1,\quad x\in \mathbf{T},\ \ r\ge 0.
\]
By  definition, the sequence $(u_n(x, r))_{n\ge 0}$ is decreasing in $r$ and satisfies
\[
0\le u_n(x,r)\le u_{n+1}(x,r)\le u^\ast \   \ \text{for all $ x\in \mathbf{T}$ and $r\ge 0$.}
\]
Furthermore, we can rewrite \eqref{9-18} equivalently in the form:
\begin{equation*}
\min_{|I|\le \al_n} S_{I^c} (A, \tau)\le   \sup_{x\in A} \Bl[ u_{n+n_0} (x, s(x))-u_n(x, s_A)\Br]+(\al_{n+n_0-1})^{1/p} \sup_{x\in A} s(x).
\end{equation*}
To complete the proof, we set
$s(x):=r_{n+n_0}^a(x)$, and observe that
\[ s_A:= \sup_{z\in A} s(z) + \diam(A, \|\cdot\|_{p'})\le r_A:=
2^{n_0/p}\sup_{x\in A} r_{n+n_0}^a (x) +\diam(A, \|\cdot\|_{p'}).\]
Since the function $u_n(x, \cdot)$ is decreasing, we prove  the growth condition \eqref{4-16-25}  for all $a>0$.
\qed

\end{document}